\documentclass[11pt,reqno]{amsart}
\usepackage{amsmath,amsfonts,amsthm,amssymb}
\usepackage{tikz}
\usetikzlibrary{cd,arrows,decorations.pathmorphing,decorations.markings,backgrounds,positioning,fit,calc,shapes,patterns,nfold}
\usepackage[utf8]{inputenc}
\usepackage[T1]{fontenc}
\usepackage{palatino}
\usepackage{setspace}
\usepackage{fullpage}
\usepackage{mathtools}
\usepackage[shortlabels]{enumitem}
\usepackage[subrefformat=parens, labelfont=up]{subcaption}
\usepackage[colorlinks=true,linkcolor=black,anchorcolor=black,citecolor=black,filecolor=black,menucolor=black,runcolor=black,urlcolor=black,linktoc=all]{hyperref}

\newtheorem{thm}{Theorem}[section]

\newtheorem{innercustomthm}{Theorem}
\newenvironment{mythm}[1]
{\renewcommand\theinnercustomthm{#1}\innercustomthm}
{\endinnercustomthm}
\newtheorem{lemma}[thm]{Lemma}
\newtheorem{porism}[thm]{Porism}

\newtheorem{coro}[thm]{Corollary}
\newtheorem{prop}[thm]{Proposition}
\theoremstyle{definition}
\newtheorem{defn}[thm]{Definition}
\newtheorem{rem}[thm]{Remark}
\newtheorem{ex}[thm]{Example}
\newtheorem{warn}[thm]{Warning}
\newtheorem{nota}[thm]{Notation}

\colorlet{dark purple}{red!35!blue}
\colorlet{dark red}{red!80!black}
\colorlet{dark blue}{blue!80!black!80!cyan}

\tikzstyle{mutable}=[fill=white,inner sep=0.9mm,circle,draw,minimum size=2mm]
\tikzstyle{mutable3}=[fill=white,inner sep=0.9mm,circle,draw,minimum size=7mm]
\tikzstyle{frozen}=[fill=white,inner sep=.9mm,rectangle,draw]
\tikzstyle{dot} = [fill=black!25,inner sep=0.5mm,circle,draw,minimum size=.5em]
\tikzstyle{outline}=[thick,line width=1.5mm,draw=black!10]

\DeclareMathOperator\Sk{Sk}
\DeclareMathOperator\Hom{Hom}
\DeclareMathOperator\Spec{Spec}
\DeclareMathOperator\Char{char}
\DeclareMathOperator*{\btie}{\bowtie}
\def\TSk{\Sk^{\btie}}
\def\barTSk{\overline{\Sk}^{\btie}}

\def\MM{\mathcal{M}}
\def\bbN{\mathbb{N}}
\def\QQ{\mathbb{Q}}
\def\ZZ{\mathbb{Z}}
\def\SS{\Sigma}
\def\CA{\mathcal{A}}
\def\CU{\mathcal{U}}
\def\cV{\mathcal{V}}

\def\Bt{\widetilde{B}}
\def\Multi{\mathrm{Multi}}
\def\TMulti{\Multi^{\btie}}
\def\Supp{\mathrm{Supp}}

\everymath{\displaystyle}

\allowdisplaybreaks[1]
\title{Deep Points of Cluster Algebras II: Punctured Marked Surfaces}  
\date{2026-10-06}

\author{James Beyer}
\address{Department of Mathematics, University of Oklahoma, Norman, OK, USA}
\address{Instituto de Matem{\'a}ticas, Universidad Nacional Aut{\'o}noma de M{\'e}xico, Ciudad Universitaria, CDMX, M{\'e}xico}
\email{jebeyer@im.unam.mx}

\keywords{Cluster algebra,  cluster variety, deep point}
\subjclass[2020]{
	Primary 13F60, %Cluster algebras
	Secondary 14B05, %Singularities in algebraic geometry,
	14C99 %Cycles and subschemes
}

\begin{document}
	
	\begin{abstract}
		We classify the deep points of cluster varieties of cluster algebras of punctured marked surfaces with boundary.
	\end{abstract}
	
	\maketitle 
	
	\tableofcontents

%============================

\section{Introduction}\label{section: intro}

Cluster algebras are commutative domains with distinguished sets of generators called \emph{cluster variables}, in which the full set of cluster variables can be recovered from \emph{mutation} of overlapping $n$-element subsets called \emph{clusters}. 
For each cluster algebra $\CA$ and field $\Bbbk$, there is an associated \emph{cluster variety (of $\Bbbk$-points of $\CA$)}, defined as \[ V(\CA,\Bbbk) \coloneqq \Hom_{\text{Ring}} (\CA,\Bbbk) = \{\text{ring homomorphisms } p \colon \CA \to \Bbbk\}.\] 
For each cluster, the variety contains a corresponding algebraic torus, called a \emph{cluster torus}. 

We refer to any points that are not contained in any cluster torus as \emph{deep points}. Explicitly, a deep point is a ring homomorphism $\CA \to \Bbbk$ that kills at least one cluster variable in every cluster. The set of all such points, called the \emph{deep locus}, forms a closed (possibly empty) subvariety of $V(\CA,\Bbbk)$.

In \cite{BM25}, the following maxim was proposed:
\[ \text{\emph{Any bad property of a cluster algebra can be blamed on its deep points.}} \]
For example, Proposition 2.5 of \cite{BM25} says that any singular point of a cluster variety is necessarily deep. 
On the other hand, a cluster algebra with no deep points has many desirable properties \cite[Theorem 2.8]{BM25}.

The presence of deep points does not necessarily imply that a cluster algebra is poorly behaved, however. In classifying the deep points of cluster algebras, we are simply identifying the points where bad behavior \emph{could} arise. As cluster algebras continue to find applications in diverse areas of mathematics and physics, locating potentially problematic points could prove to be useful.

As part of our previous work with Greg Muller \cite{BM25}, we classified the deep points of several types of cluster algebras: 
\begin{enumerate}
    \item type $A_n$ cluster algebras for all $n \in \mathbb{N}$,
    \item rank 2 cluster algebras, 
    \item the Markov cluster algebra, and
    \item cluster algebras of unpunctured triangulable surfaces with boundary.
\end{enumerate}  
Concurrent work of Castronovo, Gorsky, Simental, and Speyer \cite{CGSS26} also provided an extensive examination of the deep locus of many types of cluster algebras. Further work on properties of the deep locus has been carried out by Muller \cite{Mul12}, Matherne and Muller \cite{MM15}, Neville and Simental \cite{NevilleSimental}, Tyler \cite{tyler}, and others. 
In this article, we build on \cite{BM25}, classifying the deep points of cluster algebras of punctured surfaces with boundary.

%============================

\subsection{Outline}

In Sections \ref{section: background}, \ref{section: surface type}, and \ref{section: deepmarkedsurfaces}, we recall relevant definitions and results. We also adapt results from \cite{BM25} and \cite{Mul16} for punctured surfaces where necessary.  

In Section \ref{section: 1 punctured disks}, we classify the deep points of once-punctured disks with marked points on the boundary, alternatively referred to as once-punctured polygons. After examining once-punctured bigons, trigons, and quadrilaterals directly, we consider the general case, obtaining the following classification for once-punctured disks with $n \geq 4$ boundary marked points.
	
\begin{mythm}{\ref{thm: Dn with boundary}}
	Let $\SS$ be a once-punctured disk with $n\geq 4$ marked points on the boundary. 
	\begin{enumerate}
		\item If $n$ is odd, then the deep locus of $V(\CA(\SS),\Bbbk)$ is a single component of dimension $2n-3$.
		\item If $n$ is even, then the deep locus of $V(\CA(\SS),\Bbbk)$ contains one component of dimension $2n-3$ and four distinct subsets isomorphic to $(\Bbbk^\times)^n$. 
		The intersection of the closures of these five sets is an $(n-1)$-dimensional subvariety. 
	\end{enumerate}
\end{mythm}
    
We then extend the results of Section \ref{section: 1 punctured disks} to more general punctured marked surfaces with boundary, classifying the deep points of cluster algebras of punctured disks with at least one boundary marked point in Section \ref{section: deep punctured disks} and punctured surfaces with at least two boundary marked points in Section \ref{sec: punc surfaces}. We then obtain the following general result. 

\begin{mythm}{\ref{thm: deepsurface punctured all}}
	Let $\SS_q$ be a connected, triangulable marked surface with genus $g\geq 0$, $b\geq 1$ boundary components, $m\geq 2$ marked points on the boundary, and $q\geq 1$ punctures, such that $\SS_q$ is not a punctured bigon. 
	\begin{enumerate}
		\item If $m$ is odd, then the deep locus of $V(\CA(\SS_q),\Bbbk)$ is the union of $q$ non-disjoint sets, each isomorphic $V(\CA(\SS_{q-1}),\Bbbk)$. 
		\item If $m$ is even, then the deep locus of $V(\CA(\SS_q),\Bbbk)$ is the union of: 
        \begin{itemize}
            \item $q$ non-disjoint sets, each isomorphic to $V(\CA(\SS_{q-1}),\Bbbk)$, and
            \item $2^{2g+b+2q-1}$ disjoint sets, each isomorphic to $(\Bbbk^\times)^{2g+b+q+m-2}$.
        \end{itemize}
	\end{enumerate}
\end{mythm}

%============================

\subsection{Acknowledgments}

Many of the results of this article were contained in the author's PhD dissertation at the University of Oklahoma \cite{BeyerDissertation}. The author would like to thank Greg Muller for his guidance in the completion of that work. 

%============================

\subsubsection*{Funding}

This work was partially supported by a DFCAS dissertation completion fellowship from the University of Oklahoma and by Universidad Nacional Aut{\'o}noma de M{\'e}xico Postdoctoral Program. 

%============================

\subsubsection*{AI}

The author has not used artificial intelligence at any stage of the preparation of this article. However, examination of an LLM-generated referee report on a different article has led to the correction of a typo in this article.

%============================

%\subsection*{Change tracking}

%The primary differences between version 1 of this preprint and the current version (v2) are:
%\begin{enumerate}
%    \item In the statement of Theorems \ref{thm: deepsurface punctured benevolent} and \ref{thm: deepsurface punctured all}, the number of benevolent components has been corrected to $2^{2g+b+2q-1}$.
%    \item Sections \ref{sec: punc polygons} and \ref{sec: punc monogons} have been corrected and rewritten. Punctured bigons have been excluded from results in Section \ref{sec: punc surfaces} where appropriate.
%\end{enumerate}

%============================

\section{Background}\label{section: background}

Cluster algebras were introduced by Fomin and Zelevinsky in a series of papers, the first of which was published in 2002 \cite{FZ02}. For an introduction to the subject, we recommend Fomin, Williams, and Zelevinsky's book; drafts of the first few chapters are available on the arXiv \cite{FWZChapters123,FWZChapters45}.

Given an extended exchange matrix $\Bt$, attach a \textbf{cluster variable} $x_i$ to each row, and let $\mathcal{F} \coloneqq \QQ (x_1, \dotsc, x_m)$ be the field of rational functions on the initial cluster variables $x_1, \dotsc, x_m$.
Collectively, we refer to the cluster variables associated to an extended exchange matrix as a \textbf{cluster} $\mathbf{x} = \{ x_1, \dotsc, x_m \}$, and the tuple $(\Bt, \mathbf{x})$ is called a \textbf{seed}.
\emph{Mutation} of a seed at a mutable index $k$ produces a new seed $(\mu_k(\Bt), \mathbf{x}^\prime)$, where $\mathbf{x^\prime} = \mathbf{x}\smallsetminus\{x_k\} \cup \{x_k^\prime\}$, according to the exchange relation: 
\begin{equation}\label{eq: mutation relation}
	x_k^{}\,x_k^\prime = \prod_{1\leq j\leq m} x_j^{\max \left(b_{jk},~0\right)} ~+~ \prod_{1\leq j\leq m} x_j^{\max \left(-b_{jk},~0\right)}. 
\end{equation}
Let $\mathcal{X}$ be the set of all cluster variables produced through repeated mutation at all mutable indices. The \textbf{cluster algebra} $\CA = \CA(\Bt, \mathbf{x}) $ is the $\ZZ$-subalgebra of $\mathcal{F}$ generated by $\mathcal{X}$ and the inverses of any frozen variables. 

An important property of cluster algebras is the \emph{Laurent phenomenon} \cite[Theorem 3.1]{FZ02}: given any initial seed containing a cluster $\mathbf{x} = \{ x_1,\dotsc, x_m\}$, the cluster variables of all other clusters can be written as Laurent polynomials in this initial cluster; that is, $\CA \hookrightarrow \ZZ[x_1^{\pm 1}, \dotsc, x_m^{\pm 1}]$. As a result, the algebra depends only on the mutation class of $\Bt$.

%============================

Given a cluster algebra $\CA$ and a field $\Bbbk$, the \textbf{cluster variety} is the set
\[ V(\CA,\Bbbk) \coloneqq \Hom_{\mathrm{Ring}}(\CA,\Bbbk) \coloneqq \{\text{ring homomorphisms $p:\CA\rightarrow \Bbbk$}\} \]
This has the \emph{Zariski topology} in which a closed set consists of homomorphisms which send a fixed ideal to 0. 
Each element $a\in \CA$ determines a function $f_a:V(\CA,\Bbbk)\rightarrow \Bbbk$ by the rule that $f_a(p) \coloneqq p(a)$. 

\begin{rem}
	If $\CA$ is finitely generated, then $V(\CA,\Bbbk)$ may be identified with an affine $\Bbbk$-variety. Explicitly, a finite presentation $ \CA \simeq \mathbb{Z}[g_1,g_2,\dotsc ,g_m] / \langle p_1,p_2,\dotsc ,p_n\rangle$ induces a homeomorphism between $V(\CA,\Bbbk)$ and the solution set of the system $\{p_1=0,p_2=0,\dotsc ,p_n=0\}$ inside $\Bbbk^m$. 	
	In general, $V(\CA,\Bbbk)$ is the set of $\Bbbk$-points of the affine scheme $\Spec(\CA)$. While this scheme need not be finite-type, we still refer to $V(\CA,\Bbbk)$ as a `variety' for simplicity. 
\end{rem}

By the Laurent phenomenon, each cluster $\{x_1,x_2,\dotsc ,x_n\}\subset \mathcal{A}$
gives an inclusion of rings
\[ \CA \hookrightarrow \CA[x_1^{-1},x_2^{-1},\dotsc ,x_n^{-1}] =\mathbb{Z} [x_1^{\pm1} ,x_2^{\pm1},\dotsc ,x_n^{\pm1}] \]
A homomorphism $p:\CA\rightarrow \Bbbk$ factors through this localization map if and only if $p(x_1), p(x_2), \dotsc ,$ $p(x_n)$ are non-zero. Therefore, we have an open inclusion
\[
(\Bbbk^\times)^n \simeq \Hom_{\mathrm{Ring}}(\mathbb{Z} [x_1^{\pm1} ,x_2^{\pm1},\dotsc ,x_n^{\pm1}],\Bbbk)\hookrightarrow \Hom_{\mathrm{Ring}}(\CA,\Bbbk) \eqqcolon V(\CA,\Bbbk)
\]
whose image is the set of homomorphisms $p:\CA\rightarrow \Bbbk$ such that $p(x_1),p(x_2), \dotsc ,$ $p(x_n)$ are non-zero; or equivalently, the open subset of $V(\CA,\Bbbk)$ where none of the functions $f_{x_1},f_{x_2},\dotsc ,f_{x_n}$ vanish. 
This image is called the \textbf{cluster torus} associated to the cluster $\{x_1,x_2,\dotsc ,x_n\}$. The union of all cluster tori is commonly referred to as the \emph{cluster manifold}. 

\begin{defn}\label{def: deep point def}
	A point $p\in V(\CA,\Bbbk)$ is a \textbf{deep point} if, for all clusters \[ \{x_1, x_2, \dotsc ,x_n\}\subset \CA,\] there is a cluster variable $x_i\in \{x_1,\dotsc ,x_n\}$ such that $p(x_i)=0$.
    
    The set of all deep points in $V(\CA,\Bbbk)$ is called the \textbf{deep locus} of $V(\CA,\Bbbk)$. Equivalently, the deep locus can be defined as the complement of the cluster manifold in $V(\CA,\Bbbk)$.
\end{defn}

%============================

\section{Surface type cluster algebras and cutting}\label{section: surface type}

Cluster algebras of triangulable surfaces were introduced in \cite{GSV05} and further developed in \cite{FST08} and \cite{FT18}. We will review the necessary parts of these constructions.

A \textbf{marked surface} consists of a smooth, oriented, compact surface-with-boundary $\SS$, together with a finite set of \textbf{marked points} $\MM$.
Marked points in the interior of $\SS$ are called \textbf{punctures}, and so a marked surface is \textbf{unpunctured} if $\MM\subset \partial\SS$. 
A connected marked surface is determined (up to diffeomorphisms that preserve the marked points and orientation) by a few numbers:
\begin{itemize}
	\item The genus $g$ of the underlying surface.
	\item The number $b$ of boundary components.
	\item The number $q$ of punctures.
	\item The unordered numbers $m_1,m_2,\dotsc,m_b$ of marked points on each boundary component.
\end{itemize}

A \textbf{marked curve} in a marked surface $\SS$ is an immersion of a compact, connected curve-with-boundary\footnote{Note that a compact, connected curve-with-boundary must be diffeomorphic to either an interval or a circle.} into $\SS$ such that
\begin{itemize}
	\item any endpoints map to $\MM$,
	\item the interior maps to $\SS\smallsetminus \MM$,
	\item any \emph{crossings} (points in $\SS\smallsetminus \MM$ with more than one preimage) must be simple (exactly two preimages) and transverse (the two preimages have distinct derivatives), and
	\item the set of {crossings} is finite.
\end{itemize}  
A \textbf{marked multicurve} is an immersion of finitely many compact curves-with-boundary satisfying these properties. We consider two marked curves (resp.~multicurves) \textbf{homotopic} if they are related by orientation-reversal and homotopies through the class of marked curves  (resp.~multicurves).

Some additional terminology for marked curves:
\begin{itemize}
	\item A \textbf{marked arc} is a marked curve with two endpoints (i.e.~an immersed interval).
	\item A \textbf{loop} is a marked curve with no end points (i.e.~an immersed circle).
	\item A marked curve is \textbf{simple} if it has no crossings and is not contractible.
	\item A multicurve (i.e.~a finite set of marked curves) is \textbf{compatible} if there are no crossings and every curve is simple.
	\item A \textbf{boundary curve} is a marked curve homotopic to a marked curve contained in $\partial \SS$.
\end{itemize}

A \textbf{triangulation} of a marked surface $\SS$ is a simple marked multicurve which divides the surface into a union of topological triangles, together with the set of boundary arcs. 
If a marked surface admits a triangulation, we say it is \textbf{triangulable}.

\begin{prop}\cite[Proposition 2.10]{FST08}\label{prop: triangulationcount}
	Let $\SS$ be a connected marked surface with genus $g$, $b$-many marked points, $q$-many punctures, and $m$-many marked points on the boundary. Then $\SS$ is triangulable if and only if the set of marked points $M$ is not empty, every component of the boundary $\partial \SS$ contains at least one marked point, and the quantity
	\[ n(\SS) \coloneqq 6g + 3b + 3q + 2m - 6 \]
	is positive. Every triangulation of $\SS$ has $n(\SS)$-many marked arcs (counting boundary arcs).
\end{prop}

\begin{rem}
	The $m$ boundary marked arcs appear in every triangulation. Therefore, every triangulation contains $ 6g + 3b + 3q + m - 6$ non-boundary marked arcs; this matches Proposition 2.10 of \cite{FST08} and is the number of mutable cluster variables in the corresponding cluster algebra.
\end{rem}

%============================

\subsection{Cluster algebras of unpunctured surfaces}

Given a triangulation of an unpunctured surface, we can construct a quiver by the following steps:
\begin{enumerate}
	\item Assign a mutable vertex to each interior arc.
	\item Assign a frozen vertex to each boundary arc.
	\item Add arrows between the sides of each triangle in the clockwise direction, omitting any arrows between frozen vertices.
    \item If there is a pair of arrows with the same endpoints but opposite orientations, remove that pair; repeat until no such pair exists.\footnote{This step may be unnecessary for unpunctured surfaces, but is necessary when punctures are introduced, and should be repeated after each mutation.}
\end{enumerate}
Quiver mutation is then related by a \emph{flip} of a diagonal in any quadrilateral, as pictured in Figure \ref{fig: flip mutation}. 
Seed mutation then has a geometric description: as depicted in Figure \ref{fig: Ptolemy}, the relation between the two diagonals of a given quadrilateral is given by the Ptolemy relation: the product of the diagonals equals the sum of products of opposite sides.

\begin{figure}[tb]
	\captionsetup[subfigure]{justification=centering}
	\begin{subfigure}[b]{0.45\textwidth}
		\centering
		\begin{tikzpicture}[scale=2]
			\draw[thick,fill=black!5] (0,0) circle (1);
			\node[dot] (1) at (72:1) {};
			\node[dot] (2) at (144:1) {};
			\node[dot] (3) at (216:1) {};
			\node[dot] (4) at (288:1) {};
			\node[dot] (5) at (0:1) {};
			
			\draw[thick] (5) to (2);
			\draw[thick] (2) to (4);
			
			\node[mutable3] (x1) at (90:0.33) {$x_1$};
			\node[mutable3] (x2) at (215:0.33) {$x_2$};
			\node[frozen] (x3) at (36:1) {$x_3$};
			\node[frozen] (x4) at (108:1) {$x_4$};
			\node[frozen] (x5) at (180:1) {$x_5$};
			\node[frozen] (x6) at (252:1) {$x_6$};
			\node[frozen] (x7) at (324:1) {$x_7$};
			
			\draw[-angle 90,purple] (x1) to (x4);
			\draw[-angle 90,purple] (x3) to (x1);
			\draw[-angle 90,purple] (x2) to (x1);
			\draw[-angle 90,purple] (x1) to (x7);
			\draw[-angle 90,purple] (x7) to (x2);
			\draw[-angle 90,purple] (x5) to (x2);
			\draw[-angle 90,purple] (x2) to (x6);
			
		\end{tikzpicture}
	\end{subfigure}
	\begin{subfigure}[b]{0.45\textwidth}
		\centering
		\begin{tikzpicture}[scale=2]
			\draw[thick,fill=black!5] (0,0) circle (1);
			\node[dot] (1) at (72:1) {};
			\node[dot] (2) at (144:1) {};
			\node[dot] (3) at (216:1) {};
			\node[dot] (4) at (288:1) {};
			\node[dot] (5) at (0:1) {};
			
			\draw[thick] (5) to (2);
			\draw[thick] (3) to (5);
			
			\node[mutable3] (x1) at (90:0.33) {$x_1$};
			\node[mutable3] (x2) at (285:0.33) {$x_2^\prime$};
			\node[frozen] (x3) at (36:1) {$x_3$};
			\node[frozen] (x4) at (108:1) {$x_4$};
			\node[frozen] (x5) at (180:1) {$x_5$};
			\node[frozen] (x6) at (252:1) {$x_6$};
			\node[frozen] (x7) at (324:1) {$x_7$};
			
			\draw[-angle 90,purple] (x1) to (x4);
			\draw[-angle 90,purple] (x3) to (x1);
			\draw[-angle 90,purple] (x1) to (x2);
			\draw[-angle 90,purple] (x2) to (x5);
			\draw[-angle 90,purple] (x5) to (x1);
			\draw[-angle 90,purple] (x6) to (x2);
			\draw[-angle 90,purple] (x2) to (x7);
			
		\end{tikzpicture}
	\end{subfigure}
	\caption{Quivers (in purple) associated to two triangulations of the unpunctured disk with five marked points on the boundary, related by a flip of the diagonal marked $x_2$.}
	\label{fig: flip mutation}
\end{figure}
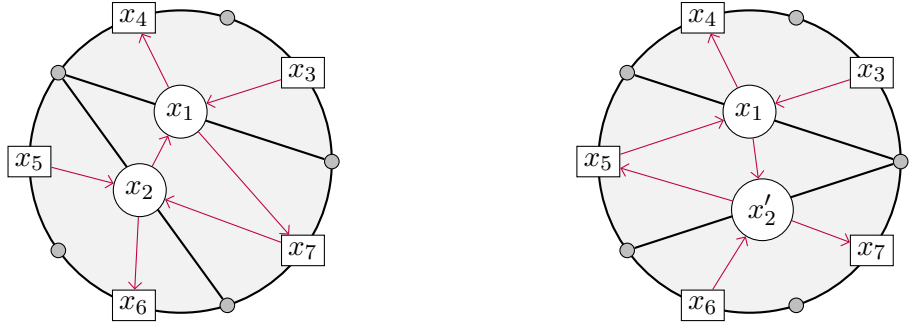

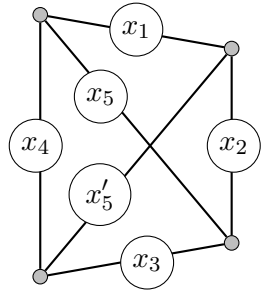
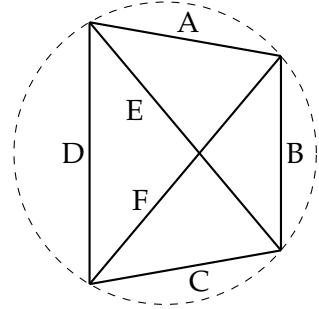
\begin{figure}[htb]
	\centering
	\captionsetup[subfigure]{justification=centering}
	\begin{subfigure}[b]{0.48\textwidth}
		\centering
		\begin{tikzpicture}[scale=2]
			\draw[draw=none] (0,0) circle (1);
			\node[dot] (1) at (40:1) {};
			\node[dot] (4) at (320:1) {};
			\node[dot] (3) at (240:1) {};
			\node[dot] (2) at (120:1) {};
			\draw[thick] (1)--(2) {}; 
			\draw[thick] (2)--(3) {}; 
			\draw[thick] (3)--(4) {}; 
			\draw[thick] (4)--(1) {}; 
			\draw[thick] (1)--(3) {}; 
			\draw[thick] (4)--(2) {};
			\node[mutable3] at (80:0.78) {$x_1$};
			\node[mutable3] at (180:0.53) {$x_4$};
			\node[mutable3] at (285:0.8) {$x_3$};
			\node[mutable3] at (0:0.78) {$x_2$};
			\node[mutable3] at (107:0.34) {$x_5$};
			\node[mutable3] at (252:0.34) {$x_5^\prime$};
		\end{tikzpicture}
		\subcaption{
			A Ptolemy relation among variables\\in a cluster algebra
			\\$x_5 x_5^\prime = x_1 x_3 + x_2 x_4$
		}
		\label{fig: Ptolemy cluster variables}
	\end{subfigure}
	\begin{subfigure}[b]{0.48\textwidth}
		\centering
		\begin{tikzpicture}[scale=2]
			\draw[thin,dashed] (0,0) circle (1);
			\draw[thick] (40:1)--(120:1) {}; 
			\draw[thick] (120:1)--(240:1) {}; 
			\draw[thick] (240:1)--(320:1) {}; 
			\draw[thick] (320:1)--(40:1) {}; 
			\draw[thick] (40:1)--(240:1) {}; 
			\draw[thick] (320:1)--(120:1) {};
			\node at (80:0.88) {A};
			\node at (180:0.61) {D};
			\node at (285:0.88) {C};
			\node at (0:0.86) {B};
			\node at (125:0.35) {E};
			\node at (242:0.35) {F};
		\end{tikzpicture}
		\subcaption{
			A Ptolemy relation among lengths\\between four cocircular points
			\\$EF=AC+BD$
		}
		\label{fig: Ptolemy cocircular}
	\end{subfigure}
	
	\caption{The Ptolemy relations}
	\label{fig: Ptolemy}
\end{figure}

%============================

\subsection{Cluster algebras of punctured surfaces}

While all triangulations of unpunctured surfaces are related by flips of diagonals in quadrilaterals, this is no longer true in the presence of punctures. Consider the bigon in Figure \ref{fig: self folded triangle}; there is no way to replace the arc $\alpha$ with a different arc to produce a different triangulation. To fix this issue, we utilize the notion of \emph{tagged triangulations}. 

\begin{defn}\cite[Definition 7.1]{FST08}\label{def: tagged arc}
	An arc $\gamma$ in $\SS$ has two \emph{ends} obtained by arbitrarily cutting the arc into three connected pieces and discarding the middle one. A \textbf{tagged arc} is an arc in which each of its ends has been tagged in one of two ways \textemdash \emph{plain} or \emph{notched} \textemdash so that the following conditions are met:
	\begin{itemize}
		\item the arc does not cut out a once-punctured monogon,
		\item an endpoint at the boundary is tagged plain, and 
		\item if both ends of the arc are at the same marked point (or puncture), then they are tagged in the same way.
	\end{itemize}
\end{defn}

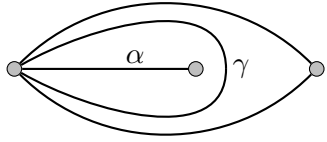
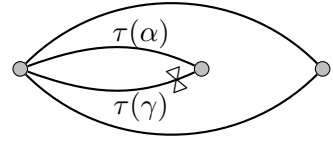
\begin{figure}[htb]
	\captionsetup[subfigure]{justification=centering}
	\begin{subfigure}[b]{0.45\textwidth}
		\centering
		\begin{tikzpicture}[scale=2]
			\node[dot] (1) at (-1,0) {};
			\node[dot] (2) at (1,0) {};
			\draw[thick,out=45,in=135] (1) to (2) {};
			\draw[thick,out=315,in=225] (1) to (2) {};
			\node[dot] (p) at (0.2,0) {};
			\draw[thick] (1) to (p) {};
			\node at (-0.2,0.08) {$\alpha$};
			\node at (0.5,0) {$\gamma$};
			\draw[thick,out=30,in=90] (1) to (0.4,0) {};
			\draw[thick,out=270,in=330] (0.4,0) to (1) {};
		\end{tikzpicture}
		\subcaption{A self-folded triangle.}
		\label{fig: self folded triangle}
	\end{subfigure}
	\begin{subfigure}[b]{0.45\textwidth}
		\centering
		\begin{tikzpicture}[scale=2]
			\node[dot] (1) at (-1,0) {};
			\node[dot] (2) at (1,0) {};
			\draw[thick,out=45,in=135] (1) to (2) {};
			\draw[thick,out=315,in=225] (1) to (2) {};
			\node[dot] (p) at (0.2,0) {};
			\node at (-0.2,0.23) {$\tau(\alpha)$};
			\node at (-0.2,-0.25) {$\tau(\gamma)$};
			\draw[thick,out=20,in=155] (1) to (p) {};
			\draw[thick,out=-20,in=205] (1) to node[pos=0.9,sloped,rotate=90]{$\bowtie$} (p) {};
		\end{tikzpicture}
		\subcaption{A tagged triangulation.}
		\label{fig: notched arc}
	\end{subfigure}
	\caption{A punctured bigon. Plain arcs are drawn without decoration and notched arcs are decorated with a bowtie (i.e. the $\bowtie$ symbol).}
	\label{fig: punctured bigon}
\end{figure}

\noindent
Each ordinary arc $\gamma$ can be replaced with a tagged arc $\tau(\gamma)$ as follows:
\begin{itemize}
	\item If $\gamma$ does not cut out a punctured monogon, then $\tau(\gamma)$ is $\gamma$ with both ends tagged plain.
	\item If $\gamma$ cuts out a punctured monogon, then both ends of $\gamma$ are adjacent to the point $a$. Labeling the puncture $b$, there is a unique ordinary arc $\alpha$ connecting $a$ to $b$ that is compatible with $\gamma$. Then $\tau(\gamma)$ is obtained by tagging $\alpha$ plain at $a$ and notched at $b$.
\end{itemize}

\begin{rem}
	Note that since $\SS$ is triangulable, it is not a sphere with three punctures. As a result, an arc $\gamma$ cannot cut out punctured monogons on both sides, so the punctured monogon cut out by $\gamma$ is unique.
\end{rem}

Referring again to Figure \ref{fig: punctured bigon}, we can see that the arc $\gamma$ in Figure \ref{fig: self folded triangle} is replaced with the arc $\tau(\gamma)$ in Figure \ref{fig: notched arc}, a copy of the ordinary arc $\alpha$ which is notched at the puncture. A \textbf{tagged triangulation} is then a maximal collection of compatible tagged arcs, according to the compatibility relation defined below.

\begin{defn}\cite[Definition 7.4]{FST08}\label{def: tagged arc compatability}
	Two tagged arcs $\alpha, \beta \in \SS$ are called \textbf{compatible} if and only if the following conditions are satisfied:
	\begin{itemize}
		\item the untagged versions of $\alpha$ and $\beta$ are compatible;
		\item if the untagged versions of $\alpha$ and $\beta$ are different, and $\alpha$ and $\beta$ share an endpoint at $a$, then the ends of $\alpha$ and $\beta$ connected to $a$ must have the same tagging;
		\item if the untagged versions of $\alpha$ and $\beta$ coincide, then at least one end of $\alpha$ must be tagged in the same way as the corresponding end of $\beta$.
	\end{itemize}
\end{defn}

\begin{rem}\cite[Remark 7.6]{FST08}
	If two plain arcs $\alpha$ and $\beta$ are compatible, then the tagged arcs $\tau(\alpha)$ and $\tau(\beta)$ are compatible. The converse is false: in a once-punctured bigon, the two curves that cut out once-punctured monogons are not compatible even though the tagged arcs representing them are compatible. 
\end{rem}

Given this definition of compatibility, every non-boundary arc in a triangulation of $\SS$ can be removed and replaced with another arc to produce a different triangulation, resulting in the following theorem:
\begin{thm}\cite[Theorem 7.10]{FST08}\label{thm: tagged flips}
	If $\SS$ is not a closed surface with exactly one puncture, then any two tagged triangulations are connected by a sequence of flips.
\end{thm}
\noindent
As an example, flips of tagged arcs inside of a punctured bigon are demonstrated in Figure \ref{fig: punctured bigon tagged triangulations}.

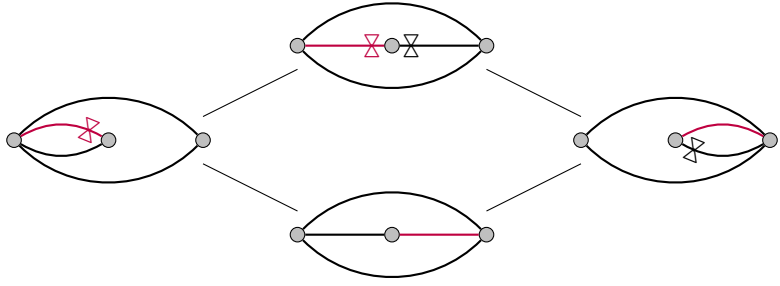
\begin{figure}[htb]
	\centering
	\begin{tikzpicture}[scale=1.25]
		\node[dot] (a1) at (-4,0) {};
		\node[dot] (a2) at (-2,0) {};
		\draw[thick,out=45,in=135] (a1) to (a2) {};
		\draw[thick,out=315,in=225] (a1) to (a2) {};
		\node[dot] (ap) at (-3,0) {};
		\draw[thick,out=30,in=150,purple] (a1) to node[pos=0.85,sloped,rotate=90]{$\bowtie$} (ap) {};
		\draw[thick,out=-30,in=210] (a1) to (ap) {};
		
		\node[dot] (b1) at (-1,1) {};
		\node[dot] (b2) at (1,1) {};
		\draw[thick,out=45,in=135] (b1) to (b2) {};
		\draw[thick,out=315,in=225] (b1) to (b2) {};
		\node[dot] (bp) at (0,1) {};
		\draw[thick,purple] (b1) to node[pos=0.85,sloped,rotate=90]{$\bowtie$} (bp) {};
		\draw[thick] (b2) to node[pos=0.85,sloped,rotate=90]{$\bowtie$} (bp) {};
		
		\node[dot] (c1) at (-1,-1) {};
		\node[dot] (c2) at (1,-1) {};
		\draw[thick,out=45,in=135] (c1) to (c2) {};
		\draw[thick,out=315,in=225] (c1) to (c2) {};
		\node[dot] (cp) at (0,-1) {};
		\draw[thick] (c1) to (cp) {};
		\draw[thick,purple] (c2) to (cp) {};
		
		\node[dot] (d1) at (2,0) {};
		\node[dot] (d2) at (4,0) {};
		\draw[thick,out=45,in=135] (d1) to (d2) {};
		\draw[thick,out=315,in=225] (d1) to (d2) {};
		\node[dot] (dp) at (3,0) {};
		\draw[thick,out=150,in=30,purple] (d2) to (dp) {};
		\draw[thick,out=210,in=330] (d2) to node[pos=0.85,sloped,rotate=90]{$\bowtie$} (dp) {};
		
		\draw[thin] (-2,0.25) to (-1,0.75) {};
		\draw[thin] (-2,-0.25) to (-1,-0.75) {};
		\draw[thin] (2,0.25) to (1,0.75) {};
		\draw[thin] (2,-0.25) to (1,-0.75) {};
	\end{tikzpicture}
	\caption{Tagged triangulations of a punctured bigon. Edges connect tagged triangulations related by a flip of a single tagged arc.}
	\label{fig: punctured bigon tagged triangulations}
\end{figure}

Given a triangulation of $\SS$ without self-folded triangles, the construction of a quiver follows in exactly the same way as in the unpunctured case. In the presence of self-folded triangles, the quiver can be generated by treating both arcs of the self-folded triangle as the third side of the other triangle containing the arc cutting out a punctured monogon, as in Figure \ref{fig: folded triangle quivers}. 

\begin{figure}[tb]
	\captionsetup[subfigure]{justification=centering}
	\begin{subfigure}[b]{0.45\textwidth}
		\centering
		\begin{tikzpicture}[scale=2]
			\draw[thick,fill=black!5] (0,0) circle (1);
			\node[dot] (1) at (90:1) {};
			\node[dot] (2) at (270:1) {};
			\node[dot] (3) at (0,0) {};
			
			\node[frozen] (x1) at (135:1) {$x_3$};
			\node[frozen] (x2) at (45:1) {$x_4$};
			\node[mutable] (x3) at (90:0.3) {$x_2$};
			
			\draw[thick,out=120,in=180] (2) to (x3) {};
			\draw[thick,out=0,in=60] (x3) to (2) {};
			\draw[thick] (2) to (3);
			
			\node[mutable] (x4) at (270:0.5) {$x_1$};
			
			\draw[-angle 90,purple] (x1) to (x2);
			\draw[-angle 90,purple] (x2) to (x3);
			\draw[-angle 90,purple] (x3) to (x1);
			\draw[-angle 90,purple] (x2) to (x4);
			\draw[-angle 90,purple] (x4) to (x1);
		\end{tikzpicture}
		\subcaption{A once-punctured bigon.}
	\end{subfigure}
	\begin{subfigure}[b]{0.45\textwidth}
		\centering
		\begin{tikzpicture}[scale=2]
			\draw[thick,fill=black!5] (0,0) circle (1);
			\node[dot] (1) at (270:1) {};
			\node[dot] (2) at (0:0.6) {};
			\node[dot] (3) at (180:0.6) {};
			
			\node[frozen] (x1) at (90:1) {$x_5$};
			\node[mutable] (x2) at (30:0.7) {$x_4$};
			\node[mutable] (x4) at (150:0.7) {$x_2$};
			
			\draw[thick] (1) to (2) {};
			\draw[thick] (1) to (3) {};
			\draw[thick,out=165,in=220] (1) to (x4) {};
			\draw[thick,out=100,in=350] (1) to (x4) {};
			\draw[thick,out=80,in=190] (1) to (x2) {};
			\draw[thick,out=15,in=320] (1) to (x2) {};
			
			\node[mutable] (x3) at (315:0.55) {$x_3$};
			\node[mutable] (x5) at (225:0.55) {$x_1$};
			
			\draw[-angle 90,purple] (x1) to (x2);
			\draw[-angle 90,purple] (x1) to (x3);
			\draw[-angle 90,purple] (x2) to (x4);
			\draw[-angle 90,purple] (x2) to (x5);
			\draw[-angle 90,purple] (x3) to (x4);
			\draw[-angle 90,purple] (x3) to (x5);
			\draw[-angle 90,purple] (x4) to (x1);
			\draw[-angle 90,purple] (x5) to (x1);
		\end{tikzpicture}
		\subcaption{A twice-punctured monogon.}
        \label{fig: twice-punc monogon quiver}
	\end{subfigure}
	\caption{Quivers (in purple) associated to triangulations of a once-punctured bigon and a twice-punctured monogon.}
	\label{fig: folded triangle quivers}
\end{figure}
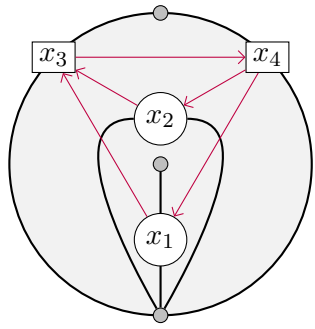
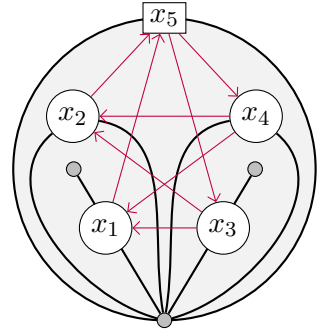

While the Ptolemy relations from triangulations of unpunctured surfaces still hold for tagged triangulations of punctured surfaces, they can become more complicated. As noted in \cite[Definition 8.4]{FT18}, some of the arcs bounding the quadrilateral may bound a once-punctured monogon, and thus, may not be present in all triangulations involving the other arcs. In such a case, we should replace every such arc by the product of the two tagged arcs that it encloses. See Figure \ref{fig: Ptolemy 4-punctured sphere} (reproduced from Figure 19 of \cite[Chapter 8]{FT18}) for an example of a Ptolemy relation in a tagged triangulation of a four-punctured sphere.

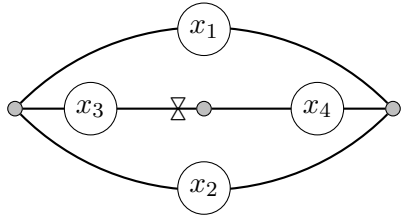
\begin{figure}[htb]
	\centering
	\begin{tikzpicture}[scale=3]
		\node[dot] (1) at (-1,0) {};
		\node[dot] (2) at (1,0) {};
		\node[dot] (3) at (-1,1) {};
		\node[dot] (4) at (1,1) {};
		\draw[thick] (1) to (2) {};
		\draw[thick] (1) to node[pos=0.85,sloped,rotate=90]{$\bowtie$} (3) {};
		\draw[thick] (2) to node[pos=0.85,sloped,rotate=90]{$\bowtie$} (4) {};
		\draw[thick,out=75,in=345] (1) to (3) {};
		\draw[thick,out=75,in=345] (2) to (4) {};
		\draw[thick,out=160,in=180] (1) to (-1.5,1.6) {};
		\draw[thick,out=0,in=160] (-1.5,1.6) to (2) {};
		\draw[thick,out=20,in=180] (1) to (1.5,1.6) {};
		\draw[thick,out=0,in=20] (1.5,1.6) to (2) {};
		
		\node[mutable] (x1) at (-1.05,0.5) {$x_1$};
		\node[mutable] (x2) at (-0.8,0.6) {$x_2$};
		\node[mutable] (x3) at (0.95,0.5) {$x_3$};
		\node[mutable] (x4) at (1.2,0.6) {$x_4$};
		\node[mutable] (x5) at (0,0) {$x_5$};
		\node[mutable] (x6) at (-0.6,1.25) {$x_6$};
		\node[mutable] (x7) at (0.55,1.25) {$x_7$};
	\end{tikzpicture}
	\begin{equation} x_6 x_7 = x_1 x_2 x_3 x_4 + x_5^2 \end{equation}
	\caption{Ptolemy relation for a tagged flip in a 4-punctured sphere.}
	\label{fig: Ptolemy 4-punctured sphere}
\end{figure}

In addition to the necessary adaptation of the Ptolemy relations in the case where one of the bounding arcs also bounds a punctured monogon, we also must adapt the Ptolemy relation in the case of a punctured bigon. As depicted in Figure \ref{fig: bigon relation}, the product of two incompatible arcs inside a punctured bigon is equal to the sum of the two bounding arcs \cite[Definition 8.5]{FT18}.

\begin{figure}[htb]
	\centering
	\begin{tikzpicture}[scale=2.5]
		\node[dot] (1) at (-1,0) {};
		\node[dot] (2) at (1,0) {};
		\draw[thick,out=45,in=135] (1) to (2) {};
		\draw[thick,out=315,in=225] (1) to (2) {};
		\node[dot] (p) at (0,0) {};
		\draw[thick] (1) to node[pos=0.9,sloped,rotate=90]{$\bowtie$} (p) {};
		\draw[thick] (2) to (p) {};
		
		\node[mutable] (x1) at (0,0.42) {$x_1$};
		\node[mutable] (x2) at (0,-0.42) {$x_2$};
		\node[mutable] (x3) at (-0.6,0) {$x_3$};
		\node[mutable] (x4) at (0.6,0) {$x_4$};
	\end{tikzpicture}
	\begin{equation} x_3 x_4 = x_1 + x_2 \end{equation}
	
	\caption{Ptolemy relation for a punctured bigon.}
	\label{fig: bigon relation}
\end{figure}

%============================

\subsection{Kauffman skein algebras of marked surfaces}

The connection between cluster algebras of unpunctured marked surfaces with boundary and \emph{Kauffman skein algebras} was explored in \cite{MW13,MSW13,Mul16}. 
In many cases it is more practical to work with the skein algebra of the marked surface, which is generated by marked curves with topologically-defined relations.

Given a marked surface $\SS$, let $\mathbb{Z}\Multi(\SS)$ be the abelian group of $\mathbb{Z}$-linear combinations of marked multicurves in $\SS$ up to homotopy.  
Define the \textbf{skein algebra}\footnote{The \emph{skein algebra} $\Sk_1(\SS)$ as defined here is more accurately called the \emph{marked Kauffman skein algebra of $(\SS,\MM)$ at $q=1$}.} $\Sk_1(\SS)$ to be the quotient of $\mathbb{Z}\Multi(\SS)$ by the subgroup generated by the families of relations depicted in Figure \ref{fig: skeinrel}. In each figure, the dashed circle represents a small contractible neighborhood in $\SS$, and all elements considered agree outside this circle.
We denote the element of $\Sk_1(\SS)$ corresponding to a marked multicurve $\gamma$ by $\langle\gamma\rangle$.

Given two marked multicurves $\gamma$ and $\lambda$, their \textbf{superposition} $\gamma\cup \lambda$ is the union of the corresponding immersed curves. This may not be a marked multicurve, as the crossings may no longer be transverse or have more than two preimages. To resolve this, $\gamma$ and $\lambda$ can be replaced with equivalent marked multicurves $\gamma'$ and $\lambda'$ whose superposition is a marked multicurve. While the resulting multicurve $\gamma' \cup \lambda'$ depends on the choice of $\gamma'$ and $\lambda'$, the element $\langle \gamma' \cup \lambda'\rangle $ in the skein algebra does not \cite[Proposition 3.5]{Mul16}. This gives a well-defined \textbf{superposition product} on the skein algebra $\Sk_1(\SS)$.

\begin{thm}{\cite[Corollary 6.16]{Mul16}}
	Let $\SS$ be an unpunctured, triangulable marked surface. Then the superposition product makes $\Sk_1(\SS)$ into a commutative domain. 
\end{thm}

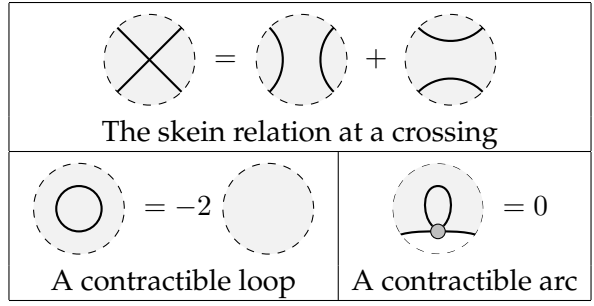
\begin{figure}[htb]
	\centering
	\begin{tabular}{|c|c|}
		\hline
		\multicolumn{2}{|c|}{
			$
			\begin{tikzpicture}[scale=.15,baseline=(invisible.base)]
				\path[use as bounding box] (0,0) circle (5);
				\draw[fill=black!5,dashed] (0,0) circle (4);
				\draw[thick] (-2.83,-2.83) to (2.83,2.83);
				\draw[thick] (-2.83,2.83) to (2.83,-2.83);
				\node[opacity=0] (invisible) at (0,0) {$a$};
			\end{tikzpicture}
			=
			\begin{tikzpicture}[scale=.15,baseline=(invisible.base)]
				\path[use as bounding box] (0,0) circle (5);
				\draw[fill=black!5,dashed] (0,0) circle (4);
				\draw[thick] (-2.83,-2.83) to [out=45,in=-45] (-2.83,2.83);
				\draw[thick] (2.83,-2.83) to [out=135,in=-135] (2.83,2.83);
				\node[opacity=0] (invisible) at (0,0) {$a$};
			\end{tikzpicture}
			+
			\begin{tikzpicture}[scale=.15,baseline=(invisible.base)]
				\path[use as bounding box] (0,0) circle (5);
				\draw[fill=black!5,dashed] (0,0) circle (4);
				\draw[thick] (-2.83,-2.83) to [out=45,in=135] (2.83,-2.83);
				\draw[thick] (-2.83,2.83) to [out=-45,in=-135] (2.83,2.83);
				\node[opacity=0] (invisible) at (0,0) {$a$};
			\end{tikzpicture}
			$} \\
		\multicolumn{2}{|c|}{
			The skein relation at a crossing
		}
		\\
		\hline
		$
		\begin{tikzpicture}[scale=.15,baseline=(invisible.base)]
			\path[use as bounding box] (0,0) circle (5);
			\draw[fill=black!5,dashed] (0,0) circle (4);
			\draw[thick] (0,0) circle (2);
			\node[opacity=0] (invisible) at (0,0) {$a$};
		\end{tikzpicture}
		=-2
		\begin{tikzpicture}[scale=.15,baseline=(invisible.base)]
			\path[use as bounding box] (0,0) circle (5);
			\draw[fill=black!5,dashed] (0,0) circle (4);
			\node[opacity=0] (invisible) at (0,0) {$a$};
		\end{tikzpicture}
		$
		&
		$
		\begin{tikzpicture}[scale=.15,baseline=(invisible.base)]
			\path[use as bounding box] (0,0) circle (5);
			\clip (0,0) circle (4);
			\draw[fill=black!5,thick] (-5,-3) to [in=180,out=30] (0,-2) to [in=150,out=0] (5,-3) to [line to] (5,5) to (0,5) to (-5,5);
			\node (1) at (0,-2) [dot] {};
			\draw[thick] (1) to [out=45,in=0] (0,2) to [out=180,in=135] (1);
			\draw[dashed] (0,0) circle (4);
			\node[opacity=0] (invisible) at (0,0) {$a$};
		\end{tikzpicture}
		=0
		$
		\\
		A contractible loop
		&
		A contractible arc
		\\
		\hline
	\end{tabular}
	\caption{The $q=1$ marked Kauffman skein relations
	}
	\label{fig: skeinrel}
\end{figure}

Extension of the Kauffman skein relations to punctured surfaces has been explored in \cite{MSW13, Thu14,FT18}, and more recently in \cite{Wilson26,glw2024banglefunctionsgenericbasis,mandel2023braceletsbasesthetabases,bwk2024skeinrelationspuncturedsurfaces}. In addition to the relations of Figure \ref{fig: skeinrel}, the tagged skein algebra requires two additional relations pictured in Figure \ref{fig: tagged skeinrel}. 

Similar to the unpunctured case, let $\mathbb{Z}\TMulti(\SS)$ denote the abelian group of $\mathbb{Z}$-linear combinations of tagged marked multicurves in $\SS$ up to homotopy. Define the \textbf{tagged skein algebra} $\barTSk (\SS)$ to be the quotient of $\mathbb{Z}\TMulti(\SS)$ by the subgroup generated by the families of relations depicted in Figures \ref{fig: skeinrel} and \ref{fig: tagged skeinrel}. Further, define a version of the tagged skein algebra localized at the boundary, denoted $\TSk (\SS) \coloneqq \barTSk (\SS) [\partial\SS^{\pm 1}]$.

\begin{figure}[tb]
	\centering
	\begin{tabular}{|c|c|}
		\hline
		$
		\begin{tikzpicture}[scale=.15,baseline=(invisible.base)]
			\path[use as bounding box] (0,0) circle (5);
			\draw[fill=black!5,dashed] (0,0) circle (4);
			\draw[thick] (0,0) circle (2.7);
			\node[dot] (0,0) {};
		\end{tikzpicture}
		=2
		\begin{tikzpicture}[scale=.15,baseline=(invisible.base)]
			\path[use as bounding box] (0,0) circle (5);
			\draw[fill=black!5,dashed] (0,0) circle (4);
			\node[dot] (0,0) {};
		\end{tikzpicture}
		$
		&
		$
		\begin{tikzpicture}[scale=.15,baseline=(invisible.base)]
			\path[use as bounding box] (0,0) circle (5);
			\draw[fill=black!5,dashed] (0,0) circle (4);
			\node[dot] (p) at (0,0) {};
			\draw[thick] (-4,0) to node[pos=0.67,sloped,rotate=90]{$\bowtie$} (p) {};
			\draw[thick] (p) to (4,0) {};
		\end{tikzpicture}
		=
		\begin{tikzpicture}[scale=.15,baseline=(invisible.base)]
			\path[use as bounding box] (0,0) circle (5);
			\draw[fill=black!5,dashed] (0,0) circle (4);
			\node[dot] (p) at (0,0) {};
			\draw[thick, out=45,in=135] (-4,0) to (4,0) {};
		\end{tikzpicture}
		+
		\begin{tikzpicture}[scale=.15,baseline=(invisible.base)]
			\path[use as bounding box] (0,0) circle (5);
			\draw[fill=black!5,dashed] (0,0) circle (4);
			\node[dot] (p) at (0,0) {};
			\draw[thick, out=315, in=225] (-4,0) to (4,0) {};
		\end{tikzpicture}
		$
		\\
		A loop around a puncture
		&
		Incompatible taggings
		\\
		\hline
	\end{tabular}
	\caption{The $q=1$ tagged Kauffman skein relations}
	\label{fig: tagged skeinrel}
\end{figure}
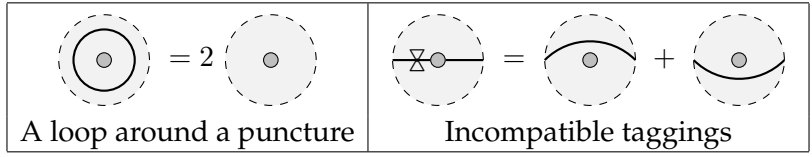

The reader should refer to \cite{mandel2023braceletsbasesthetabases} for more details on tagged skein algebras. We will introduce details only as needed, beginning with the following result.

\begin{lemma}\cite[Lemma 3.10]{mandel2023braceletsbasesthetabases}
	Simultaneously changing all tags at a puncture $p$ induces an involutive automorphism $i_p$ of $\TSk (\SS)$.
\end{lemma}

%============================

\subsubsection{Relationship between cluster algebras and skein algebras}

The cluster algebra of an unpunctured marked surface can be defined as a localization of the skein algebra.

\begin{thm}\cite[Theorem 6.9]{BM25}
	Let $\SS$ be an unpunctured, triangulable marked surface. Then the localization 
	\[ \Sk_1(\SS)[\partial\SS^{-1}] \]
	of the skein algebra $\Sk_1(\SS)$ at the set of boundary arcs $\partial \SS$ is a cluster algebra. The frozen cluster variables are the boundary arcs, the mutable cluster variables are the simple non-boundary arcs, and the clusters are the triangulations of $\SS$.
\end{thm}

When the surface is punctured, there may not necessarily be an isomorphism between the cluster algebra and the localized commutative tagged skein algebra $\TSk_1 (\SS)$. We can show that this is an isomorphism in many cases, however.

\begin{thm}\label{thm: tagged skein cluster algebra}
	Let $\SS$ be a triangulable marked surface (possibly with punctures) such that at least one of the following is true:
	\begin{enumerate}
		\item $\SS$ is a disk.
		\item $\SS$ embeds into a disk. 
		\item $\SS$ has at least two boundary marked points in each connected component. 
	\end{enumerate}. 
	Then the localized tagged skein algebra $\TSk(\SS)$ is a cluster algebra. The frozen cluster variables are the boundary arcs, the mutable cluster variables are the simple (tagged) non-boundary arcs, and the clusters are the triangulations of $\SS$.
\end{thm}

\begin{proof}
	Let $\CA(\SS)$ be the Fomin-Shapiro-Thurston cluster algebra of $\SS$ and let $\CU(\SS)$ be its upper cluster algebra. Then $\CA(\SS)$ is locally acyclic by Lemma 10.4, Theorem 10.5, and Theorem 10.6 of \cite{Mul13}. Consequently, $\CA(\SS) = \CU(\SS)$ by \cite[Theorem 2]{Mul14}.
	
	By \cite[Proposition 3.18]{mandel2023braceletsbasesthetabases},
	there are containments 
	\[\mathcal{A}(\SS)\subseteq \TSk_1 (\SS) \subseteq \mathcal{U}(\SS).\]
	Therefore, all three algebras are equal.
\end{proof}	

We now have the following characterization, analogous to \cite[Corollary 6.10]{BM25}.

\begin{coro}
	Let $\SS$ be a triangulable marked surface (possibly with punctures) such that at least one of the following is true:
	\begin{enumerate}
		\item $\SS$ is a disk.
		\item $\SS$ embeds into a disk. 
		\item $\SS$ has at least two boundary marked points in each connected component. 
	\end{enumerate} 
	A point $p\in V(\CA(\SS),\Bbbk)$ is equivalent to a map 
	\[ p:\{\text{simple marked curves in $\SS$ up to homotopy}\} \rightarrow \Bbbk\]
	such that the tagged skein relations hold (Figures \ref{fig: skeinrel} and \ref{fig: tagged skeinrel}) and boundary arcs are not sent to $0$. Such a point is deep if and only if every triangulation contains an arc which is sent to $0$.
\end{coro}

From this point on, we identify points of $V(\CA(\SS),\Bbbk)$ with such functions without comment.

%============================

\subsection{Cutting and gluing}\label{section: cutting}

Given a simple marked arc $s$ in a marked surface $\SS$, we may construct a new marked surface $\SS\smallsetminus s$ by \textbf{cutting} along $s$; that is, removing $s$ and adding two copies of $s$ (one to each side of the original $s$). 
More generally, given a compatible collection of non-boundary marked arcs $S$, we construct $\SS\smallsetminus S$ by cutting along the arcs in $S$ in any order.

For unpunctured surfaces, the cluster algebras of $\SS$ and $\SS\smallsetminus S$ are related as follows.

\begin{prop}\cite[Proposition 6.11]{BM25}\label{prop: cutting algebras}
	Let $\SS$ be an unpunctured, triangulable marked surface, and let $S$ be a compatible collection of non-boundary marked arcs in $\SS$. Then the map $\SS\smallsetminus S\rightarrow \SS$ induces an isomorphism
	\[  \CA(\SS\smallsetminus S)/\langle s'-s'', \forall s\in S\rangle \xrightarrow{\sim} \CA(\SS)[S^{-1}] \]
	Here, $s'$ and $s''$ denote the two preimages of $s$ in $\SS\smallsetminus S$.
\end{prop}

The isomorphism from Proposition \ref{prop: cutting algebras} fits into the following diagram, in which the map on the left is the localization map and the map on the right is the quotient map.
\begin{equation}\label{eq: cuttingalgebras}
	\begin{tikzcd}
		{\CA(\SS)} & {\CA(\SS\smallsetminus S)} \\
		{\CA(\SS)[S^{-1}]} & {\CA(\SS\smallsetminus S)/\langle s'-s'', \forall s\in S\rangle}
		\arrow["\sim"', from=2-2, to=2-1]
		\arrow[two heads, from=1-2, to=2-2]
		\arrow[hook, from=1-1, to=2-1]
	\end{tikzcd}
\end{equation}
The isomorphism can be translated to cluster varieties via the following terminology.
\begin{itemize}
	\item A point in $V(\CA(\SS),\Bbbk)$ is \textbf{non-zero on $S$} if it is non-zero on each arc in $S$.
	\item A point in $V(\CA(\SS\smallsetminus S),\Bbbk)$ is \textbf{gluable along $S$} if it has the same value on both preimages of each edge in $S$.
\end{itemize}
The first type of points are in bijection with homomorphisms $\CA(\SS)[S^{-1}]\rightarrow \Bbbk$ and the second type of points are in bijection with homomorphisms $\CA(\SS\smallsetminus S)/\langle s'-s'', \forall s\in S\rangle\rightarrow \Bbbk.$ Hence, Proposition \ref{prop: cutting algebras} implies the following.

\begin{prop}\cite[Proposition 6.13]{BM25}\label{prop: cutting subvarieties}
	Let $\SS$ be an unpunctured, triangulable marked surface, and let $S$ be a compatible collection of non-boundary marked arcs in $\SS$.
	Then restricting values from marked curves in $\SS$ to marked curves in $\SS\smallsetminus S$ gives an isomorphism between:
	\begin{enumerate}
		\item The open subvariety of $V(\CA(\SS),\Bbbk)$ consisting of points which are non-zero on $S$.
		\item The closed subvariety of $V(\CA(\SS\smallsetminus S),\Bbbk)$ consisting of points which are {gluable along $S$}.\qedhere
	\end{enumerate}
\end{prop}

\noindent We refer to this as the \textbf{cutting isomorphism} between the two subvarieties:
\begin{equation}\label{eq: cutting}
	\begin{Bmatrix} \text{points in }V(\CA(\SS),\Bbbk) \\ \text{non-zero on }S \end{Bmatrix} 
	\xrightarrow{\sim}
	\begin{Bmatrix} \text{points in }V(\CA(\SS\smallsetminus S),\Bbbk) \\ \text{gluable along }S \end{Bmatrix}
\end{equation}
We refer to the action of the inverse map as \textbf{gluing} points.

%============================

\subsubsection{Cutting and gluing punctured surfaces}

We now wish to extend these cutting and gluing results to punctured surfaces. In order to do so, we must extend \cite[Lemma 4.11]{Mul16} to tagged multicurves in punctured surfaces.

We define a \textbf{simple (tagged) multicurve} $X$ in $\SS$ as a transverse (tagged) multicurve with with no crossings, no contractible loops, and no contractible arcs; let $[X]$ denote the corresponding element in $\TSk (\SS)$.  Additionally, for each element $x \in \TSk (\SS)$, there is a unique subset $\Supp (x)$ of simple multicurves (called the \emph{support} of $x$) and unique $\lambda_Y \in \ZZ$ for each $Y \in \Supp (x)$ such that \[ x = \sum_{Y \in \Supp (x)} \lambda_Y [Y]. \]

\begin{defn}\label{defn: intersection pairing}
	Given simple multicurves $X$ and $Y$, define the \textbf{intersection pairing} of $X$ and $Y$ as $\mu (X,Y) \coloneqq A+D$, where 
	\begin{itemize}
		\item $A$ is the minimum number of crossings between all $X^\prime$ and $Y^\prime$, over all transverse pairs $(X^\prime, Y^\prime)$ homotopic to $(X,Y)$; note that intersections at marked points are not counted, and 
		\item $D \coloneqq D(X,Y) \cdot D(Y,X)$, where $D(\alpha,\beta)$ is the number of ends of $\beta$ that are incident to an end of $\alpha$ and are tagged differently than $\alpha$ at that puncture.\footnote{A similar definition of intersection pairing appears in \cite[Definition 8.4]{FST08} with two additional (non-positive) summands. In that definition, $D$ is simply $D(X,Y)$, so the intersection pairing defined here will always be at least as large as that one.}
	\end{itemize} 
\end{defn}

We can now adapt the proof of \cite[Lemma 4.11]{Mul16} to the setting of tagged arcs.
\begin{lemma}\label{lemma: reducing tagged crossings}
	If $x$ is a simple tagged arc in $\SS$, then for all $y \in \TSk (\SS)$ such that $\mu ([x],y) > 0$, \[ \mu([x],[x]y) \leq \mu ([x],y)-1.\]
\end{lemma}
\begin{proof}
	First consider the case when $y$ is a simple multicurve $Y$ so that $x \cdot Y$ is transverse, and $x \cdot Y$ has $\mu(x, Y)$ crossings (the minimal number, up to homotopy). 
	Consider the set $I$ of multicurves which can be obtained by applying some combination of the following local relations to each crossing in $x \cdot Y$:
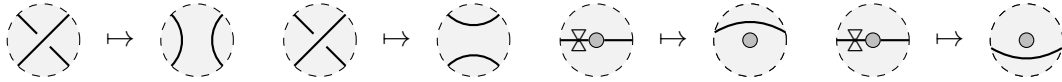
	\begin{center}
		\begin{tikzpicture}[scale=.8]
			\begin{scope}[xshift=-2.7in]
				\begin{scope}[xshift=-.5in,scale=.15]
					\draw[fill=black!5,dashed] (0,0) circle (4);
					\draw[thick] (-2.83,-2.83) to (2.83,2.83);
					\draw[thick] (-2.83,2.83) to (-.71,.71);
					\draw[thick] (.71,-.71) to (2.83,-2.83);
				\end{scope}
				\node (=) at (0,0) {$\mapsto$};
				\begin{scope}[xshift=.5in,scale=.15]
					\draw[fill=black!5,dashed] (0,0) circle (4);
					\draw[thick] (-2.83,-2.83) to [out=45,in=-45] (-2.83,2.83);
					\draw[thick] (2.83,-2.83) to [out=135,in=-135] (2.83,2.83);
				\end{scope}
			\end{scope}
			
			\begin{scope}[xshift=-0.9in]
				\begin{scope}[xshift=-.5in,scale=.15]
					\draw[fill=black!5,dashed] (0,0) circle (4);
					\draw[thick] (-2.83,-2.83) to (2.83,2.83);
					\draw[thick] (-2.83,2.83) to (-.71,.71);
					\draw[thick] (.71,-.71) to (2.83,-2.83);
				\end{scope}
				\node (=') at (0,0) {$\mapsto$};
				\begin{scope}[xshift=.5in,scale=.15]
					\draw[fill=black!5,dashed] (0,0) circle (4);
					\draw[thick] (-2.83,-2.83) to [out=45,in=135] (2.83,-2.83);
					\draw[thick] (-2.83,2.83) to [out=-45,in=-135] (2.83,2.83);
				\end{scope}
			\end{scope}
			\begin{scope}[xshift=0.9in]
				\begin{scope}[xshift=-.5in,scale=.15]
					\draw[fill=black!5,dashed] (0,0) circle (4);
					\node[dot] (p) at (0,0) {};
					\draw[thick] (-4,0) to node[pos=0.67,sloped,rotate=90]{$\bowtie$} (p) {};
					\draw[thick] (p) to (4,0) {};
				\end{scope}
				\node (=') at (0,0) {$\mapsto$};
				\begin{scope}[xshift=.5in,scale=.15]
					\draw[fill=black!5,dashed] (0,0) circle (4);
					\draw[fill=black!5,dashed] (0,0) circle (4);
					\node[dot] (p) at (0,0) {};
					\draw[thick, out=30,in=150] (-3.9,0.87) to (3.9,0.87) {};
				\end{scope}
			\end{scope}
			\begin{scope}[xshift=2.7in]
				\begin{scope}[xshift=-.5in,scale=.15]
					\draw[fill=black!5,dashed] (0,0) circle (4);
					\node[dot] (p) at (0,0) {};
					\draw[thick] (-4,0) to node[pos=0.67,sloped,rotate=90]{$\bowtie$} (p) {};
					\draw[thick] (p) to (4,0) {};
				\end{scope}
				\node (=') at (0,0) {$\mapsto$};
				\begin{scope}[xshift=.5in,scale=.15]
					\draw[fill=black!5,dashed] (0,0) circle (4);
					\draw[fill=black!5,dashed] (0,0) circle (4);
					\node[dot] (p) at (0,0) {};
					\draw[thick, out=330, in=210] (-3.9,-0.87) to (3.9,-0.87) {};
				\end{scope}
			\end{scope}
		\end{tikzpicture}
	\end{center}
	Since the simple multicurves in the support $\Supp([x][Y])$ come from applying the above relations to the crossings in $x \cdot Y$, we have $\Supp([x][Y])\subset I$.
	
	Consider a simple multicurve $Z\in I$.  For two adjacent crossings in $x \cdot Y$ along $x$, there are two local possibilities for $Z$, up to reflection across $x$.
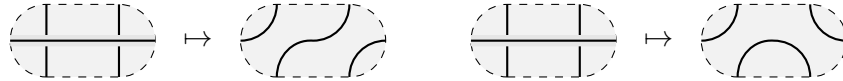
	\begin{center}
		\begin{tikzpicture}[scale=.8]
			\begin{scope}[xshift=-1.5in]
				\begin{scope}[xshift=-.75in,scale=.15]
					\draw[fill=black!5,dashed] (-4,-4)  to (4,-4) arc (-90:90:4) to (-4,4) arc (90:270:4);
					\draw[thick] (-4,4) to (-4,-4);
					\draw[thick] (4,4) to (4,-4);
					\draw[outline] (-8,0) to (8,0);
					\draw[thick] (-8,0) to (8,0);
				\end{scope}
				\node (=) at (0,0) {$\mapsto$};
				\begin{scope}[xshift=.75in,scale=.15]
					\draw[fill=black!5,dashed] (-4,-4)  to (4,-4) arc (-90:90:4) to (-4,4) arc (90:270:4);
					\draw[thick] (-8,0) to [out=0,in=270] (-4,4);
					\draw[thick] (-4,-4) to [out=90,in=180] (0,0) to [out=0,in=270] (4,4);
					\draw[thick] (4,-4) to [out=90,in=180] (8,0);
				\end{scope}
			\end{scope}
			
			\begin{scope}[xshift=1.5in]
				\begin{scope}[xshift=-.75in,scale=.15]
					\draw[fill=black!5,dashed] (-4,-4)  to (4,-4) arc (-90:90:4) to (-4,4) arc (90:270:4);
					\draw[thick] (-4,4) to (-4,-4);
					\draw[thick] (4,4) to (4,-4);
					\draw[outline] (-8,0) to (8,0);
					\draw[thick] (-8,0) to (8,0);
				\end{scope}
				\node (=) at (0,0) {$\mapsto$};
				\begin{scope}[xshift=.75in,scale=.15]
					\draw[fill=black!5,dashed] (-4,-4)  to (4,-4) arc (-90:90:4) to (-4,4) arc (90:270:4);
					\draw[thick] (-8,0) to [out=0,in=270] (-4,4);
					\draw[thick] (-4,-4) to [out=90,in=180] (0,0) to [out=0,in=90] (4,-4);
					\draw[thick] (4,4) to [out=270,in=180] (8,0);
				\end{scope}
			\end{scope}
		\end{tikzpicture}
	\end{center}
	In this local picture, the first case is homotopic to a multicurve with crossing $x$ once, and the second is homotopic to a multicurve which does not cross $x$.
	
	Similarly, between a crossing in $x \cdot Y$ and an end of $x$ at a boundary, there is one local possibility for $Z$, up to reflection across $x$.
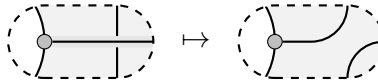
	\begin{center}
		\begin{tikzpicture}[scale=.8]
			\begin{scope}[xshift=-1.5in]
				\begin{scope}[xshift=-.75in,scale=.15]
					\begin{scope}
						\clip (-4,-4)  to (4,-4) arc (-90:90:4) to (-4,4) arc (90:270:4);
						\draw[thick, fill=black!5] (-12,-8) to [out=0,in=270] (-4,0) to [out=90,in=0] (-12,8) to [line to] (12,8) to (12,-8) to (-12,-8);
						\node[dot] (1) at (-4,0) {};
						\draw[thick] (4,4) to (4,-4);
						\draw[outline] (1) to (8,0);
						\draw[thick] (1) to (8,0);
					\end{scope}
					\draw[thick, dashed] (-4,-4)  to (4,-4) arc (-90:90:4) to (-4,4) arc (90:270:4);
				\end{scope}
				\node (=) at (0,0) {$\mapsto$};
				\begin{scope}[xshift=.75in,scale=.15]
					\begin{scope}
						\clip (-4,-4)  to (4,-4) arc (-90:90:4) to (-4,4) arc (90:270:4);
						\draw[thick, fill=black!5] (-12,-8) to [out=0,in=270] (-4,0) to [out=90,in=0] (-12,8) to [line to] (12,8) to (12,-8) to (-12,-8);
						\node[dot] (1) at (-4,0) {};
						\draw[thick] (4,-4) to [out=90,in=180] (8,0);
						\draw[thick] (1) to (0,0) to [out=0,in=270] (4,4);
					\end{scope}
					\draw[thick, dashed] (-4,-4)  to (4,-4) arc (-90:90:4) to (-4,4) arc (90:270:4);
				\end{scope}
			\end{scope}
		\end{tikzpicture}
	\end{center}
	This local picture is homotopic to one which does not cross $x$.
	
	Between a crossing in $x \cdot Y$ and an end of $x$ at a puncture, there is again one local possibility for $Z$, up to reflection across $x$.
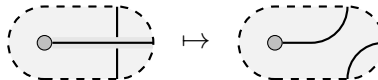
	\begin{center}
		\begin{tikzpicture}[scale=.8]
			\begin{scope}[xshift=-1.5in]
				\begin{scope}[xshift=-.75in,scale=.15]
					\begin{scope}
						\clip (-4,-4)  to (4,-4) arc (-90:90:4) to (-4,4) arc (90:270:4);
						\draw[fill=black!5,dashed] (-4,-4)  to (4,-4) arc (-90:90:4) to (-4,4) arc (90:270:4);
						\node[dot] (1) at (-4,0) {};
						\draw[thick] (4,4) to (4,-4);
						\draw[outline] (1) to (8,0);
						\draw[thick] (1) to (8,0);
					\end{scope}
					\draw[thick, dashed] (-4,-4)  to (4,-4) arc (-90:90:4) to (-4,4) arc (90:270:4);
				\end{scope}
				\node (=) at (0,0) {$\mapsto$};
				\begin{scope}[xshift=.75in,scale=.15]
					\begin{scope}
						\clip (-4,-4)  to (4,-4) arc (-90:90:4) to (-4,4) arc (90:270:4);
						\draw[fill=black!5,dashed] (-4,-4)  to (4,-4) arc (-90:90:4) to (-4,4) arc (90:270:4);
						\node[dot] (1) at (-4,0) {};
						\draw[thick] (4,-4) to [out=90,in=180] (8,0);
						\draw[thick] (1) to (0,0) to [out=0,in=270] (4,4);
					\end{scope}
					\draw[thick, dashed] (-4,-4)  to (4,-4) arc (-90:90:4) to (-4,4) arc (90:270:4);
				\end{scope}
			\end{scope}
		\end{tikzpicture}
	\end{center}
	This local picture is homotopic to one which does not cross $x$.
	
	Finally, we consider the case where $x \cdot Y$ has incompatible taggings at an end of $x$ at a puncture, so $D>0$, which we will break into multiple cases. 
	
	\begin{enumerate}
		\item Suppose that $x$ and $Y$ both have a single end incident to the puncture and they are tagged differently. Again, there is one local possibility for $Z$, up to reflection across $x$.
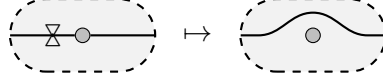
		\begin{center}
			\begin{tikzpicture}[scale=.8]
				\begin{scope}[xshift=-1.5in]
					\begin{scope}[xshift=-.75in,scale=.15]
						\begin{scope}
							\clip (-4,-4)  to (4,-4) arc (-90:90:4) to (-4,4) arc (90:270:4);
							\draw[fill=black!5,dashed] (-4,-4)  to (4,-4) arc (-90:90:4) to (-4,4) arc (90:270:4);
							\node[dot] (1) at (0,0) {};
							\draw[thick] (-8,0) to node[pos=0.67,sloped,rotate=90]{$\bowtie$} (1) {};
							\draw[thick] (1) to (8,0);
						\end{scope}
						\draw[thick, dashed] (-4,-4)  to (4,-4) arc (-90:90:4) to (-4,4) arc (90:270:4);
					\end{scope}
					\node (=) at (0,0) {$\mapsto$};
					\begin{scope}[xshift=.75in,scale=.15]
						\begin{scope}
							\clip (-4,-4)  to (4,-4) arc (-90:90:4) to (-4,4) arc (90:270:4);
							\draw[fill=black!5,dashed] (-4,-4)  to (4,-4) arc (-90:90:4) to (-4,4) arc (90:270:4);
							\node[dot] (1) at (0,0) {};
							\draw[thick] (-8,0) to (-6,0) to [out=0, in=180] (0,2.5) to [out=0, in=180] (6,0) to (8,0);
						\end{scope}
						\draw[thick, dashed] (-4,-4)  to (4,-4) arc (-90:90:4) to (-4,4) arc (90:270:4);
					\end{scope}
				\end{scope}
			\end{tikzpicture}
		\end{center}
		This local picture is homotopic to one which does not cross $x$.
		
		\item Suppose that one end of $x$ is at the puncture, but $Y$ has at least $2$ ends at the puncture tagged differently than $x$. Here, the local contribution to $\mu(x,Y)$ is the value of $D\geq 2$. When $D=2$, there are two local possibilities, up to reflection across $x$.
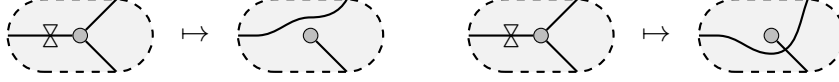
		\begin{center}
			\begin{tikzpicture}[scale=.8]
				\begin{scope}[xshift=-1.5in]
					\begin{scope}[xshift=-.75in,scale=.15]
						\begin{scope}
							\clip (-4,-4)  to (4,-4) arc (-90:90:4) to (-4,4) arc (90:270:4);
							\draw[fill=black!5,dashed] (-4,-4)  to (4,-4) arc (-90:90:4) to (-4,4) arc (90:270:4);
							\node[dot] (1) at (0,0) {};
							\draw[thick] (-8,0) to node[pos=0.67,sloped,rotate=90]{$\bowtie$} (1) {};
							\draw[thick] (1) to (4,4);
							\draw[thick] (1) to (4,-4);
						\end{scope}
						\draw[thick, dashed] (-4,-4)  to (4,-4) arc (-90:90:4) to (-4,4) arc (90:270:4);
					\end{scope}
					\node (=) at (0,0) {$\mapsto$};
					\begin{scope}[xshift=.75in,scale=.15]
						\begin{scope}
							\clip (-4,-4)  to (4,-4) arc (-90:90:4) to (-4,4) arc (90:270:4);
							\draw[fill=black!5,dashed] (-4,-4)  to (4,-4) arc (-90:90:4) to (-4,4) arc (90:270:4);
							\node[dot] (1) at (0,0) {};
							\draw[thick] (-8,0) to (-6,0) to [out=0, in=180] (0,2) to [out=0,in=240] (4,4);
							\draw[thick] (1) to (4,-4);
						\end{scope}
						\draw[thick, dashed] (-4,-4)  to (4,-4) arc (-90:90:4) to (-4,4) arc (90:270:4);
					\end{scope}
				\end{scope}
				
				\begin{scope}[xshift=1.5in]
					\begin{scope}[xshift=-.75in,scale=.15]
						\begin{scope}
							\clip (-4,-4)  to (4,-4) arc (-90:90:4) to (-4,4) arc (90:270:4);
							\draw[fill=black!5,dashed] (-4,-4)  to (4,-4) arc (-90:90:4) to (-4,4) arc (90:270:4);
							\node[dot] (1) at (0,0) {};
							\draw[thick] (-8,0) to node[pos=0.67,sloped,rotate=90]{$\bowtie$} (1) {};
							\draw[thick] (1) to (4,4);
							\draw[thick] (1) to (4,-4);
						\end{scope}
						\draw[thick, dashed] (-4,-4)  to (4,-4) arc (-90:90:4) to (-4,4) arc (90:270:4);
					\end{scope}
					\node (=) at (0,0) {$\mapsto$};
					\begin{scope}[xshift=.75in,scale=.15]
						\begin{scope}
							\clip (-4,-4)  to (4,-4) arc (-90:90:4) to (-4,4) arc (90:270:4);
							\draw[fill=black!5,dashed] (-4,-4)  to (4,-4) arc (-90:90:4) to (-4,4) arc (90:270:4);
							\node[dot] (1) at (0,0) {};
							\draw[thick] (-8,0) to (-6,0) to [out=0, in=180] (0,-2) to [out=0, in=250] (4,4);
							\draw[thick] (1) to (4,-4);
						\end{scope}
						\draw[thick, dashed] (-4,-4)  to (4,-4) arc (-90:90:4) to (-4,4) arc (90:270:4);
					\end{scope}
				\end{scope}
			\end{tikzpicture}
		\end{center}
		In this local picture, both cases are homotopic to a multicurve intersecting $x$ once. For $D>2$, the same patterns (with additional arcs at the puncture tagged differently than $x$) hold, giving us multicurves that have strictly fewer intersections with $x$. 
		
		\item Suppose that $x$ has both ends at the puncture, and $Y$ has a single end at the puncture which is tagged differently than $x$. Here, the local contribution to $\mu(x,Y)$ is the value of $D=2$. 
		
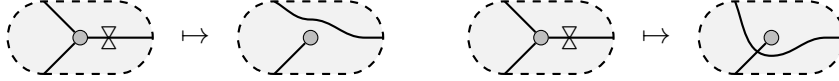
\begin{center}
			\begin{tikzpicture}[scale=.8]
				\begin{scope}[xshift=-1.5in]
					\begin{scope}[xshift=-.75in,scale=.15]
						\begin{scope}
							\clip (-4,-4)  to (4,-4) arc (-90:90:4) to (-4,4) arc (90:270:4);
							\draw[fill=black!5,dashed] (-4,-4)  to (4,-4) arc (-90:90:4) to (-4,4) arc (90:270:4);
							\node[dot] (1) at (0,0) {};
							\draw[thick] (8,0) to node[pos=0.67,sloped,rotate=90]{$\bowtie$} (1) {};
							\draw[thick] (1) to (-4,4);
							\draw[thick] (1) to (-4,-4);
						\end{scope}
						\draw[thick, dashed] (-4,-4)  to (4,-4) arc (-90:90:4) to (-4,4) arc (90:270:4);
					\end{scope}
					\node (=) at (0,0) {$\mapsto$};
					\begin{scope}[xshift=.75in,scale=.15]
						\begin{scope}
							\clip (-4,-4)  to (4,-4) arc (-90:90:4) to (-4,4) arc (90:270:4);
							\draw[fill=black!5,dashed] (-4,-4)  to (4,-4) arc (-90:90:4) to (-4,4) arc (90:270:4);
							\node[dot] (1) at (0,0) {};
							\draw[thick] (8,0) to (6,0) to [out=180, in=0] (0,2) to [out=180,in=330] (-4,4);
							\draw[thick] (1) to (-4,-4);
						\end{scope}
						\draw[thick, dashed] (-4,-4)  to (4,-4) arc (-90:90:4) to (-4,4) arc (90:270:4);
					\end{scope}
				\end{scope}
				
				\begin{scope}[xshift=1.5in]
					\begin{scope}[xshift=-.75in,scale=.15]
						\begin{scope}
							\clip (-4,-4)  to (4,-4) arc (-90:90:4) to (-4,4) arc (90:270:4);
							\draw[fill=black!5,dashed] (-4,-4)  to (4,-4) arc (-90:90:4) to (-4,4) arc (90:270:4);
							\node[dot] (1) at (0,0) {};
							\draw[thick] (8,0) to node[pos=0.67,sloped,rotate=90]{$\bowtie$} (1) {};
							\draw[thick] (1) to (-4,4);
							\draw[thick] (1) to (-4,-4);
						\end{scope}
						\draw[thick, dashed] (-4,-4)  to (4,-4) arc (-90:90:4) to (-4,4) arc (90:270:4);
					\end{scope}
					\node (=) at (0,0) {$\mapsto$};
					\begin{scope}[xshift=.75in,scale=.15]
						\begin{scope}
							\clip (-4,-4)  to (4,-4) arc (-90:90:4) to (-4,4) arc (90:270:4);
							\draw[fill=black!5,dashed] (-4,-4)  to (4,-4) arc (-90:90:4) to (-4,4) arc (90:270:4);
							\node[dot] (1) at (0,0) {};
							\draw[thick] (8,0) to (6,0) to [out=180, in=0] (0,-2) to [out=180, in=290] (-4,4);
							\draw[thick] (1) to (-4,-4);
						\end{scope}
						\draw[thick, dashed] (-4,-4)  to (4,-4) arc (-90:90:4) to (-4,4) arc (90:270:4);
					\end{scope}
				\end{scope}
			\end{tikzpicture}
		\end{center}
		In this local picture, the first case is homotopic to a multicurve that does not intersect $x$, and the second case is homotopic to a multicurve that intersects $x$ once.
		
		\item Suppose that $x$ has both ends at the puncture, and $Y$ has at least 2 ends at the puncture tagged differently than $x$. Here, the local contribution to $\mu(x,Y)$ is the value of $D\geq 4$. When $D=4$, there are six possibilities, up to reflection and rotation. 
		
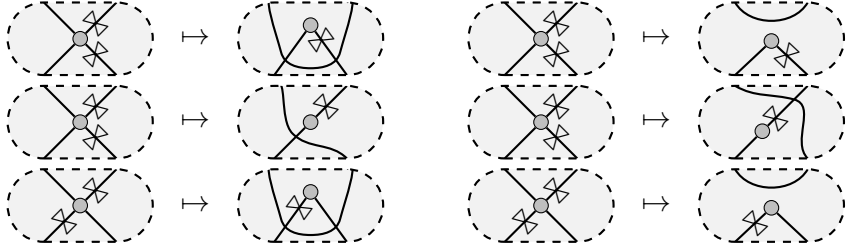
\begin{center}
			\begin{tikzpicture}[scale=.8]
				\begin{scope}[xshift=-1.5in]
					\begin{scope}[xshift=-.75in,scale=.15]
						\begin{scope}
							\clip (-4,-4)  to (4,-4) arc (-90:90:4) to (-4,4) arc (90:270:4);
							\draw[fill=black!5,dashed] (-4,-4)  to (4,-4) arc (-90:90:4) to (-4,4) arc (90:270:4);
							\node[dot] (1) at (0,0) {};
							\draw[thick] (4,4) to node[pos=0.67,sloped,rotate=90]{$\bowtie$} (1) {};
							\draw[thick] (4,-4) to node[pos=0.67,sloped,rotate=90]{$\bowtie$} (1) {};
							\draw[thick] (1) to (-4,4);
							\draw[thick] (1) to (-4,-4);
						\end{scope}
						\draw[thick, dashed] (-4,-4)  to (4,-4) arc (-90:90:4) to (-4,4) arc (90:270:4);
					\end{scope}
					\node (=) at (0,0) {$\mapsto$};
					\begin{scope}[xshift=.75in,scale=.15]
						\begin{scope}
							\clip (-4,-4)  to (4,-4) arc (-90:90:4) to (-4,4) arc (90:270:4);
							\draw[fill=black!5,dashed] (-4,-4)  to (4,-4) arc (-90:90:4) to (-4,4) arc (90:270:4);
							\node[dot] (1) at (0,1.5) {};
							\draw[thick] (4,4) to [in=75] (3.25,-1.5) to [out=255,in=0] (0,-3.25) to [out=180,in=285] (-3.25,-1.5) to [out=105] (-4,4);
							\draw[thick] (4,-4) to node[pos=0.8,sloped,rotate=90]{$\bowtie$} (1) {};
							\draw[thick] (1) to (-4,-4);
						\end{scope}
						\draw[thick, dashed] (-4,-4)  to (4,-4) arc (-90:90:4) to (-4,4) arc (90:270:4);
					\end{scope}
				\end{scope}
				
				\begin{scope}[xshift=1.5in]
					\begin{scope}[xshift=-.75in,scale=.15]
						\begin{scope}
							\clip (-4,-4)  to (4,-4) arc (-90:90:4) to (-4,4) arc (90:270:4);
							\draw[fill=black!5,dashed] (-4,-4)  to (4,-4) arc (-90:90:4) to (-4,4) arc (90:270:4);
							\node[dot] (1) at (0,0) {};
							\draw[thick] (4,4) to node[pos=0.67,sloped,rotate=90]{$\bowtie$} (1) {};
							\draw[thick] (4,-4) to node[pos=0.67,sloped,rotate=90]{$\bowtie$} (1) {};
							\draw[thick] (1) to (-4,4);
							\draw[thick] (1) to (-4,-4);
						\end{scope}
						\draw[thick, dashed] (-4,-4)  to (4,-4) arc (-90:90:4) to (-4,4) arc (90:270:4);
					\end{scope}
					\node (=) at (0,0) {$\mapsto$};
					\begin{scope}[xshift=.75in,scale=.15]
						\begin{scope}
							\clip (-4,-4)  to (4,-4) arc (-90:90:4) to (-4,4) arc (90:270:4);
							\draw[fill=black!5,dashed] (-4,-4)  to (4,-4) arc (-90:90:4) to (-4,4) arc (90:270:4);
							\node[dot] (1) at (0,-0.25) {};
							\draw[thick] (4,4) to [out=240,in=300] (-4,4);
							\draw[thick] (1) to (-4,-4);
							\draw[thick] (4,-4) to node[pos=0.67,sloped,rotate=90]{$\bowtie$} (1) {};
						\end{scope}
						\draw[thick, dashed] (-4,-4)  to (4,-4) arc (-90:90:4) to (-4,4) arc (90:270:4);
					\end{scope}
				\end{scope}
			\end{tikzpicture}
		\end{center}
		
		\begin{center}
			\begin{tikzpicture}[scale=.8]
				\begin{scope}[xshift=-1.5in]
					\begin{scope}[xshift=-.75in,scale=.15]
						\begin{scope}
							\clip (-4,-4)  to (4,-4) arc (-90:90:4) to (-4,4) arc (90:270:4);
							\draw[fill=black!5,dashed] (-4,-4)  to (4,-4) arc (-90:90:4) to (-4,4) arc (90:270:4);
							\node[dot] (1) at (0,0) {};
							\draw[thick] (4,4) to node[pos=0.67,sloped,rotate=90]{$\bowtie$} (1) {};
							\draw[thick] (4,-4) to node[pos=0.67,sloped,rotate=90]{$\bowtie$} (1) {};
							\draw[thick] (1) to (-4,4);
							\draw[thick] (1) to (-4,-4);
						\end{scope}
						\draw[thick, dashed] (-4,-4)  to (4,-4) arc (-90:90:4) to (-4,4) arc (90:270:4);
					\end{scope}
					\node (=) at (0,0) {$\mapsto$};
					\begin{scope}[xshift=.75in,scale=.15]
						\begin{scope}
							\clip (-4,-4)  to (4,-4) arc (-90:90:4) to (-4,4) arc (90:270:4);
							\draw[fill=black!5,dashed] (-4,-4)  to (4,-4) arc (-90:90:4) to (-4,4) arc (90:270:4);
							\node[dot] (1) at (0,0) {};
							\draw[thick] (4,4) to node[pos=0.67,sloped,rotate=90]{$\bowtie$} (1) {};
							\draw[thick] (1) to (-4,-4);
							\draw[thick] (-4,4) to [in=135] (-2,-1.25) to [out=315] (4,-4);
						\end{scope}
						\draw[thick, dashed] (-4,-4)  to (4,-4) arc (-90:90:4) to (-4,4) arc (90:270:4);
					\end{scope}
				\end{scope}
				
				\begin{scope}[xshift=1.5in]
					\begin{scope}[xshift=-.75in,scale=.15]
						\begin{scope}
							\clip (-4,-4)  to (4,-4) arc (-90:90:4) to (-4,4) arc (90:270:4);
							\draw[fill=black!5,dashed] (-4,-4)  to (4,-4) arc (-90:90:4) to (-4,4) arc (90:270:4);
							\node[dot] (1) at (0,0) {};
							\draw[thick] (4,4) to node[pos=0.67,sloped,rotate=90]{$\bowtie$} (1) {};
							\draw[thick] (4,-4) to node[pos=0.67,sloped,rotate=90]{$\bowtie$} (1) {};
							\draw[thick] (1) to (-4,4);
							\draw[thick] (1) to (-4,-4);
						\end{scope}
						\draw[thick, dashed] (-4,-4)  to (4,-4) arc (-90:90:4) to (-4,4) arc (90:270:4);
					\end{scope}
					\node (=) at (0,0) {$\mapsto$};
					\begin{scope}[xshift=.75in,scale=.15]
						\begin{scope}
							\clip (-4,-4)  to (4,-4) arc (-90:90:4) to (-4,4) arc (90:270:4);
							\draw[fill=black!5,dashed] (-4,-4)  to (4,-4) arc (-90:90:4) to (-4,4) arc (90:270:4);
							\node[dot] (1) at (-1,-1) {};
							\draw[thick] (4,4) to node[pos=0.8,sloped,rotate=90]{$\bowtie$} (1) {};
							\draw[thick] (1) to (-4,-4);
							\draw[thick] (-4,4) to [out=330,in=165] (2.5,2.5) to [out=345] (4,-4);
						\end{scope}
						\draw[thick, dashed] (-4,-4)  to (4,-4) arc (-90:90:4) to (-4,4) arc (90:270:4);
					\end{scope}
				\end{scope}
			\end{tikzpicture}
		\end{center}
		
		\begin{center}
			\begin{tikzpicture}[scale=.8]
				\begin{scope}[xshift=-1.5in]
					\begin{scope}[xshift=-.75in,scale=.15]
						\begin{scope}
							\clip (-4,-4)  to (4,-4) arc (-90:90:4) to (-4,4) arc (90:270:4);
							\draw[fill=black!5,dashed] (-4,-4)  to (4,-4) arc (-90:90:4) to (-4,4) arc (90:270:4);
							\node[dot] (1) at (0,0) {};
							\draw[thick] (4,4) to node[pos=0.67,sloped,rotate=90]{$\bowtie$} (1) {};
							\draw[thick] (-4,-4) to node[pos=0.67,sloped,rotate=90]{$\bowtie$} (1) {};
							\draw[thick] (1) to (-4,4);
							\draw[thick] (1) to (4,-4);
						\end{scope}
						\draw[thick, dashed] (-4,-4)  to (4,-4) arc (-90:90:4) to (-4,4) arc (90:270:4);
					\end{scope}
					\node (=) at (0,0) {$\mapsto$};
					\begin{scope}[xshift=.75in,scale=.15]
						\begin{scope}
							\clip (-4,-4)  to (4,-4) arc (-90:90:4) to (-4,4) arc (90:270:4);
							\draw[fill=black!5,dashed] (-4,-4)  to (4,-4) arc (-90:90:4) to (-4,4) arc (90:270:4);
							\node[dot] (1) at (0,1.5) {};
							\draw[thick] (4,4) to [in=75] (3.25,-1.5) to [out=255,in=0] (0,-3.25) to [out=180,in=285] (-3.25,-1.5) to [out=105] (-4,4);
							\draw[thick] (4,-4) to (1) {};
							\draw[thick] (-4,-4) to node[pos=0.8,sloped,rotate=90]{$\bowtie$} (1) {};
						\end{scope}
						\draw[thick, dashed] (-4,-4)  to (4,-4) arc (-90:90:4) to (-4,4) arc (90:270:4);
					\end{scope}
				\end{scope}
				
				\begin{scope}[xshift=1.5in]
					\begin{scope}[xshift=-.75in,scale=.15]
						\begin{scope}
							\clip (-4,-4)  to (4,-4) arc (-90:90:4) to (-4,4) arc (90:270:4);
							\draw[fill=black!5,dashed] (-4,-4)  to (4,-4) arc (-90:90:4) to (-4,4) arc (90:270:4);
							\node[dot] (1) at (0,0) {};
							\draw[thick] (4,4) to node[pos=0.67,sloped,rotate=90]{$\bowtie$} (1) {};
							\draw[thick] (-4,-4) to node[pos=0.67,sloped,rotate=90]{$\bowtie$} (1) {};
							\draw[thick] (1) to (-4,4);
							\draw[thick] (1) to (4,-4);
						\end{scope}
						\draw[thick, dashed] (-4,-4)  to (4,-4) arc (-90:90:4) to (-4,4) arc (90:270:4);
					\end{scope}
					\node (=) at (0,0) {$\mapsto$};
					\begin{scope}[xshift=.75in,scale=.15]
						\begin{scope}
							\clip (-4,-4)  to (4,-4) arc (-90:90:4) to (-4,4) arc (90:270:4);
							\draw[fill=black!5,dashed] (-4,-4)  to (4,-4) arc (-90:90:4) to (-4,4) arc (90:270:4);
							\node[dot] (1) at (0,-0.25) {};
							\draw[thick] (4,4) to [out=240,in=300] (-4,4);
							\draw[thick] (4,-4) to (1) {};
							\draw[thick] (-4,-4) to node[pos=0.67,sloped,rotate=90]{$\bowtie$} (1) {};
						\end{scope}
						\draw[thick, dashed] (-4,-4)  to (4,-4) arc (-90:90:4) to (-4,4) arc (90:270:4);
					\end{scope}
				\end{scope}
			\end{tikzpicture}
		\end{center}
		In this local picture, all cases are homotopic to a multicurve that has strictly fewer intersections with $x$. In the left column, we have $A+D=1+2=3 < 4$; in the right column, we have $A+D = 0+2 =2 < 4$. For $D>4$, the same patterns (with additional arcs at the puncture tagged differently than $x$) hold, giving us multicurves that have strictly fewer intersections with $x$. In particular, the value of $D$ will be cut in half, while $A$ will increase by a number less than the number of ends of $Y$ incident to the puncture, so the total number of intersections has decreased.
	\end{enumerate}
	
	Hence, $Z$ is homotopic to a simple multicurve $Z'$, such that $x \cdot Z'$ is transverse and $x\cdot Z'$ has strictly fewer crossings than $x \cdot Y$.  Since the latter already has $\mu(x,Y)$ crossings,
	\[ \mu([x],[Z'])\leq\mu(x,Y)-1.\]
	Because $\Supp([x][Y])\subset I$,
	\[ \mu([x],[x][Y])\leq\mu([x],[Y])-1.\]
	The general form of the lemma follows from this case.	
\end{proof}

\begin{coro}\label{coro: reducing tagged crossings}
	If $x$ is a simple tagged arc in $\SS$, then for all $y \in \TSk (\SS)$, \[ \mu ([x],[x]^{\mu([x],y)}y) = 0. \]
\end{coro}
The proof is identical to that of \cite[Corollary 4.13]{Mul16}. 
\begin{proof}
	By iterating Lemma \ref{lemma: reducing tagged crossings}, if $i \leq \mu([x],[x]^i y)$, then $\mu([x],[x]^i y) \leq \mu([x], y) -i$. In particular, $\mu ([x], [x]^{\mu ([x],y)} y) \leq 0$, so it is zero. 
\end{proof}

By replacing the reference to \cite[Corollary 4.13]{Mul16} in the proof of Proposition \ref{prop: cutting algebras} with a reference to Proposition \ref{coro: reducing tagged crossings}, and replacing $\Sk_1 (\SS)$ with $\TSk (\SS)$, we now have the following results for punctured surfaces.

\begin{prop}\label{prop: cutting algebras punctured}
	Let $\SS$ be a triangulable marked surface, and let $S$ be a compatible collection of non-boundary (tagged) marked arcs in $\SS$. Then the map $\SS\smallsetminus S\rightarrow S$ induces an isomorphism
	\[  \CA(\SS\smallsetminus S)/\langle s'-s'', \forall s\in S\rangle \xrightarrow{\sim} \CA(\SS)[S^{-1}] \]
	Here, $s'$ and $s''$ denote the two preimages of $s$ in $\SS\smallsetminus S$.
\end{prop}    

\begin{prop}\label{prop: cutting subvarieties punctured}
	Let $\SS$ be a triangulable marked surface, and let $S$ be a compatible collection of non-boundary (tagged) marked arcs in $\SS$.
	Then restricting values from (tagged) marked curves in $\SS$ to (tagged) marked curves in $\SS\smallsetminus S$ gives an isomorphism between:
	\begin{enumerate}
		\item The open subvariety of $V(\CA(\SS),\Bbbk)$ consisting of points which are non-zero on $S$.
		\item The closed subvariety of $V(\CA(\SS\smallsetminus S),\Bbbk)$ consisting of points which are {gluable along $S$}.
	\end{enumerate}
\end{prop}

%========================================================

\section{Unpunctured marked surfaces}\label{section: deepmarkedsurfaces}

In this section, we recall the characterization of deep points of cluster algebras of unpunctured marked surfaces with boundary from \cite{BM25}. Many of these results will be referenced in the following examination of punctured surfaces.  
The first step in characterizing deep points of $\CA(\SS)$ is to characterize which marked arcs can simultaneously vanish. 

A \textbf{triangle of arcs} in $\SS$ is a compatible triple of marked arcs $\{x,y,z\}$ which bound a disk in $\SS$. The following lemma shows that, for connected surfaces, a deep point must kill an odd number of arcs in each triangle.

\begin{lemma}\cite[Lemma 7.1]{BM25}\label{lemma: trianglesurf}
Let $\SS$ be a connected, unpunctured, triangulable marked surface, let $p\in V(\CA(\SS),\Bbbk)$, and let $\{x,y,z\}$ be a triangle of arcs in $\SS$.
\begin{enumerate}
	\item If $p$ kills any two of $\{x,y,z\}$, then $p$ also kills the third.
	\item If $p$ is deep, then $p$ kills at least one of $\{x,y,z\}$.\qedhere
\end{enumerate}
\end{lemma}

Let $\Delta_n$ denote the unpunctured disk with $n$ marked points on the boundary.  Then as a result of Lemma \ref{lemma: trianglesurf}, we have the following characterization of deep points in $V(\CA(\Delta_n),\Bbbk)$.

\begin{thm}\cite[Theorem 3.5]{BM25}\label{thm: deeppolygon}
Let $n \geq 3$.
\begin{enumerate}
    \item If $n$ is odd, then the cluster variety $V(\CA(\Delta_n),\Bbbk)$ has no deep points.
    \item If $n$ is even, then a choice of values on the edges of $\Delta_n$
    \[ p:\{\text{edges in $\Delta_n$}\} \rightarrow \Bbbk^\times\]
    extends to a deep point of $V(\CA(\Delta_n),\Bbbk)$ iff
    \begin{equation}\label{eq: alternating product}
    \frac{p(x_{1,2})p(x_{3,4}) \dotsm p(x_{n-1,n})}{p(x_{1,n})p(x_{2,3}) \dotsm p(x_{n-2,n-1})} = (-1)^{\frac{n+2}{2}}
    \end{equation}
    and this extension is unique, with values on an arbitrary diagonal $x_{i,j}$ given by
    \begin{equation}\label{eq: arbitrary diagonal}
        p(x_{i,j}) = \left\{\begin{array}{cc}
        (-1)^{\frac{j-i-1}{2}} \frac{p(x_{i,i+1})p(x_{i+2,i+3})\dotsm p(x_{j-1,j})}{p(x_{i+1,i+2})\dotsm p(x_{j-2,j-1})} & \text{if $i \not\equiv j\bmod{2}$} \\
        0 & \text{if $i \equiv j\bmod{2}$}
        \end{array}\right\}
    \end{equation}   
    As a consequence, the deep locus of $V(\CA(\Delta_n),\Bbbk)$ freely determined by its values on all but one edge, and is therefore isomorphic to $(\Bbbk^\times)^{n-1}$.\qedhere
\end{enumerate}
\end{thm}

To consider general unpunctured marked surfaces with boundary, we want to reduce to the polygonal case. To do so, we identify collections of arcs in $\SS$ whose cutting is a polygon. 

\begin{prop}\cite[Proposition 7.2]{BM25}\label{prop: polydiss}
Let $\SS$ be a connected, triangulable marked surface. Then there exists a compatible collection of marked arcs $D$ such that the cutting $\SS\smallsetminus D$ (as defined in Section \ref{section: cutting}) is a polygon. We call such a set of arcs a \textbf{polygonal dissection} of $\SS$. 
\end{prop}

\begin{prop}\cite[Proposition 7.3]{BM25}\label{prop: polydiss arcs}
Let $\SS$ be a connected, triangulable marked surface with genus $g$, $b$-many boundary components, $q$-many punctures, and $m$-many boundary marked points. Then every polygonal dissection of $\SS$ cuts along 
\[
d(\SS) \coloneqq 2g + b + q - 1
\]
many arcs and produces a polygon with 
\[
\delta(\SS) \coloneqq 4g + 2b + 2q + m - 2 
\]
sides.
\end{prop}

\begin{lemma}\cite[Lemma 7.2]{BM25}\label{lemma: polydissavoid}
Let $\SS$ be a connected, unpunctured, triangulable marked surface, and let $\cV$ be a set of simple, non-boundary marked arcs which does not contain exactly two of the three arcs in any triangle in $\SS$. Then there exists a polygonal dissection $D$ of $\SS$ consisting of arcs not in $\cV$.
\end{lemma}

The special case of $\cV\coloneqq\{a \text{ such that }p(a)=0\}$ is the following.

\begin{coro}\cite[Corollary 7.4]{BM25}\label{coro: deeppolydiss}
Let $\SS$ be a connected, unpunctured, triangulable marked surface. For every point $p\in V(\CA(\SS),\Bbbk)$, there is a polygonal dissection $D$ of $\SS$ on which $p$ is non-zero.
\end{coro}

By Proposition \ref{prop: cutting subvarieties}, each polygonal dissection $D$ determines an open embedding
\begin{equation}\label{eq: polydissinclusion}
\{ \text{points in $V(\CA(\SS\smallsetminus D),\Bbbk)$ gluable along $D$}\}\hookrightarrow V(\CA(\SS),\Bbbk).
\end{equation}
Corollary \ref{coro: deeppolydiss} is equivalent to the fact that these open charts collectively cover the cluster variety.

Every point in the cluster variety $V(\CA(\SS),\Bbbk)$ is then non-zero on some polygonal dissection, and is therefore contained in an open neighborhood given by gluing the cluster variety of a polygon (Equation \eqref{eq: polydissinclusion}). Since Theorem \ref{thm: deeppolygon} gives a characterization of the deep points in the latter, we need to understand how cutting and gluing affect deepness.
The immediate result is the following.

\begin{prop}\cite[Proposition 7.5]{BM25}\label{prop: deep cutting inclusion}
The cutting isomorphism restricts to an inclusion on deep points; that is,
\[ 
\begin{Bmatrix} \text{deep points in } V(\CA(\SS),\Bbbk) \\ \text{non-zero on }S \end{Bmatrix} 
\hookrightarrow
\begin{Bmatrix}\text{deep points in } V(\CA(\SS\smallsetminus S),\Bbbk) \\ \text{gluable along } S \end{Bmatrix}
\]
\end{prop}

Unfortunately, this restriction is not always a bijection. If $\SS$ is a disjoint union of two components $\SS_1$ and $\SS_2$, then a point $p$ of $V(\CA(\SS),\Bbbk)$ is equivalent to a pair of points $p_1$ and $p_2$ of $V(\CA(\SS_1),\Bbbk)$ and $V(\CA(\SS_2),\Bbbk)$, respectively; that is,
\[ V(\CA(\SS), \Bbbk) \simeq V(\CA(\SS_1),\Bbbk) \times V(\CA(\SS_2),\Bbbk) \]
Since a triangulation of $\SS$ is a union of triangulations of $\SS_1$ and $\SS_2$, we see that a point $p$ of $V(\CA(\SS),\Bbbk)$ is deep if and only if \emph{either} of the points $p_1$ or $p_2$ are deep. 

We introduce a definition which behaves better for disconnected surfaces and which won't be used beyond the following lemma. Let us say a point $p$ in $V(\CA(\SS),\Bbbk)$ is \textbf{deep-on-components} if the restriction of $p$ to each connected component of $\SS$ is deep; that is, it kills a marked arc in each triangulation of each connected component. Every deep-on-components point is deep, but the converse fails.

Many prior results extend to disconnected surfaces by replacing \emph{deep} with \emph{deep-on-components}. 
For example, applying Lemma \ref{lemma: trianglesurf} to each connected component of $\SS$, we see that a point $p$ of $V(\CA(\SS),\Bbbk)$ is deep-on-components if and only if it kills an odd number of arcs in any triangle in $\SS$.

\begin{lemma}\cite[Lemma 7.3]{BM25}\label{lemma: deepcut}
Let $\SS$ be an unpunctured, triangulable marked surface. For any compatible collection of non-boundary arcs $S$ in $\SS$, the cutting isomorphism \eqref{eq: cutting} restricts to a bijection between:
\begin{enumerate}
\item The deep-on-components points of $V(\CA(\SS),\Bbbk)$ which are non-zero on $S$.
\item The deep-on-components points of $V(\CA(\SS\smallsetminus S),\Bbbk)$ which are gluable along $S$.\qedhere
\end{enumerate}
\end{lemma}

In what follows, we only need a specialization of Lemma \ref{lemma: deepcut} to the case of polygonal dissections.

\begin{coro}\cite[Corollary 7.7]{BM25}\label{coro: deepcutpoly}
Let $\SS$ be a connected, unpunctured, triangulable marked surface. For any polygonal dissection $D$ in $\SS$, the cutting isomorphism \eqref{eq: cutting} restricts to a bijection between:
\begin{enumerate}
\item The deep points in $V(\CA(\SS),\Bbbk)$ which are non-zero on $D$.
\item The deep points in $V(\CA(\SS\smallsetminus D),\Bbbk)$ which are gluable along $D$.\qedhere
\end{enumerate}
\end{coro}

While Corollary \ref{coro: deepcutpoly} allows us to construct deep points of $\CA(\SS)$ by gluing deep points of polygons, a single deep point of $\CA(\SS)$ may be non-zero on many polygonal dissections, and can therefore be constructed via many different gluings.

In this section, we use relative cohomology classes in $H^1(\SS,\MM;\mathbb{Z}_2)$ to characterize when a deep point is non-zero on a polygonal dissection, and when two polygonal dissections determine the same set of deep points.
Note that every marked arc in $\SS$ is a 1-cycle relative to the marked points $\MM$, and so elements of $H^1(\Sigma,\MM;\mathbb{Z}_2)$ have a well-defined value on each marked arc in $\SS$.

\begin{prop}\cite[Proposition 7.8]{BM25}
Let $\SS$ be a connected, unpunctured, triangulable marked surface with an even number of marked points.
\begin{enumerate}
\item Given a deep point $p$ of $V(\CA(\SS),\Bbbk)$, there is a unique cohomology class\\
$
\chi_p\in H^1(\SS,\MM;\mathbb{Z}_2)
$
such that the value of $\chi_p$ on each simple marked arc $a$ in $\SS$ is
\begin{equation}\label{eq: cohomologyclass}
	\chi_p(a) \coloneqq
	\begin{cases}
		0 & \text{if $p(a)=0$} \\
		1 & \text{if $p(a)\neq0$}
	\end{cases}
\end{equation}
\item Given a polygonal dissection $D$ of $\Sigma$, there is a unique cohomology class 
$
\chi_D\in H^1(\SS,\MM;\mathbb{Z}_2)
$
such that $\chi_D$ is non-zero on each marked arc in $D$ and each boundary arc of $\SS$.
\item A deep point $p$ is non-zero on a polygonal dissection $D$ if and only if $\chi_p=\chi_D$.\qedhere
\end{enumerate}
\end{prop}

As a consequence, two polygonal dissections can be used to construct the same deep points if and only if they have the same cohomology class. This can be used to define an equivalence relation.

\begin{defn}\label{defn: congruent}
Let $\SS$ be as in \cite[Proposition 7.8]{BM25}.
Then two polygonal dissections of $\SS$ are \textbf{congruent} if their relative cohomology classes are equal.
\end{defn}

\begin{prop}\cite[Proposition 7.10]{BM25}\label{prop: polydisscount}
Let $\SS$ be a connected, unpunctured, triangulable marked surface with an even number of marked points. Then the function $D\mapsto \chi_D$ defines a bijection between
\begin{enumerate}
\item polygonal dissections $D$ of $\SS$ up to congruence, and
\item elements $\chi$ in $H^1(\SS,\MM;\mathbb{Z}_2)$ with value $1$ on every boundary marked arc in $\SS$.
\end{enumerate}
The size of each set is $2^{2g+b-1}$, where $g$ is the genus and $b$ is the number of boundary components.
\end{prop}

These results can now be combined to characterize the deep locus of $V(\CA(\SS),\Bbbk)$.

\begin{thm}\cite[Theorem 7.12]{BM25}
Let $\SS$ be a connected, unpunctured, triangulable marked surface with genus $g$, $b$-many boundary components, and $m$-many marked points.
\begin{enumerate}
\item If $m$ is odd, then $V(\CA(\SS),\Bbbk)$ has no deep points.
\item If $m$ is even, then each deep point of $V(\CA(\SS),\Bbbk)$ is non-zero on some polygonal dissection. 
\begin{enumerate}
	\item 
	Given a polygonal dissection $D$ of $\SS$, a function
	\[ p:\{\text{boundary arcs in $\SS$}\} \cup D \rightarrow \Bbbk^\times\]
	extends to a (necessarily unique) deep point of $V(\CA(\SS),\Bbbk)$ which is non-zero on $D$ if and only if the alternating product of the values of $p$ around the boundary of $\SS\smallsetminus D$ is $(-1)^{b+\frac{m}{2}}$.
	\item Two polygonal dissections of $\SS$ parameterize the same deep points if and only if they are congruent (see Definition \ref{defn: congruent}); otherwise they parameterize disjoint sets of deep points.
\end{enumerate}
Therefore, the deep locus of $V(\CA(\SS),\Bbbk)$ consists of $2^{2g+b-1}$-many disjoint copies of $(\Bbbk^\times)^{2g+b+m-2}$.\qedhere
\end{enumerate}
\end{thm}

%============================

\section{Once-punctured marked disks}\label{section: 1 punctured disks}

Recall that a \textbf{puncture} in a surface $\SS$ is a distinguished (or marked) point in the interior of $\SS$. A (tagged) \textbf{radius} is a simple arc with one end at the boundary and the other end at a puncture (tagged plain or notched). In a disk $\SS$ with one puncture, there is a triangulation consisting entirely of boundary arcs and plain radii, so any deep point $p \in V(\CA(\SS),\Bbbk)$ must kill at least one plain radius; likewise, there is a triangulation with all notched radii, so $p$ must kill at least one notched radius.

There are also deep points that kill all radii of one or both taggings. We will refer to a deep point that kills all radii of both taggings as a \textbf{totally radial deep point}. If a deep point kills all plain radii (but not all notched radii), we will refer to it as a \textbf{plain deep point}; similarly, if a deep point kills all notched radii (but not all plain radii), we will refer to it as a \textbf{notched deep point}. 

Unfortunately, many of the tools used in the classification of deep points of cluster algebras of unpunctured surfaces will not directly apply to cluster algebras of punctured surfaces. In this section, we will consider disks with a single puncture with the goal of later generalizing results. We begin with some low-complexity examples.

%============================

\subsection{Minimal examples}\label{section: minimal once-punctured disks}

\begin{nota}
	Throughout this section and the next, we will denote the cluster variable associated to the plain tagging of a radius as $x_j$ and the variable associated to the notched tagging of that radius as $x_j^\vee$.
\end{nota}

%============================

\subsubsection{Once-punctured bigon}\label{sec: once-punc bigon}

The once-punctured bigon, a punctured disk with two boundary marked points as depicted in Figure \ref{fig: once-punctured bigon}, has 4 triangulations with a total of 4 mutable arcs, as demonstrated in Figure \ref{fig: punctured bigon tagged triangulations}. There are no arrows between the radii in any of the four quivers, but there are arrows to the frozen variables which give us the relations
\[ x_1^{} x_2^\vee = x_3 + x_4 \quad \text{and} \quad x_2^{} x_1^\vee = x_3 + x_4.\]
Any deep point $p \in V(\CA(\SS),\Bbbk)$ would need to kill at least one of the plain radii and at least one of the notched radii, and consequently satisfy $0 = p(x_3) + p(x_4)$. Further, $p$ also needs to kill at least one radius in every triangulation (see Figure \ref{fig: punctured bigon tagged triangulations}), so $p$ must kill mismatched taggings. Finally, both $p(x_1) p(x_2^\vee) = 0$ and $p(x_1^\vee) p(x_2) = 0$, so at most of one of the four radii can take a non-zero value.

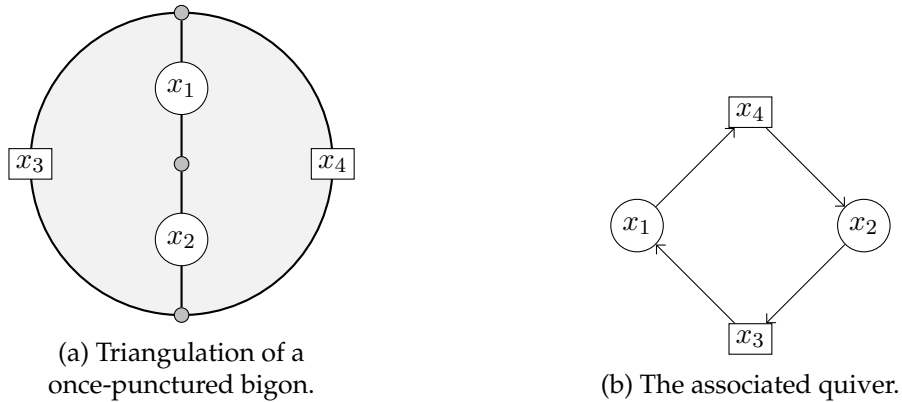
\begin{figure}[htb]
	\centering
	\captionsetup[subfigure]{justification=centering}
	\begin{subfigure}[b]{0.45\textwidth}
		\centering
		\begin{tikzpicture}[scale=2]
			\draw[thick,fill=black!5] (0,0) circle (1);
			\node[dot] (1) at (90:1) {};
			\node[dot] (2) at (270:1) {};
			\node[dot] (3) at (0,0) {};
			\draw[thick] (2) to (3);
			\draw[thick] (1) to (3);
			
			\node[mutable] (x1) at (90:0.5) {$x_1$};
			\node[mutable] (x2) at (270:0.5) {$x_2$};
			\node[frozen] (x3) at (180:1) {$x_3$};
			\node[frozen] (x4) at (0:1) {$x_4$};
		\end{tikzpicture}
		\subcaption{Triangulation of a\\once-punctured bigon.}
		\label{fig: once-punctured bigon triang}
	\end{subfigure}
	\begin{subfigure}[b]{0.45\textwidth}
		\centering
		\begin{tikzpicture}[scale=1.5]
			\node[mutable] (x1) at (-1,0) {$x_1$};
			\node[mutable] (x2) at (1,0) {$x_2$};
			\node[frozen] (x3) at (0,-1) {$x_3$};
			\node[frozen] (x4) at (0,1) {$x_4$};
			\draw[-angle 90] (x1) to (x4);
			\draw[-angle 90] (x4) to (x2);
			\draw[-angle 90] (x2) to (x3);
			\draw[-angle 90] (x3) to (x1);
		\end{tikzpicture}
		\subcaption{The associated quiver.}
		\label{fig: one-punctured bigon quiver}
	\end{subfigure}
	\caption{Initial triangulations of a once-punctured bigon.}
	\label{fig: once-punctured bigon}
\end{figure}

This gives us four subsets of deep points, with $p(x_1, x_1^\vee, x_2, x_2^\vee, x_3, x_4)$ taking the values \[ (a,0,0,0,c,-c), \quad (0,a,0,0,c,-c), \quad (0,0,a,0,c,-c), \quad \text{and} \quad (0,0,0,a,c,-c)\] for $a\in\Bbbk$ and $c\in\Bbbk^\times$. These two-dimensional subsets then intersect along the one-dimensional algebraic torus with coordinates $(0,0,0,0,c,-c)$.

\begin{rem}
	Observe that the mutable part of the quiver is the disjoint union of two $A_1$ quivers, so $\CA (\SS)$ is a cluster algebra of type $A_1 \times A_1$. If we specialize the two frozen boundary variables to 1, we get relations of the form $x_i^{} x_j^\vee = 2$, so	the cluster variety of the type $A_1 \times A_1$ cluster algebra with trivial coefficients has no deep points if $\Char(\Bbbk) \neq 2$.     
	If $\Char(\Bbbk) = 2$, there are four lines of deep points, 
	\[ (a,0,0,0,1,1), \quad (0,a,0,0,1,1), \quad (0,0,a,0,1,1), \quad \text{and} \quad (0,0,0,a,1,1),\]
	which intersect at the point $(0,0,0,0,1,1)$.
\end{rem}

%============================

\subsubsection{Once-punctured trigon}

The once-punctured trigon, a punctured disk with three boundary marked points, has nine mutable arcs (as pictured in Figure \ref{fig: punctured trigon}) and a total of 14 triangulations. Examining the figure, we see that the three mutable vertices form a quiver in the mutation class of the $A_3$ quiver.

\begin{figure}[htb]
	\centering
	\begin{subfigure}[b]{0.3\textwidth}
		\centering
		\begin{tikzpicture}[scale=2]
			\draw[thick,fill=black!5] (0,0) circle (1);
			\node[dot] (1) at (90:1) {};
			\node[dot] (2) at (210:1) {};
			\node[dot] (3) at (330:1) {};
			\node[dot] (p) at (0,0) {};
			\node[frozen] (x4) at (45:1) {$x_{4}$};
			\node[frozen] (x5) at (135:1) {$x_{5}$};
			\node[frozen] (x6) at (270:1) {$x_{6}$};
			\draw[thick,bend left] (1) to (p) {};
			\draw[thick,bend right] (1) to node[pos=0.9,sloped,rotate=90]{$\bowtie$} (p) {};
			\draw[thick] (1) to [out=330,in=60] (330:0.4) to [out=240,in=0] (2) {};
			\node[mutable3] at (69:0.5) {$x_1$};
			\node[mutable3] at (116:0.6) {$x_1^\vee$};
			\node[mutable3] at (330:0.44) {$x_{12}$};
		\end{tikzpicture}
	\end{subfigure}
	\begin{subfigure}[b]{0.3\textwidth}
		\centering
		\begin{tikzpicture}[scale=2]
			\draw[thick,fill=black!5] (0,0) circle (1);
			\node[dot] (1) at (90:1) {};
			\node[dot] (2) at (210:1) {};
			\node[dot] (3) at (330:1) {};
			\node[dot] (p) at (0,0) {};
			\node[frozen] (x4) at (45:1) {$x_{4}$};
			\node[frozen] (x5) at (135:1) {$x_{5}$};
			\node[frozen] (x6) at (270:1) {$x_{6}$};
			\draw[thick,bend left] (2) to (p) {};
			\draw[thick,bend right] (2) to node[pos=0.9,sloped,rotate=90]{$\bowtie$} (p) {};
			\draw[thick] (2) to [out=90,in=180] (90:0.4) to [out=0,in=105] (3) {};
			\node[mutable3] at (190:0.5) {$x_2$};
			\node[mutable3] at (236:0.63) {$x_2^\vee$};
			\node[mutable3] at (90:0.42) {$x_{23}$};
		\end{tikzpicture}
	\end{subfigure}
	\begin{subfigure}[b]{0.3\textwidth}
		\centering
		\begin{tikzpicture}[scale=2]
			\draw[thick,fill=black!5] (0,0) circle (1);
			\node[dot] (1) at (90:1) {};
			\node[dot] (2) at (210:1) {};
			\node[dot] (3) at (330:1) {};
			\node[dot] (p) at (0,0) {};
			\node[frozen] (x4) at (45:1) {$x_{4}$};
			\node[frozen] (x5) at (135:1) {$x_{5}$};
			\node[frozen] (x6) at (270:1) {$x_{6}$};
			\draw[thick,bend right] (3) to node[pos=0.9,sloped,rotate=90]{$\bowtie$} (p) {};
			\draw[thick,bend left] (3) to (p) {};
			\draw[thick] (3) to [out=210,in=300] (210:0.4) to [out=120,in=225] (1) {};
			\node[mutable3] at (307:0.46) {$x_3$};
			\node[mutable3] at (354:0.6) {$x_3^\vee$};
			\node[mutable3] at (210:0.42) {$x_{13}$};
		\end{tikzpicture}
	\end{subfigure}
	\caption{Arcs in the punctured trigon.}
	\label{fig: punctured trigon}
\end{figure}
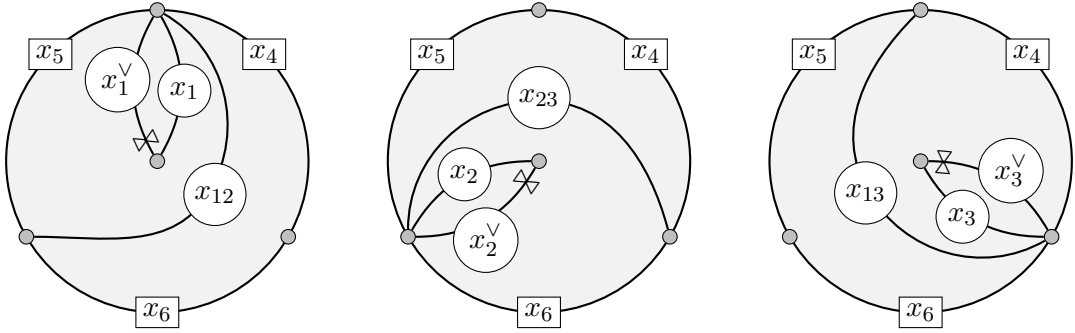

Let us choose the triangulation on the left of Figure \ref{fig: punctured trigon} as our initial seed, giving us an initial quiver: 
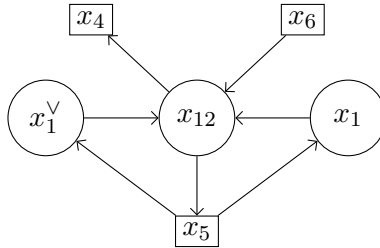
\[ \begin{tikzpicture}
	\node[mutable3] (1) at (-2,0) {$x_1^\vee$};
	\node[mutable3] (2) at (0,0) {$x_{12}$};
	\node[mutable3] (3) at (2,0) {$x_1$};
	\node[frozen] (4) at (-1.4,1.3) {$x_4$};
	\node[frozen] (6) at (1.4,1.3) {$x_6$};
	\node[frozen] (5) at (0,-1.5) {$x_5$};
	\draw[-angle 90] (1) to (2);
	\draw[-angle 90] (3) to (2);
	\draw[-angle 90] (2) to (4);
	\draw[-angle 90] (6) to (2);
	\draw[-angle 90] (2) to (5);
	\draw[-angle 90] (5) to (1);
	\draw[-angle 90] (5) to (3);
\end{tikzpicture}. \] 
The cluster algebra is generated by the cluster variables $x_1^{}, x_2^\vee, x_1^\vee, x_2^{}, x_{12}^{}$ and $x_{13}^{}$, along with the frozen variables and their inverses; the ideal of relations is generated by:
\[ 
x_1^{} x_2^\vee = x_1^\vee x_2^{} = x_{12} + x_5 \quad \text{ and } \quad x_{12} x_{13} = x_1^{} x_1^\vee x_6 + x_4 x_5. \]

We know that any deep point must kill at least one plain radius and at least one notched radius. Suppose $p(x_3) = 0$, but $p(x_2) \neq 0$ and $p(x_1) \neq 0$. Then in order to kill at least one arc in each triangulation, $p(x_{12}) = 0$. The relations $x_1^{} x_2^\vee = x_2^{} x_1^\vee = x_{12}^{} + x_5^{}$ then imply that $p(x_2^\vee) \neq 0$ and $p(x_1^\vee) \neq 0$ as well, so $p(x_3^\vee) = 0$. Then from the triangulations pictured in Figure \ref{fig: punctured trigon}, we now see that $p$ must kill $x_{23}$ and $x_{13}$; the relation $x_{13}^{} x_{23}^{} = x_3^{} x_3^\vee x_5^{} + x_4^{} x_6^{}$ then implies that $p(x_4) p(x_6) = 0$, which is impossible. Hence, $p$ cannot kill a single plain radius while sparing the other two.

Next, suppose $p$ kills both $x_1$ and $x_2$. From the relation $x_1 x_{23} = x_2 x_4 + x_3 x_5$, we get $p(x_3) p(x_5) = 0$, which implies that $p(x_3) = 0$ as well, so $p$ must kill all plain radii. Repeating these arguments with notched radii, we see that a deep point $p$ must kill all six radii; that is, $p$ is totally radial.

From the above relations, we see that sending the 6 radii to 0 gives us the following relations:
\[ 0 = p(x_{12}) + p(x_5), \quad\quad 0 = p(x_{23}) + p(x_6), \quad\quad 0 = p(x_{13}) + p(x_4). \]  
Hence, any deep point must kill all radii and send each pair of arcs which bound a punctured bigon to additive inverses. While six variables are sent to non-zero values, there are only three possible choices for their values. Therefore, the deep locus is isomorphic to the algebraic torus $(\Bbbk^\times)^3$.

\begin{rem}
	As in Theorem \ref{thm: deeppolygon}, if we specialize the frozen boundary variables to 1, there is a unique deep point which is realized here by sending $x_{12}, x_{13}, $ and $x_{23}$ to a value of $-1$. 
\end{rem}

%============================

\subsection{Characterization of deep points}\label{sec: char once-punct disks}

As noted, many of the tools used in the classification of deep points in the unpunctured case will not directly apply to punctured surfaces. In this section, we characterize some local results for once-punctured disks with $n\geq 4$ marked points on the boundary.

%============================

\subsubsection{Vanishing Lemmata}

Consider Figure \ref{fig: radial quadrilateral} below. To make the following arguments easier to follow, we will conflate the cluster variables with their image (i.e. we will write $a$ instead of $p(x_i) = a$). We will use the convention that if $r_j$ is the value of a plain radius, $r_j^\vee$ will be the value of the notched version of the same radius. The dashed arc labeled $a$ is not a boundary arc when $n\geq 5$, but is a boundary arc when $n=4$. Arcs labeled $c_{j}$ cut out a punctured bigon along with the boundary arc $b_j$. Arcs labeled $d_{ij}$ cut out punctured trigons along with boundary arcs $b_i$ and $b_j$. 

\begin{figure}[h!tb]
	\begin{subfigure}[b]{0.49\textwidth}
		\centering
		\begin{tikzpicture}[scale=1.2]
			\path[fill=black!5] (-3,1) to (3,1) to (2,-1) to (-2,-1) to (-3,1);
			\draw[thick] (-3,1) to (-2,-1) to (2,-1) to (3,1);
			\draw[dashed] (-3,1) to (3,1);
			\node[dot] (1) at (3,1) {};
			\node[dot] (2) at (2,-1) {};
			\node[dot] (3) at (-2,-1) {};
			\node[dot] (4) at (-3,1) {};
			\node[dot] (p) at (0,0) {};
			\draw[thick] (p) to (1) (p) to (2) (p) to (3) (p) to (4) {};
			\node[mutable] (r1) at (1.5,0.5) {$r_1$};
			\node[mutable] (r2) at (1,-0.5) {$r_2$};
			\node[mutable] (r3) at (-1,-0.5) {$r_3$};
			\node[mutable] (r4) at (-1.5,0.5) {$r_4$};
			\node[frozen] (b1) at (2.5,0) {$b_1$};
			\node[frozen] (b2) at (0,-1) {$b_2$};
			\node[frozen] (b3) at (-2.5,0) {$b_3$};
			\node[mutable] (a) at (0,1) {$a$};
		\end{tikzpicture}
	\end{subfigure}
	\begin{subfigure}[b]{0.49\textwidth}
		\centering
		\begin{tikzpicture}[scale=1.2]
			\path[fill=black!5] (-3,1) to (3,1) to (2,-1) to (-2,-1) to (-3,1);
			\draw[thick] (-3,1) to (-2,-1) to (2,-1) to (3,1);
			\draw[dashed] (-3,1) to (3,1);
			\node[dot] (1) at (3,1) {};
			\node[dot] (2) at (2,-1) {};
			\node[dot] (3) at (-2,-1) {};
			\node[dot] (4) at (-3,1) {};
			\node[dot] (p) at (0,0) {};
			\draw[thick] (3) to [out=60,in=180] (0,0.2) to [out=0,in=120] (2) {}; 
			\node[mutable] (c2) at (1.5,-0.4) {$c_{2}$};
			\draw[thick,out=80,in=200] (3) to (1) {};
			\node[mutable] (d12) at (1.7,0.6) {$d_{12}$};
			\node[frozen] (b1) at (2.5,0) {$b_1$};
			\node[frozen] (b2) at (0,-1) {$b_2$};
			\node[frozen] (b3) at (-2.5,0) {$b_3$};
			\node[mutable] (a) at (0,1) {$a$};
		\end{tikzpicture}
	\end{subfigure}
	\caption{Arcs in a punctured quadrilateral.}
	\label{fig: radial quadrilateral}
\end{figure}
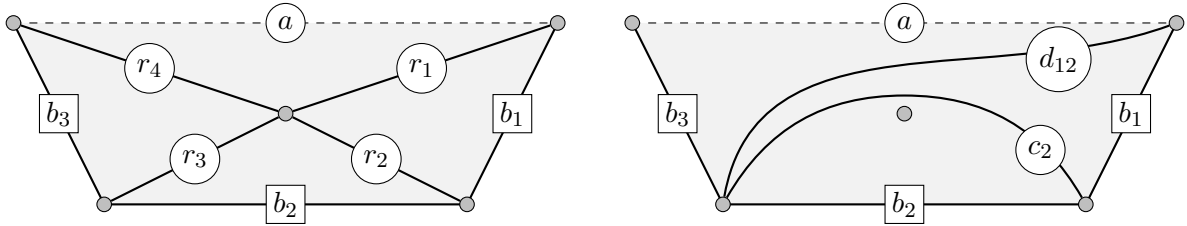

\begin{lemma}\label{lemma: kill neigbor radii}
	Let $\SS$ be a once-punctured disk with $n\geq 4$ marked points on the boundary, and let $p\in V(\CA (\SS),\Bbbk)$ have image as above. If $r_2=0$ and $r_3=0$, then $r_1=0$ and $r_4=0$.
\end{lemma}

\begin{proof}
	First consider the mutation relation $r_3 c_1 = r_1 b_2 + r_2 d_{12}$, which implies that $r_1 b_2 = 0$. Because it corresponds to a boundary arc, $b_2 \neq 0$, so $r_1 = 0$. 
	
	Next, consider the mutation relation $r_4 c_1 = r_2 a + r_1 d_{1a}$, which implies that $ r_4 c_1 = 0$, and the relation $r_1 r_2^\vee = c_1 + b_1$, which implies $c_1 + b_1 = 0$. Since $b_1$ corresponds to a boundary arc, $b_1\neq 0$, so $c_1 \neq 0$ as well. Therefore, $r_4 = 0$.
\end{proof}
That is, if a point $p$ kills two radially adjacent radii, then it must also kill the adjacent radius on either side. To obtain the result for notched radii, simply replace all $r_i$ in Lemma \ref{lemma: kill neigbor radii} with $r_i^\vee$. 

\begin{coro}\label{cor: kill all radii}
	Let $\SS$ be a once-punctured disk with $n\geq 2$ marked points on the boundary, and let $p\in V(\CA (\SS),\Bbbk)$. If $p$ kills two radially adjacent plain (resp. notched) radii, then $p$ kills all plain (resp. notched) radii.
\end{coro}

We also have the following porism from the proof of Lemma \ref{lemma: kill neigbor radii}, which emphasizes an idea we will use later. 
\begin{porism}\label{por: c_i non-zero}
	Let $\SS$ be a once-punctured disk with $n\geq 4$ marked points on the boundary, and let $p\in V(\CA (\SS),\Bbbk)$ have image as above. If $p$ kills two radially adjacent radii, either both plain or both notched, then for all $c_i$ which cut out a once-punctured bigon along with a boundary arc $b_i$, $c_i = -b_i$, so $c_i \neq 0$. 
\end{porism}

Next, we want to examine what happens to the plain (resp. notched) radii when all notched (resp. plain) radii are killed by a deep point. 

\begin{lemma}\label{lemma: opposite radii}
	Let $\SS$ be a once-punctured disk with $n\geq 4$ marked points on the boundary, and let $p\in V(\CA (\SS),\Bbbk)$ be a deep point with image as above, such that $p$ kills all notched radii; that is, $r_i^\vee = 0$ for all $i$.
	\begin{enumerate}
		\item $b_1 b_3 - a b_2 = 0$ if and only if $d_{23} = d_{1a} = 0$ or $d_{12} = d_{3a} = 0$.
		\begin{enumerate}
			\item If $d_{23} = d_{1a} = 0$, then $r_2=0$ if and only if $r_4 = 0$.
			\item If $d_{12} = d_{3a} = 0$, then $r_1=0$ if and only if $r_3=0$. 
		\end{enumerate}
		\item if $b_1 b_3 - a b_2 \neq 0$ and $r_2=0$, then $p$ does not kill any $d_{ij}$, and $p$ kills all plain radii. 
	\end{enumerate}
\end{lemma}
\begin{proof}
	By assumption, $p$ kills all notched radii, so Porism \ref{por: c_i non-zero} gives us $c_i = -b_i$. From the relations $r_1 r_3^\vee = d_{12} + d_{3a} = 0$ and $r_2 r_4^\vee = d_{23} + d_{1a} = 0$, we also have $d_{12} = -d_{3a}$ and $d_{23} = -d_{1a}$. Substituting, we get the relation $d_{23} d_{12} = b_1 b_3 + ac_2 = b_1 b_3 - ab_2$. 
	\begin{enumerate}
		\item If $b_1 b_3 - a b_2 = 0$, then either $d_{23}=0$ or $d_{12}=0$. Thus, if $b_1 b_3 - a b_2 = 0$, then either $d_{23} = d_{1a} = 0$ or $d_{12} = d_{3a} = 0$. Conversely, if either $d_{23} = 0$ or $d_{3a} = 0$, then $b_1 b_3 - a b_2 = 0$.
		
		\begin{enumerate}
			\item Suppose $d_{23}=d_{1a}=0$. We have the simplification $-b_2 r_4 = b_3 r_2$; since $b_i \neq 0$ for all $i$, $r_4=0$ if and only if $r_2=0$.
			
			\item Next, suppose $d_{12}=d_{3a}=0$. We have the simplification $-b_2 r_1 = b_1 r_3$; since $b_i \neq 0$ for all $i$, $r_3=0$ if and only if $r_1=0$.
		\end{enumerate}
		
		\item Now suppose that $b_1 b_3 - ab_2 \neq 0$ and $r_2 = 0$. Since $d_{23} d_{12} = b_1 b_3 - ab_2$, none of the $d_{ij}$ are killed by $p$. Simplifying the exchange relations, we have
	    \[ r_3 = \frac{-b_2}{b_1} r_1 = \frac{-b_2}{d_{23}} r_4. \]     
		Then substituting into the equation $c_3 r_1 = d_{3a} r_4 + r_3 a$, we get 
		\[ c_3 r_1 =  \frac{d_{23} d_{3a} r_1}{b_1} - \frac{a b_2 r_1}{b_1}, \] 
		so $ b_1 b_3 r_1 = (a b_2 - d_{23} d_{3a}) r_1$ and $ b_1 b_3 r_1 = (a b_2 - b_1 b_3 + a b_2) r_1$, giving us $2(b_1 b_3 - a b_2) r_1 = 0$. Since $b_1 b_3 - a b_2 \neq 0$, it must be that $r_1 = 0$, so all plain radii are killed by $p$ as a result of Corollary \ref{cor: kill all radii}.\qedhere
	\end{enumerate}
\end{proof}

By instead killing all plain radii, we immediately get the analogous result with the taggings swapped.

\begin{coro}\label{cor: opposite radii}
	Let $\SS$ be a once-punctured disk with $n\geq 4$ marked points on the boundary, and let $p\in V(\CA (\SS),\Bbbk)$ be a deep point with image as above, such that $p$ kills all plain radii; that is, $r_i = 0$ for all $i$.
	\begin{enumerate}
		\item $b_1 b_3 - a b_2 = 0$ if and only if $d_{23} = d_{1a} = 0$ or $d_{12} = d_{3a} = 0$.
		\begin{enumerate}
			\item If $d_{23} = d_{1a} = 0$, then $r_2^\vee=0$ if and only if $r_4^\vee = 0$.
			\item If $d_{12} = d_{3a} = 0$, then $r_1^\vee=0$ if and only if $r_3^\vee=0$. 
		\end{enumerate}
		\item if $b_1 b_3 - a b_2 \neq 0$ and $r_2^\vee=0$, then $p$ does not kill any $d_{ij}$, and $p$ kills all notched radii. 
	\end{enumerate}
\end{coro}

\begin{coro}\label{cor: kill all radii both taggings}
	Let $\SS$ be a once-punctured disk with $n\geq 4$ marked points on the boundary, and let $p\in V(\CA (\SS),\Bbbk)$ be a deep point with image as above, such that $p$ kills all notched (resp. plain) radii. If $b_1 b_3 - a b_2 \neq 0$, then $p$ kills all plain (resp. notched) radii as well.
\end{coro}
\begin{proof}
	Since $p$ is a deep point, it must kill at least one plain radius in $\SS$. Choose a punctured quadrilateral subsurface matching Figure \ref{fig: radial quadrilateral} such that the plain radius being killed is in the $r_2$ position. Since $b_1 b_3 - a b_2 \neq 0$, Lemma \ref{lemma: opposite radii} implies that all plain radii in $\SS$ are killed by $p$. To prove for notched radii, replace Lemma \ref{lemma: opposite radii} with Corollary \ref{cor: opposite radii}.
\end{proof}

Next, we will consider the deep points that do not kill all radii of a given tagging.

\begin{lemma}\label{lemma: non-radial deep quad}
	Let $\SS$ be a once-punctured disk with $n\geq 4$ marked points on the boundary, and let $p \in V(\CA(\SS),\Bbbk)$ be a deep point such that $r_2 = 0$ but $r_3 \neq 0$. Then 
	\begin{enumerate}
		\item $r_2^\vee = 0$.
		\item $r_1 \neq 0$.
		\item $r_4 = 0$ if and only if $d_{1a} = d_{23} = 0$.
		\item if $r_4 = 0$, then $b_1 b_3 - a b_2 = 0$. 
	\end{enumerate} 
\end{lemma}
\begin{proof}
	From the relation $r_2 r_3^\vee = b_2 + c_2 = 0$, we have $r_3 r_2^\vee = 0$ as well, so $r_2^\vee = 0$. 
	From the contrapositive of Corollary \ref{cor: kill all radii}, since $r_3 \neq 0$, $p$ cannot kill both $r_1$ and $r_2$; consequently, $r_1 \neq 0$.
	
	Next, consider the relation $r_2 b_3 + r_4 b_2 = r_3 d_{1a}$, which simplifies to $r_4 b_2 = r_3 d_{1a}$. Since $r_3 \neq 0$, we have $r_4 = 0$ if and only if $d_{1a} = 0$. From the relation $r_4 r_2^\vee = 0 = d_{1a} + d_{23}$, we likewise have that $d_{1a} = 0$ if and only if $d_{23} = 0$. 
	Combining with the relation $d_{1a} d_{12} = ab_2 + c_1 b_3 = ab_2 - b_1 b_3$, if $r_4 = 0$ then $b_1 b_3 - ab_2 = 0$. 
\end{proof}

%============================

\subsubsection{Cutting and gluing}\label{section: cutting once-punc disks}

To further classify the deep points of once-punctured disks, we will want to extend and utilize the idea of cutting and gluing. 

\begin{lemma}\label{lemma: suitable c_i}
	Let $\SS$ be a once-punctured disk with $n\geq 4$ marked points on the boundary, and let $p \in V(\CA(\SS),\Bbbk)$ be deep. There is an index $i$ such that $c_i \neq 0$, $r_i = 0$, and $r_i^\vee = 0$. 
\end{lemma}
\begin{proof}
	Since $p$ is deep, it must kill at least one plain radius, so choose an indexing such that $r_2 = 0$. If $r_3 = 0$, then $p$ must kill all plain radii by Lemma \ref{lemma: kill neigbor radii}. Then for all $i$, we have $r_i r_{i+1}^\vee = c_i + b_i = 0$, so $c_i \neq 0$ for $1 \leq i \leq n$. Since $p$ must also kill at least one notched radius, there is an index $i$ such that $r_i^\vee = 0$, $r_i = 0$, and $c_i \neq 0$. 
	
	If $r_3 \neq 0$, then by Lemma \ref{lemma: non-radial deep quad}, we have $r_2^\vee = 0$. Then $r_2 r_1^\vee = c_2 + b_2 = 0$, and we have $c_2 \neq 0$. 
\end{proof}

Cutting $\SS$ along $c_i$ produces a disconnected surface $\SS\smallsetminus c_i = \SS_1 \sqcup \SS_n$, where $\SS_1$ is a once-punctured bigon and $\SS_n$ is homeomorphic to the unpunctured $n$-gon $\Delta_n$. Recall that a deep point  $p \in V(\CA(\SS\smallsetminus c_i),\Bbbk)$ must kill at least one arc in every triangulation of $\SS\smallsetminus c_i$, so it must kill at least one arc in every triangulation of either $\SS_1$ or $\SS_n$, but not necessarily both. For convenience, let us denote the separate restrictions as $p_1 \coloneqq p \vert_{\SS_1}^{}$ and $p_n \coloneqq p \vert_{\SS_n}^{}$.

Recall from Proposition \ref{prop: deep cutting inclusion} that 
\[ 
\begin{Bmatrix} \text{deep points in } V(\CA(\SS),\Bbbk) \\ \text{non-zero on }S \end{Bmatrix} 
\hookrightarrow
\begin{Bmatrix}\text{deep points in } V(\CA(\SS\smallsetminus S),\Bbbk) \\ \text{gluable along } S \end{Bmatrix}.
\]
As noted in the discussion following that proposition, this is not always a bijection, so we must show that there is a bijection for $S = \{c_i\}$. 

\begin{prop}\label{prop: cutting totally radial}
	Let $\SS$ be a once-punctured disk with $n\geq 4$ marked points on the boundary. There is a bijection 
	\[ 
	\begin{Bmatrix} \text{totally radial deep} \\ \text{points in } V(\CA(\SS),\Bbbk) \\ \text{non-zero on }c_i \end{Bmatrix} 
	\longleftrightarrow
	\begin{Bmatrix}\text{totally radial deep} \\ \text{points in } V(\CA(\SS\smallsetminus c_i),\Bbbk) \\ \text{gluable along } c_i \end{Bmatrix}.
	\]
	Further, the subvariety of totally radial deep points of $V(\CA(\SS),\Bbbk)$ is isomorphic to the cluster variety $V(\CA(\Delta_n),\Bbbk)$.
\end{prop}
\begin{proof}
	Suppose that $p\in V(\CA(\SS),\Bbbk)$ is a totally radial deep point, so it kills all radii. Then its image under the gluing isomorphism $\bar{p}$ must also kill all radii in $\SS\smallsetminus c_i$. 
	Since $\SS \smallsetminus c_i$ is disconnected, $\bar{p} \in V(\CA(\SS\smallsetminus c_i),\Bbbk)$ is deep if either the restriction $p_1$ to the punctured bigon is deep, the restriction $p_n$ to the unpunctured $n$-gon is deep, or both.  
	Since $p_1$ kills all radii in the punctured bigon, it is deep, and $\bar{p}$ is deep.  
	
	Now, suppose that $\bar{p} \in V(\CA(\SS\smallsetminus c_i),\Bbbk)$ is a totally radial deep point, so it kills all radii. Since there are no radii in $\SS_n$, there are no conditions on $p_n$, but $p_1$ kills all radii in $\SS_1$. Lifting to $p \in V(\CA(\SS),\Bbbk)$, we see that $p$ kills $r_i$, $r_i^\vee$, $r_{i+1}$, and $r_{i+1}^\vee$. By Corollary \ref{cor: kill all radii}, $p$ must kill all radii of both taggings, so $p$ is a totally radial deep point. 
	
	Since the boundary values of $p_1$ on $\SS_1$ are determined by the boundary values of $p_n$ on $\SS_n$, every totally radial point $p$ is determined by a point $p_n \in V(\CA(\SS_n),\Bbbk) \cong V(\CA(\Delta_n),\Bbbk)$. On the other hand, totally radial deep points on $\SS$ have no restriction on boundary values, so every point of $V(\CA(\SS_n),\Bbbk) \cong V(\CA(\Delta_n),\Bbbk)$ can be realized as the restriction $p_n$ of a totally radial deep point $\bar{p}$. 
\end{proof}

Proposition \ref{prop: cutting totally radial} tells us that the points of $V(\CA(\Delta_n),\Bbbk)$ are in bijection with the totally radial deep points of $V(\CA(\SS),\Bbbk)$, so there is an inclusion of the cluster variety of the unpunctured polygon $V(\CA(\Delta_n),\Bbbk)$ into the deep locus of $V(\CA(\SS),\Bbbk)$. Restricting $\CA(\Delta_n)$ to $\CA(A_{n-3})$ by specializing the frozen boundary variables to 1, we also get an inclusion of the cluster variety $V(\CA(A_{n-3}),\Bbbk)$ into the deep locus of $V(\CA(\SS),\Bbbk)$. 

\begin{coro}\label{coro: inclusions disks}
	Let $\Delta_n$ be an unpunctured disk with $n\geq 4$ marked points on the boundary, and let $\SS$ denote the same disk with a distinguished marked point (puncture) in the interior. There are inclusions of:
	\begin{itemize}
		\item the cluster variety $V(\CA(\Delta_n),\Bbbk)$ into the deep locus of $V(\CA(\SS),\Bbbk)$, and
		\item the cluster variety $V(\CA(A_{n-3}),\Bbbk)$ into the deep locus of $V(\CA(\SS),\Bbbk)$. 
	\end{itemize}
\end{coro}

\begin{warn}
Recall that specializing the frozen boundary variables of $\SS$ to 1 forces one of the frozen boundary variables of $\SS_n$ to be -1. Thus, $\mathcal{F}$ does not induce an inclusion of $V(\CA(A_{n-3}),\Bbbk)$ into the deep locus of $V(\CA(D_n),\Bbbk)$, but rather a deep subvariety of $V(\CA(\SS),\Bbbk)$ corresponding to points assigning boundary values $(-1, 1, 1, \dotsc, 1)$.    
\end{warn}

Next, let us characterize the deep points that do not kill all radii. Recall that a point $p$ in $V(\CA(\SS\smallsetminus c_i),\Bbbk)$ is called \emph{deep-on-components} if its restriction to each component is deep. 

\begin{prop}\label{prop: non-totally radial deep-on-components}
	Let $\SS$ be a once-punctured disk with $n\geq 4$ marked points on the boundary. There is a bijection 
	\[ 
	\begin{Bmatrix} \text{non-totally radial deep} \\ \text{points in } V(\CA(\SS),\Bbbk) \\ \text{non-zero on }c_i \end{Bmatrix} 
	\longleftrightarrow
	\begin{Bmatrix}\text{non-totally radial} \\ \text{deep-on-components} \\ \text{points in } V(\CA(\SS\smallsetminus c_i),\Bbbk) \\ \text{gluable along } c_i \end{Bmatrix}.
	\]    
\end{prop}
\begin{proof}
    Let $p \in V(\CA(\SS),\Bbbk)$ be a deep point such that $r_2 = 0$ and $r_3 \neq 0$. First, since $r_2 r_3^\vee = 0 = c_2 + b_2$, we have $c_2 = -b_2$ and $c_2 \neq 0$. From Lemma \ref{lemma: non-radial deep quad}, we have $r_2^\vee = 0$ as well. 
    Recall from Proposition \ref{prop: deep cutting inclusion} that, since $c_2 \neq 0$, the image of $p$ under the cutting isomorphism is a deep point $\bar{p} \in V(\CA(\SS\smallsetminus c_2), \Bbbk)$ that is gluable along $c_2$. Using the same notation as above, since $\bar{p}$ is deep, we know that either $p_1$ is deep on the once-punctured bigon, $p_n$ is deep on the unpunctured $n$-gon, or both. 

    If $p_1$ is deep on the once-punctured bigon, then at least one arc in the triangulation $\{r_3, r_3^\vee\}$ must be killed, so we know that $r_3^\vee = 0$. By Corollary \ref{cor: kill all radii}, this means that $p$ must kill all notched radii. Since $r_3 \neq 0$, the contrapositive of Lemma \ref{lemma: opposite radii} part 2 gives us $b_1 b_3 - a b_2 = 0$, so $b_1 b_3 + a c_2 = 0$. 
    	\begin{itemize}
    		\item \textbf{Claim:} $r_4 = 0$.
    		
    		\emph{Proof of Claim:} Seeking a contradiction, suppose $r_4 \neq 0$. Since all notched radii are killed by $p$, we have $r_3 r_4^\vee = c_3 + b_3 = 0$, so $c_3 \neq 0$. If we cut along $c_3$ instead of $c_2$, $\widehat{p} \in V(\CA(\SS\smallsetminus c_3),\Bbbk)$ must be deep. Since both $r_3 \neq 0$ and $r_4 \neq 0$, the restriction of $\widehat{p}$ to the bigon is not deep, so it must instead be deep on the unpunctured $n$-gon. Since $\widehat{p}$ is deep on the unpunctured $n$-gon, $p$ must kill both $d_{23}$ and $d_{3a}$; since all notched radii are killed, we have $d_{23} = d_{1a} = 0$ and $d_{12} = d_{3a} = 0$. By Lemma \ref{lemma: non-radial deep quad}, $r_4 = 0$. $\qed$
    	\end{itemize}
    	
    	By Corollary \ref{cor: kill all radii}, we also know that $r_5 \neq 0$ since $p$ does not kill all plain radii. Repeatedly applying the proof of the claim around the puncture, we see that $p$ must kill all even-indexed plain radii and spare all odd-indexed radii. Note that if $n$ is odd, all radii are even-indexed, so all deep points must be totally radial. Hence, under our assumptions, $n$ must be even.
    	
    	If $n=4$, we observe that $r_1 b_2 + r_3 b_1 = 0$ and $r_3 b_4 + r_1 b_3 = 0$, so $r_1 \neq 0$ and we have \[ r_1 = \frac{-b_1}{b_2} r_3 = \frac{b_1 b_3}{b_2 b_4} r_1.\]
    	More generally, all even-indexed radii are killed by $p$, so $ r_{j+2} b_{j+3} + r_j b_{j+2} = 0$. As a result, 
    	\[ r_1 ~=~ (-1)~\frac{b_1}{b_2} r_3 ~=~ (-1)^2~\frac{b_1 b_3}{b_2 b_4} r_5 ~=~ \dotsb ~=~ (-1)^{n/2}~ \frac{b_1 b_3 \dotsm b_{n-1}}{b_2 b_4 \dotsm b_n} r_1.\]
    	
    	Since $r_1 \neq 0$, the values on the boundary edges satisfy the equation
    	\begin{equation}
    		\frac{b_1 b_3 \dotsm b_{n-1}}{b_2 b_4 \dotsm b_n} = (-1)^{n/2}. 
    	\end{equation}
    	As a result, the boundary values on the unpunctured $n$-gon satisfy the equation 
    	\begin{equation}
    		\frac{b_1 b_3 \dotsm b_{n-1}}{c_2 b_4 \dotsm b_n} = (-1)^{\frac{n+2}{2}}, 
    	\end{equation}
    	and $p_n$ is necessarily deep on the unpunctured $n$-gon by Theorem \ref{thm: deeppolygon}. That is, if $n$ is even and $p_1$ is deep, then $p_n$ is deep as well. If $n$ is odd, then Theorem \ref{thm: deeppolygon} says that $p_n$ cannot be deep.
    	
    	If $n$ is even and $p_n$ is deep, Theorem \ref{thm: deeppolygon} gives us the boundary relation \[ \frac{b_1 b_3 \dotsm b_{n-1}}{c_2 b_4 \dotsm b_n} = (-1)^{\frac{n+2}{2}}, \quad \text{and} \quad a = \frac{b_1 b_3}{b_2},\] so $b_1 b_3 - ab_2 = 0$. The same theorem gives us $d_{12} = d_{23} = 0$, and since $r_2 = 0$, we have $d_{23} = -d_{1a}$, so $d_{1a} = 0$ as well. By Lemma \ref{lemma: non-radial deep quad}, we also now have $r_4 = 0$. Since $r_2 = 0$ and $r_3 \neq 0$, Corollary \ref{cor: kill all radii} says that $r_1 \neq 0$ as well.
    	
    	By symmetry (repeating with a cut of $\SS$ along $c_{2k}$ for all $k$), we see that all even-indexed plain radii are killed and all odd-indexed plain radii are spared. By the same argument, all non-radial arcs of even length must be killed, so all $d_{ij} = 0$. Since all non-radial arcs of even length are killed, the relations $r_{2k+1}^{} r_{2k+3}^\vee = 0 + 0$ force all odd-indexed notched radii to be killed. Since $r_2^\vee = 0$, we have adjacent notched radii being killed, so all notched radii are killed by Corollary \ref{cor: kill all radii}. 
    	
    	Note that since all notched radii are killed and $r_2 = 0$, the restriction to $p_1$ is deep as well. That is, if $n$ is even and $p_n$ is deep, then $p_1$ is deep as well.

        To complete the proof, repeat the above argument with the taggings switched.
\end{proof}

%============================

\subsubsection{Classifying deep points}

From the above proofs, we can extract the following porisms.
\begin{porism}\label{porism: kill all opposite tagging}
	Let $\SS$ be a once-punctured disk with $n\geq 4$ marked points on the boundary.
	\begin{enumerate}
		\item Let $p \in V(\CA(\SS),\Bbbk)$ be a deep point such that $r_2 = 0$ and $r_3 \neq 0$. Then $p$ kills all notched radii.
		\item Let $p \in V(\CA(\SS),\Bbbk)$ be a deep point such that $r_2^\vee = 0$ and $r_3^\vee \neq 0$. Then $p$ kills all plain radii.
	\end{enumerate}
\end{porism}

\begin{porism}\label{porism: kill alternating opposite radii}
	Let $\SS$ be a once-punctured disk with $n\geq 4$ marked points on the boundary.
	\begin{enumerate}    
		\item Let $p\in V(\CA (\SS),\Bbbk)$ be a deep point with image as above, such that $p$ kills all notched radii. Then $p$ either kills all radii, or $p$ kills alternating plain radii (that is, all even-indexed or all odd-indexed plain radii). 
		
		\item Let $p\in V(\CA (\SS),\Bbbk)$ be a deep point with image as above, such that $p$ kills all plain radii. Then $p$ either kills all radii, or $p$ kills alternating notched radii (that is, all even-indexed or all odd-indexed notched radii). 
	\end{enumerate}
\end{porism}
\begin{coro}\label{cor: odd-gon kill all radii}
	Let $\SS$ be a once-punctured disk with $n \geq 5$ marked points on the boundary, and let $p\in V(\CA (\SS),\Bbbk)$ be a deep point. If $n$ is odd, then $p$ kills all radii of both taggings; that is, $p$ is a totally radial deep point. 
\end{coro}

\begin{porism}\label{porism: boundary alternating product}
	Let $\SS$ be a once-punctured disk with $n = 2k \geq 4$ marked points on the boundary, and let $p \in V(\CA(\SS),\Bbbk)$ be a deep point such that not all radii are killed; that is, $p$ is not a totally radial deep point. 
	Then the values on the boundary edges satisfy the equation
	\begin{equation}\label{eq: boundary alternating product}
		\frac{b_1 b_3 \dotsm b_{n-1}}{b_2 b_4 \dotsm b_n} = (-1)^{n/2} .
	\end{equation}
\end{porism}

Finally, we can combine these results to complete the classification of deep points of once-punctured disks.
\begin{thm}\label{thm: punctured disk deep types}
	Let $\SS$ be a once-punctured disk with $n\geq 4$ marked points on the boundary, and let $p \in V(\CA(\SS),\Bbbk)$ be a deep point.
	\begin{enumerate}
		\item If $n$ is odd, then $p$ kills all radii of both taggings; that is, $p$ is a totally radial deep point.
		\item If $n$ is even, then $p$ is one of the following three types:
		\begin{enumerate}
			\item $p$ is a totally radial deep point; that is, it kills all radii of both taggings. 
			\item $p$ is a notched deep point; that is, it kills all notched radii, but not all plain radii. In this case, $p$ kills alternating plain radii \textemdash either all odd-indexed plain radii or all even-indexed plain radii \textemdash and the values of the boundary edges satisfy Equation \ref{eq: boundary alternating product}.
			\item $p$ is a plain deep point; that is, it kills all plain radii, but not all notched radii. In this case, $p$ kills alternating notched radii \textemdash either all odd-indexed notched radii or all even-indexed notched radii \textemdash and the values of the boundary edges satisfy Equation \ref{eq: boundary alternating product}.
		\end{enumerate}            
	\end{enumerate}
	
\end{thm}

\subsection{Once-punctured quadrilateral}

We will now explicitly compute the deep points corresponding to the once-punctured disk with $n=4$ marked points on the boundary; let $\SS$ denote this surface. Using the quiver in Figure \ref{fig: punctured quadrilateral} for our initial seed, the cluster algebra is generated by $x_1$, $x_2^\vee$, $x_1^\vee$, $x_2$, $x_3$, $x_3^\prime$, $x_4$ and $x_4^\prime$, along with the frozen variables and their inverses; here, $x_2$ denotes the plain radius realized by a flip of $x_1^\vee$ in the given triangulation. The ideal of relations is generated by:
\begin{align*}
	x_1 x_2^\vee &= x_3 + x_5 & x_1^\vee x_2 &= x_3 + x_5 \\
	x_3 x_3^\prime &= x_1 x_1^\vee x_6 + x_4 x_5 & x_4 x_4^\prime &= x_3 x_7 + x_6 x_8.    
\end{align*}

\begin{figure}[htb]
	\centering
	\begin{subfigure}[b]{0.45\textwidth}
		\centering
		\begin{tikzpicture}[scale=2,baseline={(0,0)}]
			\draw[thick,fill=black!5] (0,0) circle (1);
			\node[dot] (2) at (270:1) {};
			\node[dot] (3) at (225:1) {};
			\node[dot] (4) at (150:1) {};
			\node[dot] (1) at (90:1) {};
			\node[dot] (p) at (5:0.6) {};
			\node[frozen] (x5) at (315:1.03) {$x_{5}$};
			\node[frozen] (x6) at (245:1.03) {$x_{6}$};
			\node[frozen] (x7) at (175:1.03) {$x_{7}$};
			\node[frozen] (x8) at (130:1.03) {$x_{8}$};
			\draw[thick,bend right] (1) to node[pos=0.9,sloped,rotate=90]{$\bowtie$} (p) {};
			\draw[thick,bend left] (1) to (p) {};
			\draw[thick,bend right] (1) to (2) {};
			\draw[thick,bend right] (1) to (3) {};
			\node[mutable3] at (45:0.75) {$x_1$};
			\node[mutable3] at (70:0.37) {$x_1^\vee$};
			\node[mutable3] at (255:0.58) {$x_3$};
			\node[mutable3] at (195:0.67) {$x_4$};
		\end{tikzpicture}
	\end{subfigure}
	\begin{subfigure}[b]{0.45\textwidth}
		\centering
		\begin{tikzpicture}[scale=1,baseline={(0,0)}]
			\node[frozen] (x5) at (-3,0) {$x_5$};
			\node[mutable3] (x1) at (-2,1) {$x_1$};
			\node[mutable3] (x2) at (-2,-1) {$x_1^\vee$};
			\node[mutable3] (x3) at (-1,0) {$x_3$};
			\node[frozen] (x6) at (0,-1) {$x_6$};
			\node[mutable3] (x4) at (1,0) {$x_4$};
			\node[frozen] (x7) at (2,1) {$x_7$};
			\node[frozen] (x8) at (2,-1) {$x_8$};
			\draw[-angle 90] (x1) to (x5);
			\draw[-angle 90] (x2) to (x5);
			\draw[-angle 90] (x3) to (x1);
			\draw[-angle 90] (x3) to (x2);
			\draw[-angle 90] (x5) to (x3);
			\draw[-angle 90] (x4) to (x3);
			\draw[-angle 90] (x6) to (x4);
			\draw[-angle 90] (x3) to (x6);
			\draw[-angle 90] (x8) to (x4);
			\draw[-angle 90] (x4) to (x7);
		\end{tikzpicture}
	\end{subfigure}
	\caption{A triangulation of a once-punctured quadrilateral and its associated quiver.}
	\label{fig: punctured quadrilateral}
\end{figure}
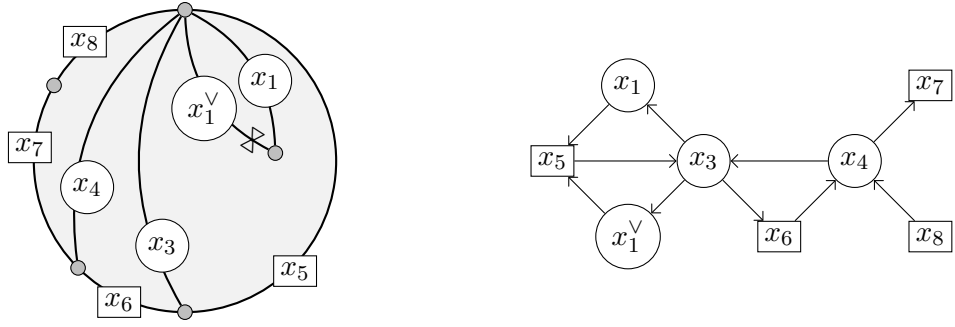

From Theorem \ref{thm: punctured disk deep types}, we know that every deep point is one of three types, which we will consider separately. Throughout the following, let $p\in V(\CA(\SS),\Bbbk)$ be deep.  
In relating previous results, the arcs in Figure \ref{fig: radial quadrilateral} correspond to the variables in Figure \ref{fig: punctured quadrilateral} via 
\begin{align*}
	p(x_1, x_2^\vee, x_1^\vee, x_2, x_3, x_3^\prime, x_4, x_4^\prime) &= (r_1, r_2^\vee, r_1^\vee, r_2, c_1, d_{3a}, d_{12}, d_{1a}) \\
	p(x_5, x_6, x_7, x_8) &= (b_1,b_2,b_3,a)
\end{align*}

\begin{enumerate}
	\item First, suppose that $p$ kills all radii of both taggings. Because $p$ kills all radii, we have $0 = p(x_3^\prime) + p(x_4)$ and $p(x_3) + p(x_5) = 0$. We also have the relation \[ p(x_4) p(x_4^\prime) = p(x_3) p(x_7) + p(x_6) p(x_8) = p(x_6) p(x_8) - p(x_5) p(x_7). \] 
	
	\begin{enumerate}
		\item If $p(x_4) \neq 0$, then $p(x_4)$ and the values of the boundary variables determine the values of both $x_4^\prime$ and $x_3^\prime$ (with both non-zero). Thus, such a deep point is of the form
		\begin{align*}
			p(x_1, x_2^\vee, x_1^\vee, x_2, x_3, x_3^\prime, x_4, x_4^\prime) &= (0,0,0,0,-a, -u, u, \frac{bd-ac}{u}) \\
			p(x_5, x_6, x_7, x_8) &= (a, b, c, d)
		\end{align*}
		for some $u,a,b,c,d \in \Bbbk^\times$. These points determine a deep subset of $V(\CA(\SS),\Bbbk)$ isomorphic to $(\Bbbk^\times)^5$. 
		
		\item On the other hand, if $p(x_4) = 0$, then $p(x_3^\prime) = 0$ and \[p(x_6) p(x_8) - p(x_5) p(x_7) = 0,\] in which case $p(x_4^\prime)$ can take any value.
		\begin{align*}
			p(x_1, x_2^\vee, x_1^\vee, x_2, x_3, x_3^\prime, x_4, x_4^\prime) &= (0,0,0,0,-a, 0, 0, u) \\
			p(x_5, x_6, x_7, x_8) &= (a, b, c, \frac{ac}{b})
		\end{align*}
		for some $a,b,c \in \Bbbk^\times$ and $u\in \Bbbk$. These points determine a deep subset of $V(\CA(\SS),\Bbbk)$ isomorphic to $\Bbbk \times (\Bbbk^\times)^3$.
	\end{enumerate}
	
	Further, as a result of Proposition \ref{prop: cutting totally radial}, we know that the union of these deep subsets is isomorphic to the cluster variety $V(\CA(\Delta_4),\Bbbk)$.
	
	\item Second, suppose that $p$ kills all notched radii, but not all plain radii. These points can be divided into two types: those that kill only the odd-indexed plain radii and those that kill only the even-indexed plain radii. 
	
	By Porism \ref{porism: boundary alternating product}, we have $p(x_6) p(x_8) = p(x_5) p(x_7)$. Further, because $p$ kills all notched radii, we have $p(x_3) + p(x_5) = 0$. Finally, by Proposition \ref{prop: non-totally radial deep-on-components}, $p$ restricts to a deep point on any unpunctured quadrilateral $\SS\smallsetminus c_i$. By Theorem \ref{thm: deeppolygon}, $p$ must kill all non-radial arcs of even length; as a result, $p$ kills all of $x_3^\prime$, $x_4$, and $x_4^\prime$.
	
	\begin{enumerate}
		\item Suppose that $p$ kills only the odd-indexed plain radii. Such a deep point is of the form
		\begin{align*}
			p(x_1, x_2^\vee, x_1^\vee, x_2, x_3, x_3^\prime, x_4, x_4^\prime) &= (0,0,0,u,-a, 0, 0, 0) \\
			p(x_5, x_6, x_7, x_8) &= (a, b, c, \frac{ac}{b})
		\end{align*}
		for some $u,a,b,c \in \Bbbk^\times$. These deep points determine a subset of $V(\CA(\SS),\Bbbk)$ isomorphic to $(\Bbbk^\times)^4$.
		
		\item Suppose that $p$ kills only the even-indexed plain radii. Such a deep point is of the form
		\begin{align*}
			p(x_1, x_2^\vee, x_1^\vee, x_2, x_3, x_3^\prime, x_4, x_4^\prime) &= (u,0,0,0,-a, 0, 0, 0) \\
			p(x_5, x_6, x_7, x_8) &= (a, b, c, \frac{ac}{b})
		\end{align*}
		for some $u,a,b,c \in \Bbbk^\times$. These deep points determine a subset of $V(\CA(\SS),\Bbbk)$ isomorphic to $(\Bbbk^\times)^4$.
	\end{enumerate}
	
	\item Third, suppose that $p$ kills all plain radii, but not all notched radii. These points can be divided into two types: those that kill only the odd-indexed notched radii and those that kill only the even-indexed notched radii. 
	
	As in the previous case, Porism \ref{porism: boundary alternating product} implies that $p(x_6) p(x_8) - p(x_5) p(x_7) = 0$. Further, because $p$ kills all plain radii, we have $p(x_3) + p(x_5) = 0$. Finally, by Proposition \ref{prop: non-totally radial deep-on-components}, $p$ restricts to a deep point on any unpunctured quadrilateral $\SS\smallsetminus c_i$. By Theorem \ref{thm: deeppolygon}, $p$ must kill all non-radial arcs of even length; as a result, $p$ kills all of $x_3^\prime$, $x_4$, and $x_4^\prime$.
	
	\begin{enumerate}
		\item Suppose that $p$ kills only the odd-indexed notched radii. Such a deep point is of the form
		\begin{align*}
			p(x_1, x_2^\vee, x_1^\vee, x_2, x_3, x_3^\prime, x_4, x_4^\prime) &= (0,u,0,0,-a, 0, 0, 0) \\
			p(x_5, x_6, x_7, x_8) &= (a, b, c, \frac{ac}{b})
		\end{align*}
		for some $u,a,b,c \in \Bbbk^\times$. These deep points determine a subset of $V(\CA(\SS),\Bbbk)$ isomorphic to $(\Bbbk^\times)^4$.
		
		\item Suppose that $p$ kills only the even-indexed notched radii. Such a deep point is of the form
		\begin{align*}
			p(x_1, x_2^\vee, x_1^\vee, x_2, x_3, x_3^\prime, x_4, x_4^\prime) &= (0,0,u,0,-a, 0, 0, 0) \\
			p(x_5, x_6, x_7, x_8) &= (a, b, c, \frac{ac}{b})
		\end{align*}
		for some $u,a,b,c \in \Bbbk^\times$. These deep points determine a subset of $V(\CA(\SS),\Bbbk)$ isomorphic to $(\Bbbk^\times)^4$.
	\end{enumerate}
\end{enumerate}

Combining the above, we have the following result:

\begin{prop}\label{prop: D4 with boundary}
	Let $\SS$ be a once-punctured disk with $4$ marked points on the boundary. Then the deep locus of $V(\CA(\SS),\Bbbk)$ is comprised of four subsets isomorphic to $(\Bbbk^\times)^4$, one subset isomorphic to $\Bbbk \times (\Bbbk^\times)^3$, and one subset isomorphic to $(\Bbbk^\times)^5$. The intersection of the closures of these six sets is a three-dimensional algebraic torus consisting of deep points of the form 
	\begin{align*}
		p(x_1, x_2^\vee, x_1^\vee, x_2, x_3, x_3^\prime, x_4, x_4^\prime) &= (0,0,0,0,-a, 0, 0, 0) \\
		p(x_5, x_6, x_7, x_8) &= (a, b, c, \frac{ac}{b})
	\end{align*}
	for $a,b,c \in \Bbbk^\times$.
\end{prop}

\begin{rem}\label{rem: D4 deep}
	If we specialize the frozen boundary variables to 1, the deep locus is comprised of six lines which intersect at a single point. We have the constraint \[p(x_6) p(x_8) - p(x_5) p(x_7) = 1 - 1 = 0,\] so away from the point of intersection, each deep point is of one of the forms 
	\begin{align*}
		p(x_1, x_2^\vee, x_1^\vee, x_2, x_3, x_3^\prime, x_4, x_4^\prime) &= (0,0,0,0,-1, -u, u, 0), \\
		p(x_1, x_2^\vee, x_1^\vee, x_2, x_3, x_3^\prime, x_4, x_4^\prime) &= (0,0,0,0,-1, 0, 0, u), \\
		p(x_1, x_2^\vee, x_1^\vee, x_2, x_3, x_3^\prime, x_4, x_4^\prime) &= (0,0,0,u,-1, 0, 0, 0), \\
		p(x_1, x_2^\vee, x_1^\vee, x_2, x_3, x_3^\prime, x_4, x_4^\prime) &= (u,0,0,0,-1, 0, 0, 0), \\
		p(x_1, x_2^\vee, x_1^\vee, x_2, x_3, x_3^\prime, x_4, x_4^\prime) &= (0,u,0,0,-1, 0, 0, 0), \text{ or }\\
		p(x_1, x_2^\vee, x_1^\vee, x_2, x_3, x_3^\prime, x_4, x_4^\prime) &= (0,0,u,0,-1, 0, 0, 0),
	\end{align*}
	and the point of intersection is the point \[(0,0,0,0,-1,0,0,0,1,1,1,1) \in \Bbbk^8 \times (\Bbbk^\times)^4.\] 
	
	An abbreviated version of this specialized result for the $D_4$ cluster variety previously appeared in \cite[Remark 6.11]{CGSS26}, calculated using a different initial seed. We should also remark that this specialized deep locus coincides exactly with the singular locus of $\CA(D_4)$ in part (4)(a) of \cite[Theorem A]{BFMS23}. 
	
	We can also characterize deep points of $V(\CA(D_4), \Bbbk)$ topologically as one of seven types:
	\begin{enumerate}
		\item those that spare the even-indexed plain radii,
		\item those that spare the odd-indexed plain radii, 
		\item those that spare the even-indexed notched radii, 
		\item those that spare the odd-indexed notched radii, 
		\item those that spare the diagonals connecting even-indexed vertices, 
		\item those that spare the diagonals connecting odd-indexed vertices, and  
		\item the point that does not spare any diagonals or radii.
	\end{enumerate}
	Note that all deep points spare the $c_i$, the arcs that cut out once-punctured bigons along with boundary arcs.
\end{rem}

\begin{warn}
	Although Proposition \ref{prop: cutting totally radial} implies that the subset corresponding to totally radial deep points is isomorphic to $V(\CA(\Delta_4),\Bbbk)$, after specializing the boundary variables to 1, the subvariety of totally-radial points is not exactly $V(\CA(A_1),\Bbbk)$. We are specializing the boundary values of the punctured disk, so the value of $c_i$ is -1, and the unpunctured polygon we get after cutting does not have boundary values all equal to 1. 
\end{warn}

%============================

\subsection{Once-punctured disks with at least five boundary marked points}

We now proceed with the explicit computation of deep points for once-punctured disks with $n\geq 5$ marked points on the boundary. 
Using the quiver in Figure \ref{fig: punctured ngon} as our initial seed, the cluster algebra is generated by $x_1$, $x_2^\vee$, $x_1^\vee$, $x_2$, $x_3$, $x_3^\prime, \dotsc, x_n$, $x_n^\prime$, along with the frozen variables and their inverses, where $x_2$ denotes the plain radius realized by a flip of $x_1^\vee$ in the given triangulation. The ideal of relations is generated by:
\begin{align*}
	x_1^{} x_2^\vee &= x_3^{} + x_{n+1}^{} \quad & x_4^{} x_4^\prime &= x_3^{} x_{n+3}^{} + x_5^{} x_{n+2}^{} \\ 
	x_1^\vee x_2^{} &= x_3^{} + x_{n+1}^{} \quad & &\phantom{a}\vdots\\
	x_3^{} x_3^\prime &= x_1^{} x_1^\vee x_{n+2}^{} + x_4^{} x_{n+1}^{} \quad & x_{n-1}^{} x_{n-1}^\prime &= x_{n-2}^{} x_{2n-2}^{} + x_n^{} x_{2n-3}^{} \\
	x_n^{} x_n^\prime &= x_{n-1}^{} x_{2n-1}^{} + x_{2n}^{} x_{2n-2}^{} \quad & &
\end{align*}

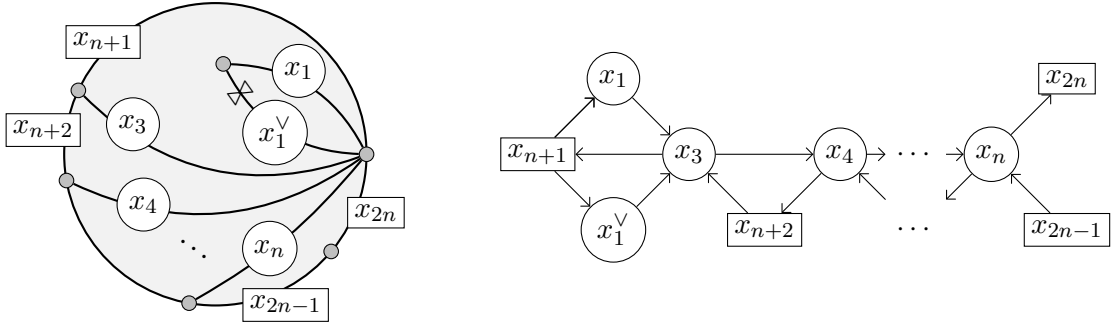
\begin{figure}[h!tb]
	\centering
	\begin{subfigure}[b]{0.4\textwidth}
		\centering
		\begin{tikzpicture}[scale=2,baseline={(0,0)}]
			\draw[thick,fill=black!5] (0,0) circle (1);
			\node[dot] (2) at (155:1) {};
			\node[dot] (3) at (190:1) {};
			\node[dot] (1n) at (260:1) {};
			\node[dot] (n) at (320:1) {};
			\node[dot] (1) at (0:1) {};
			\node[dot] (p) at (85:0.6) {};
			\node[frozen] (xn) at (135:1.05) {$x_{n+1}$};
			\node[frozen] (xn1) at (172:1.15) {$x_{n+2}$};
			\node[frozen] (x2n1) at (295:1.1) {$x_{2n-1}$};
			\node[frozen] (x2n) at (340:1.13) {$x_{2n}$};
			\draw[thick,bend right] (1) to (p) {};
			\draw[thick,bend left] (1) to node[pos=0.85,sloped,rotate=90]{$\bowtie$} (p) {};
			\draw[thick,out=200,in=315] (1) to (2) {};
			\draw[thick,bend left] (1) to (3) {};
			\draw[thick,out=230,in=30] (1) to (1n) {};
			\node[mutable3] at (45:0.78) {$x_1$};
			\node[mutable3] at (20:0.42) {$x_1^\vee$};
			\node[mutable3] at (160:0.58) {$x_3$};
			\node[mutable3] at (215:0.58) {$x_4$};
			\node[mutable3] at (300:0.72) {$x_n$};
			\node at (256:0.59) {$\ddots$};
		\end{tikzpicture}
	\end{subfigure}
	\begin{subfigure}[b]{0.55\textwidth}
		\centering
		\begin{tikzpicture}[scale=1,baseline={(0,0)}]
			\node[frozen] (x5) at (-3,0) {$x_{n+1}$};
			\node[mutable] (x1) at (-2,1) {$x_1$};
			\node[mutable] (x2) at (-2,-1) {$x_1^\vee$};
			\node[mutable] (x3) at (-1,0) {$x_3$};
			\node[frozen] (x6) at (0,-1) {$x_{n+2}$};
			\node[mutable] (x4) at (1,0) {$x_4$};
			\node (d1) at (2,0) {$\cdots$};
			\node (d2) at (2,-1) {$\cdots$};
			\node[mutable] (xn) at (3,0) {$x_n$};
			\node[frozen] (x7) at (4,-1) {$x_{2n-1}$};
			\node[frozen] (x8) at (4,1) {$x_{2n}$};
			\draw[-angle 90] (x5) to (x1);
			\draw[-angle 90] (x5) to (x2);
			\draw[-angle 90] (x1) to (x3);
			\draw[-angle 90] (x2) to (x3);
			\draw[-angle 90] (x5) to (x1);
			\draw[-angle 90] (x3) to (x5);
			\draw[-angle 90] (x3) to (x4);
			\draw[-angle 90] (x4) to (x6);
			\draw[-angle 90] (x4) to (1.6,0) {};
			\draw[-angle 90] (1.6,-0.6) to (x4) {};
			\draw[-angle 90] (2.4,0) to (xn) {};
			\draw[-angle 90] (xn) to (2.4,-0.6) {};
			\draw[-angle 90] (x6) to (x3);
			\draw[-angle 90] (xn) to (x8);
			\draw[-angle 90] (x7) to (xn);
		\end{tikzpicture}
	\end{subfigure}
	\caption{A triangulation of a once-punctured $n$-gon and its associated quiver.}
	\label{fig: punctured ngon}
\end{figure}

As stated in Theorem \ref{thm: punctured disk deep types}, if $n$ is odd, then the only deep points are those that kill all radii of both taggings \textemdash the totally radial deep points. If $n$ is even, then we will have all three types of deep points, which we will consider separately. Throughout the following, let $p\in V(\CA(\SS),\Bbbk)$ be deep. 

\begin{enumerate}
	\item First, suppose that $p$ is a totally radial deep point; that is, it kills all radii of both taggings. Since all radii are killed, we have $p(x_3) + p(x_{n+1}) = 0$ and $p(x_3^\prime) + p(x_4) = 0$.
	
	Proposition \ref{prop: cutting totally radial} implies that the totally radial points comprise a subvariety isomorphic to $V(\CA(\Delta_n),\Bbbk)$. Cutting along $c_i$ disconnects the disk into a once-punctured bigon and an unpunctured $n$-gon. As the totally radial deep points do not have a requirement on the boundary values, $c_i$ can take on any non-zero value, and the value of $b_i$ is determined by the value of $c_i$. The restriction to the unpunctured $n$-gon need not be deep, so the remaining arcs can take on any values satisfying the Ptolemy relations. 
    Choosing the values of the initial $n-3$ diagonal arcs of the unpunctured $n$-gon $\SS\smallsetminus c_i$, in addition to the $n$ boundary arcs, determines the values on the other generators via the generating relations. Thus, the subvariety of $V(\CA(\SS), \Bbbk)$ corresponding to totally radial points is of dimension $2n-3$. 
	
	\item Second, suppose that $n$ is even and $p$ is a notched deep point. In this case, $p$ kills alternating plain radii \textemdash either all odd-indexed plain radii or all even-indexed plain radii \textemdash and the values of the boundary edges are subject to Equation \ref{eq: boundary alternating product}.
	
	\begin{enumerate}
		\item Suppose that $p$ kills only the odd-indexed plain radii. Then $p(x_2) \neq 0$, and $p(x_3) = -p(x_{n+1}) \neq 0$, so we can cut along the arc labeled $x_3$. From Proposition \ref{prop: non-totally radial deep-on-components}, $p$ restricts to a deep point on the unpunctured $n$-gon which is determined by its values on the boundary arcs. Since the boundary arcs must satisfy Equation \ref{eq: boundary alternating product}, we can choose non-zero values for $n-1$ of the boundary arcs and the last value is determined. Along with the choice of non-zero value for $x_2$, this gives us a total of $n$ choices. Thus, these deep points constitute a subset of the deep locus of $V(\CA(\SS),\Bbbk)$ isomorphic to $(\Bbbk^\times)^n$.
		
		\item Next, suppose that $p$ kills only the even-indexed plain radii. Then $p(x_1) \neq 0$, and $p(x_3) = -p(x_{n+1}) \neq 0$, so we can cut along the arc labeled $x_3$. Again, $p$ restricts to a deep point on the unpunctured $n$-gon which is determined by its values on the boundary arcs. Thus, these deep points likewise constitute a subset of the deep locus of $V(\CA(\SS),\Bbbk)$ isomorphic to $(\Bbbk^\times)^n$.
	\end{enumerate}
	
	\item Third, suppose that $n$ is even and $p$ is a plain deep point; that is, it kills all plain radii, but not all notched radii. In this case, $p$ kills alternating notched radii \textemdash either all odd-indexed notched radii or all even-indexed notched radii \textemdash and the values of the boundary edges are subject to Equation \ref{eq: boundary alternating product}.
	
	\begin{enumerate}
		\item Suppose that $p$ kills only the odd-indexed notched radii. Then $p(x_2^\vee) \neq 0$, and $p(x_3) = -p(x_{n+1}) \neq 0$, so we can cut along the arc labeled $x_3$. From Proposition \ref{prop: non-totally radial deep-on-components}, $p$ restricts to a deep point on the unpunctured $n$-gon which is determined by its values on the boundary arcs. Since the boundary arcs must satisfy Equation \ref{eq: boundary alternating product}, we can choose non-zero values for $n-1$ of the boundary arcs and the last value is determined. Along with the choice of non-zero value for $x_2^\vee$, this gives us a total of $n$ choices. Thus, these deep points constitute a subset of the deep locus of $V(\CA(\SS),\Bbbk)$ isomorphic to $(\Bbbk^\times)^n$.
		
		\item Suppose that $p$ kills only the even-indexed notched radii. Then $p(x_1^\vee) \neq 0$, and $p(x_3) = -p(x_{n+1}) \neq 0$, so we can cut along the arc labeled $x_3$. Once again, $p$ restricts to a deep point on the unpunctured $n$-gon which is determined by its values on the boundary arcs. Thus, these deep points constitute a subset of the deep locus of $V(\CA(\SS),\Bbbk)$ isomorphic to $(\Bbbk^\times)^n$.
	\end{enumerate}
	
	For both plain and notched radial deep points, cutting along $c_i$ disconnects the disk into a once-punctured bigon and an unpunctured $n$-gon, so we have the relation \[ \frac{b_{i+1} \dotsm b_{i+n-1}}{c_i b_{i+2} \dotsm b_{i+n-2}} = (-1)^{\frac{n+2}{2}}. \]
	Thus, for each fixed choice of non-zero value on the spared radius, we have a subvariety of $V(\CA(\SS \smallsetminus c_i), \Bbbk)$ isomorphic to the deep locus of $V(\CA(\Delta_n), \Bbbk)$, which is of dimension $n-1$. For $n \geq 4$, any deep point $p$ will kill three of the four radii in the punctured bigon, while the fourth can take on any non-zero value. Thus, we have four deep subsets of $V(\CA(\SS),\Bbbk)$, each isomorphic to $(\Bbbk^\times)^n$.
\end{enumerate}

Combining with Propositions \ref{prop: cutting totally radial} and \ref{prop: D4 with boundary}, we have the following.
\begin{thm}\label{thm: Dn with boundary}
	Let $\SS$ be a once-punctured disk with $n\geq 4$ marked points on the boundary. 
	\begin{enumerate}
		\item If $n$ is odd, then the deep locus of $V(\CA(\SS),\Bbbk)$ is a single component of dimension $2n-3$.
		\item If $n$ is even, then the deep locus of $V(\CA(\SS),\Bbbk)$ contains one component of dimension $2n-3$ and four distinct subsets isomorphic to $(\Bbbk^\times)^n$. The intersection of the closures of these five sets is the $(n-1)$-dimensional subvariety consisting of deep points of the form 
		\[ p(x_1, x_2^\vee, x_1^\vee, x_2, x_3, x_3^\prime, x_4, x_4^\prime,\dotsc, x_n, x_n^\prime) = (0,0,0,0,-b_1, 0, 0, 0,\dotsc, 0,0) \]
		\[ p(x_{n+1}, x_{n+2}, \dotsc, x_{2n-1}, x_{2n}) = \left( b_1, b_2, \dotsc, b_{n-1}, (-1)^{n/2}~\frac{b_1 b_3 \dotsm b_{n-1}}{b_2 b_4 \dotsm b_{n-2} } \right) \]
		for all $b_1, \dotsc, b_{n-1} \in \Bbbk^\times$.
	\end{enumerate}
\end{thm}

\begin{rem}\label{rem: Dn deep}
	If we specialize the frozen boundary variables to 1, then Equation \ref{eq: boundary alternating product} is satisfied only when $n \equiv 0 \bmod{4}$, or if $\Char (\Bbbk) = 2$ and $n \equiv 0 \bmod{2}$. Thus, we have the following description of the deep locus of $V(\CA(D_n),\Bbbk)$:
	\begin{enumerate}
		\item If $n \equiv 0 \bmod{4}$ or $\Char (\Bbbk) = 2$ and $n \equiv 0\bmod{2}$, then the deep locus is comprised of one $(n-3)$-dimensional component and four copies of $\Bbbk^\times$. The intersection of the closures of these sets is the point 
		\[ p(x_1, x_2^\vee, x_1^\vee, x_2, x_3, x_3^\prime, x_4, x_4^\prime,\dotsc, x_n, x_n^\prime) = (0,0,0,0,-1, 0, 0, 0,\dotsc, 0,0) \]
		\[ p(x_{n+1}, x_{n+2}, \dotsc, x_{2n-1}, x_{2n}) = (1, 1, \dotsc, 1, 1) \]
		in $\Bbbk^{2n} \times (\Bbbk^\times)^n$. 
		
		\item Otherwise, the deep locus is comprised of a single $(n-3)$-dimensional component.
	\end{enumerate}
	
	\noindent
	This coincides exactly with the singular locus of $\CA(D_n)$ in parts (4)(b) and (4)(c) of \cite[Theorem A]{BFMS23}. 
\end{rem}

%============================

\section{Marked disks with multiple punctures}\label{section: deep punctured disks}

In this section, we will classify and parameterize the deep points of cluster varieties associated to disks with multiple punctures. We will assume throughout that our surface is connected, orientable, and triangulable, and will take care to specify the genus, number of boundary components, and number and configuration of marked points at each step. 

Recall that a (tagged) radius is a simple arc with one end at the boundary and the other end at a puncture (tagged plain or notched). In contrast to the once-punctured disks we discussed in Section \ref{section: 1 punctured disks}, in a surface with multiple punctures, not every simple arc emanating from a puncture is a radius. We refer to arcs with both ends at punctures as \textbf{tunnels}. While every triangulation includes at least two arcs emanating from each puncture, a triangulation may not include radii emanating from each puncture, as demonstrated in Figure \ref{fig: no radii at puncture}; 

\begin{figure}[htb]
	\centering
	\begin{tikzpicture}[scale=2,baseline={(0,0)}]
		\draw[thick,fill=black!5] (0,0) circle (1);
		\node[dot] (1) at (90:1) {};
		\node[dot] (2) at (270:1) {};
		\node[dot] (3) at (45:0.5) {};
		\node[dot] (4) at (180:0.2) {};
		\node[dot] (5) at (315:0.5) {};
		\draw[thick] (1) to (3);
		\draw[thick] (3) to (4);
		\draw[thick] (4) to (5);
		\draw[thick] (5) to (3);
		\draw[thick] (5) to [out=180,in=270] (180:0.5) to [out=90,in=180] (3);
		\draw[thick] (1) to [out=225,in=90] (180:0.65) to [out=270,in=225] (5);
		\draw[thick] (2) to [out=40,in=270] (345:0.75) to [out=90,in=315] (3);
		\draw[thick] (5) to (2);
	\end{tikzpicture}
	\caption{A triangulation of a thrice-punctured bigon where not every puncture is connected to the boundary by a radius.}
	\label{fig: no radii at puncture}
\end{figure}
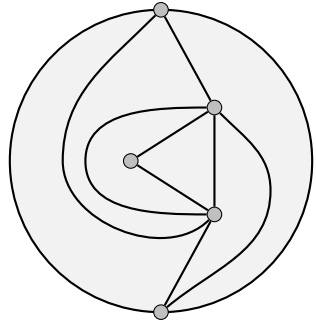

If a deep point $p$ kills all arcs emanating from a puncture $\xi$, we will say that $\xi$ is \textbf{quashed} by $p$. If all punctures are quashed by $p$, we will refer to $p$ as a \textbf{malevolent} deep point; on the other hand, if $p$ does not quash any punctures, we will refer to it as \textbf{benevolent}. Note that because every triangulation will include multiple arcs emanating from each puncture, if any one puncture is quashed by $p$, that is enough to guarantee the deepness of $p$ and we should not place any a priori expectations on what happens at the other punctures. 

\begin{ex}
    In the case of a once-punctured disk, the totally radial points quash the lone puncture, so they are malevolent. The non-totally radial points are benevolent.
\end{ex}

Another difference from the case of once-punctured disks is that a disk with $q\geq 2$ punctures and a single marked point on the boundary is triangulable, as demonstrated in Figure \ref{fig: disk with one marked point}. We will divide our discussion of punctured disks into two parts \textemdash those disks with $n\geq 2$ marked points on the boundary and those with only $n=1$ marked point on the boundary.

\begin{figure}[htb]
	\centering
	\begin{tikzpicture}[scale=2,baseline={(0,0)}]
		\draw[thick,fill=black!5] (0,0) circle (1);
		\node[dot] (1) at (270:1) {};
		\node[dot] (3) at (0:0.7) {};
		\node[dot] (4) at (180:0.2) {};
		\draw[thick] (1) to (3);
		\draw[thick] (1) to (4);
		\draw[thick] (3) to (4);
		\draw[thick] (1) to [out=150,in=225] (150:0.6) to [out=45,in=120] (3);
	\end{tikzpicture}
	\caption{A triangulation of a twice-punctured monogon.}
	\label{fig: disk with one marked point}
\end{figure}
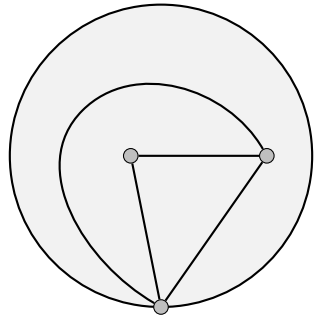

Before moving on to the separate cases, let us prove some results that will apply in both settings.

%============================

\subsection{Cutting along radii}

First, we want to examine what happens when we cut along a radius. Suppose that a deep point $p$ assigns a non-zero value to a radius $r$. Then we can undo the tagging on $r^\vee$, as demonstrated in Figure \ref{fig: cutting radius 2 step}, and cut along $r$ from the boundary up to the puncture. Note that this creates not just two copies of the radius $r$, but also two copies of the boundary marked point.

Observe that if the notched version of that radius $r^\vee$ is instead non-zero, we can apply an isomorphism to the cluster algebra that switches the taggings at a given puncture before performing the cut. 
We will ignore this step in most the following arguments and simply assume that the radius being cut has a plain tagging, but we will need to remember to consider the taggings when counting the number of deep components.

\begin{figure}
	\centering
	$
	\begin{tikzpicture}[xscale=1,scale=1.7,baseline={(0,0)}]
		\draw[thick,fill=black!5] (0,0) circle (1);
		\node[dot] (3) at (0:1) {};
		\node[dot] (p) at (180:0.2) {};
		\draw[thick,bend right] (3) to node[pos=0.85,sloped,rotate=90]{$\bowtie$} (p) {};
		\draw[thick,bend left] (3) to (p) {};
		\node[mutable3] (r1) at (340:0.6) {$r$};
		\node[mutable3] (r2) at (30:0.6) {$r^\vee$};
	\end{tikzpicture}
	\quad
	\longrightarrow
	\quad
	\begin{tikzpicture}[xscale=1,scale=1.7,baseline={(0,0)}]
		\draw[thick,fill=black!5] (0,0) circle (1);
		\node[dot] (3) at (0:1) {};
		\node[dot] (p) at (180:0.2) {};
		\draw[thick] (3) to [out=225,in=270] (180:0.4) to [out=90,in=135] (3) {};
		\draw[thick] (3) to (p) {};
		\node[mutable3] (r1) at (0:0.4) {$r$};
		\node[mutable3] (r2) at (60:0.6) {$r^\vee$};
	\end{tikzpicture}
	\quad
	\longrightarrow
	\quad
	\begin{tikzpicture}[xscale=1,scale=1.7,baseline={(0,0)}]
		\draw[thick,fill=black!5] (0,0) circle (1);
		\node[dot] (3) at (0:1) {};
		\node[dot] (4) at (45:1) {};
		\node[dot] (5) at (315:1) {};
		\draw[thick] (5) to [out=150,in=270] (0:0.2) to [out=90,in=210] (4) {};
		\node[frozen] (r1) at (22:1.05) {$r^{\prime \prime}$};
		\node[frozen] (r2) at (338:1.05) {$r^\prime$};
		\node[mutable3] (r3) at (0:0.2) {$r^\vee$};
	\end{tikzpicture}
	$
	\caption{Cutting along a radius}
	\label{fig: cutting radius 2 step}
\end{figure}
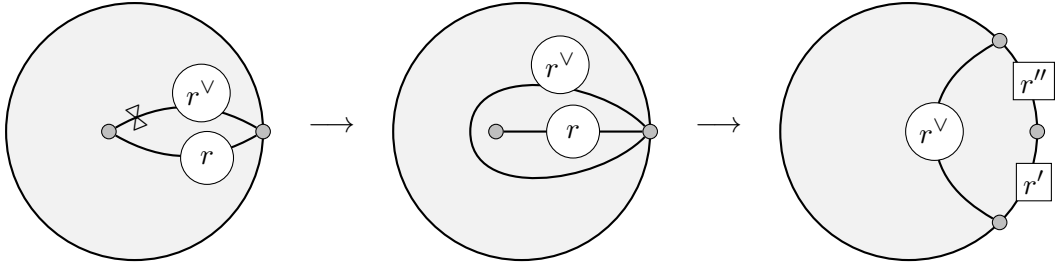

First, a proposition that we could have proven in the previous section: 

\begin{prop}\label{prop: benevolent bijection 1 punc}
	Let $\SS$ be a disk with $n\geq 2$ marked points on the boundary and $q=1$ punctures. There is a bijection 
	\[ 
	\begin{Bmatrix} \text{non-totally radial deep} \\ \text{points in } V(\CA(\SS),\Bbbk) \\ \text{non-zero on a radius }r_i \end{Bmatrix} 
	\longleftrightarrow
	\begin{Bmatrix}\text{deep points in } \\ V(\CA(\SS\smallsetminus r_i),\Bbbk) \\ \text{gluable along } r_i \end{Bmatrix}
	\]    
\end{prop}
\begin{proof}
	Suppose that $p \in V(\CA(\SS),\Bbbk)$ is a deep point such that $p$ spares some radius $r_i$, which connects the puncture $\xi_i$ to the boundary marked point $m_i$. From Proposition \ref{prop: deep cutting inclusion}, we know that its image under the cutting isomorphism $\bar{p} \in V(\CA(\SS\smallsetminus r_i),\Bbbk)$ must be deep as well.
	
	Now, suppose that $\bar{p} \in V(\CA(\SS\smallsetminus r_i),\Bbbk)$ is a deep point that is gluable along $r_i$. That is, two adjacent boundary arcs are assigned the same value $p(r_i)$; denote the marked point between these two boundary arcs as $\xi_i^\prime$.
	
	Since $\SS\smallsetminus r_i$ is an unpunctured $(n+2)$-gon, Theorem \ref{thm: deeppolygon} says that $n$ must be even and the values on the boundary arcs must satisfy \[ \frac{b_1 \dotsm b_{n-1}\, p(r_i)}{b_2 \dotsm b_n\, p(r_i)} ~=~ \frac{b_1 \dotsm b_{n-1}}{b_2 \dotsm b_n} ~=~ (-1)^{\frac{(n+2)+2}{2}} ~=~ (-1)^{n/2}. \]
	By the same theorem, we see that the diagonals from boundary marked points to $\xi_i^\prime$ alternate between zero and non-zero values, while all length 2 arcs that form triangles with two boundary arcs are killed. In particular, the image of $r_i^\vee$ is killed. 
	
	Recall that every triangulation of $\SS$ will contain at least two radii. Suppose that $T$ is a triangulation of $\SS$ containing two radii spared by $p$. These radii are arcs in $\SS\smallsetminus r_i$ and are both spared by $\bar{p}$. By Lemma \ref{lemma: trianglesurf}, the arc completing a triangle with these arcs must be killed by $\bar{p}$, and its gluing in $\SS$ is killed by $p$. Thus, $p$ kills at least one arc in every triangulation of $\SS$. Hence, $p \in V(\CA(\SS),\Bbbk)$ is a non-totally radial deep point.
\end{proof}

We already know from Theorem \ref{thm: punctured disk deep types} that if $p$ is not a totally radial deep point, there is a radius that is spared by $p$, so the above holds for all non-totally radial deep points. Before extending this Proposition to allow for multiple punctures, we want to show that a benevolent deep point always spares at least one radius.

\begin{figure}[htb]
	\centering
	$
	\begin{tikzpicture}[scale=2,baseline={(0,0)}]
		\draw[thick,fill=black!5] (0,0) circle (1);
		\node[dot] (1) at (270:1) {};
		\node[dot] (3) at (0:0.33) {};
		\node[dot] (4) at (180:0.33) {};
		\node[dot] (5) at (130:0.5) {};
		\node[dot] (6) at (50:0.5) {};
		\node[dot] (7) at (270:0.3) {};
		\draw[thick,dark red] (4) to [out=45,in=300] (100:0.7) to [out=120,in=90] (180:0.8) to [out=270,in=150] (1);
		\draw[thick,dark red] (3) to [out=135,in=240] (80:0.7) to [out=60,in=90] (0:0.8) to [out=270,in=30] (1);
		\draw[thick] (3) to (4);
	\end{tikzpicture}
	\quad\quad\quad
	\begin{tikzpicture}[scale=2,baseline={(0,0)}]
		\draw[thick,fill=black!5] (0,0) circle (1);
		\node[dot] (1) at (225:1) {};
		\node[dot] (2) at (75:1) {};
		\node[dot] (3) at (0:0.33) {};
		\node[dot] (4) at (180:0.33) {};
		\node[dot] (5) at (130:0.5) {};
		\node[dot] (6) at (50:0.5) {};
		\node[dot] (7) at (270:0.3) {};
		\draw[thick,dark red] (4) to [out=45,in=300] (105:0.7) to [out=120,in=90] (180:0.8) to [out=270,in=100] (1);
		\draw[thick,dark red] (3) to [out=135,in=240] (90:0.7) to [out=60,in=225] (2);
		\draw[thick] (3) to (4);
	\end{tikzpicture}
	$
	\caption{Unpunctured quadrilaterals with a tunnel opposite a boundary component.}
	\label{fig: spared tunnel}
\end{figure}
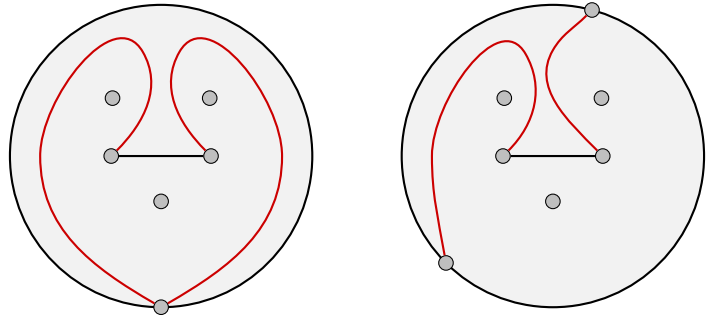

\begin{lemma}\label{lemma: tunnel quadrilateral}
	Let $\SS$ be a disk with $n\geq 1$ marked points on the boundary and $q\geq 2$ punctures. For any given tunnel $\tau$, there is an unpunctured quadrilateral in $\SS$ bounded by one boundary arc, two radii, and the tunnel $\tau$.
\end{lemma}
\begin{proof}
	As in Figure \ref{fig: spared tunnel}, take two radii emanating from opposite ends of a tunnel and have them travel together in parallel between any other punctures toward the boundary. Before reaching the boundary, have them travel in opposite directions along the boundary until they each reach a marked point. 
\end{proof}

\begin{lemma}\label{lemma: benevolent spares a radius}
	Let $\SS$ be a disk with $n\geq 1$ marked points on the boundary and $q\geq 2$ punctures, and let $p \in V(\CA(\SS),\Bbbk)$ be a benevolent deep point. Then $p$ must spare at least one radius.
\end{lemma}
\begin{proof}
	Seeking a contradiction, suppose that $p$ kills all radii. Since $p$ is benevolent by assumption, $p$ must spare a tunnel $\tau$. By Lemma \ref{lemma: tunnel quadrilateral}, there exists a quadrilateral consisting of two radii, a boundary component $b$, and $\tau$, with $\tau$ opposite $b$ as in Figure \ref{fig: spared tunnel}. The Ptolemy relation on such a quadrilateral is then $0 = p(\tau) p(b) + 0$, but both $p(b)$ and $p(\tau)$ are non-zero by assumption. This is a contradiction, so $p$ must spare at least one radius.
\end{proof}

And now, we extend Proposition \ref{prop: benevolent bijection 1 punc} to the current setting. 

\begin{prop}\label{prop: benevolent bijection multiple punctures}
	Let $\SS$ be a disk with $n\geq 1$ marked points on the boundary and $q\geq 2$ punctures. There is a bijection 
	\[ 
	\begin{Bmatrix} \text{benevolent deep points} \\ \text{in } V(\CA(\SS),\Bbbk) \\ \text{non-zero on }r_i \end{Bmatrix} 
	\longleftrightarrow
	\begin{Bmatrix}\text{benevolent deep points} \\ \text{in } V(\CA(\SS\smallsetminus r_i),\Bbbk) \\ \text{gluable along }r_i \end{Bmatrix}
	\]    
\end{prop}
\begin{proof}
	Suppose that $p \in V(\CA(\SS),\Bbbk)$ is a benevolent deep point such that $p$ spares the radius $r_i$, with ends at the puncture $\xi_i$ and marked point $m_i$.  From Proposition \ref{prop: deep cutting inclusion}, we know that its image under the cutting isomorphism $\bar{p} \in V(\CA(\SS\smallsetminus r_i),\Bbbk)$ must be deep as well. Since $p$ does not quash any punctures, $\bar{p}$ will not quash any remaining puncture after cutting, so $\bar{p}$ is a benevolent deep point on $\SS\smallsetminus r_i$.
	
	To prove the opposite direction, we use induction on the number of punctures. For the following, assume that $\bar{p} \in V(\CA(\SS\smallsetminus r_i),\Bbbk)$ is a benevolent deep point that is gluable along $r_i$.
	\begin{itemize}
		\item First, suppose that $q=2$. Since $\SS\smallsetminus r_i$ is a once-punctured $(n+2)$-gon, Theorem \ref{thm: punctured disk deep types} says that since $p$ is non-zero on $r_i$, it must kill alternating plain radii and all notched radii, and the boundary values satisfy \[ \frac{b_1 \dotsm b_{n-1}\, r_i}{b_2 \dotsm b_n\, r_i} ~=~ \frac{b_1 \dotsm b_{n-1}}{b_2 \dotsm b_n} ~=~ (-1)^{\frac{n+2}{2}}. \]
		Also, $n$ must be even. 
		Since $\bar{p}$ is gluable along $r_i$, $p$ will have a non-zero value on $r_i$, so $p$ does not quash either puncture. 
		
		Also, observe that any triangulation of $\SS$ must either include radially adjacent radii or a length 2 diagonal. Thus, every triangulation must include an arc that is killed by $p$ as a gluing of $\bar{p}$, and $p$ is therefore a benevolent deep point that spares $r_i$.
		
		\item For our inductive hypothesis, now suppose that $q>2$ and the bijection holds for all $q_0 < q$. 
		Since $\SS\smallsetminus r_i$ is a $(q-1)$-punctured $(n+2)$-gon, our inductive hypothesis says that there exists an ordered sequence of non-boundary arcs $r_1, \dotsc, r_{q-1}$ in $\SS\smallsetminus r_i$ that are spared in $\bar{p}$ such that the sets 
		\[ R_k \coloneqq \{ r_1, \dotsc, r_k \}, \text{ and} \]
		\[ S_k \coloneqq \{ \text{benevolent deep points $p\in V(\CA(\SS\smallsetminus R_{k-1}),\Bbbk)$ non-zero on $R_k$}\} \]
		give us bijections \[ S_2 \leftrightarrow S_3 \leftrightarrow \dotsb \leftrightarrow S_{q-2} \leftrightarrow S_{q-1}. \]
		There is then a benevolent deep point $\widehat{p} \in V(\CA(\SS\smallsetminus R_{q-1}),\Bbbk)$ that glues to a benevolent deep point on $\SS \smallsetminus r_i$. Hence, $\widehat{p}$ kills alternating plain radii and all notched radii in $\SS\smallsetminus R_{q-1}$, and satisfies the boundary relation 
		\[ \frac{b_1 \dotsm b_{n-1}\, r_1 \dotsm r_{q-1}\, r_i}{b_2 \dotsm b_n\, r_1 \dotsm r_{q-1}\, r_i} ~=~ \frac{b_1 \dotsm b_{n-1}}{b_2 \dotsm b_n} ~=~ (-1)^{\frac{n+2q + 2}{2}}. \]
		As a result, $n$ must be even, and alternating radii around each puncture must be killed by $p$. 
		
		Again, observe that any triangulation of $\SS$ must either include radially adjacent radii or a length 2 diagonal. Thus, every triangulation must include an arc that is killed by $p$ as a gluing of $\bar{p}$, and $p$ is therefore a benevolent deep point that spares $r_i$.            
	\end{itemize}
\end{proof}

\begin{rem} When we say \emph{alternating radii around each puncture}, we mean the radii connecting that puncture to either the even-indexed marked points or the odd-indexed marked points. It is possible that a triangulation will contain multiple radii with the same endpoints (e.g. Figure \ref{fig: disk with one marked point}), with both radii killed or both spared. This will be further clarified in Lemma \ref{lemma: radii with same endpoints}.
\end{rem}

Let us now emphasize a few takeaways from the above proof in the form a porism.
\begin{porism}\label{por: benevolent characterization}
	Let $\SS$ be a disk with $n\geq 1$ marked points on the boundary and $q\geq 2$ punctures, and let $p \in V(\CA(\SS),\Bbbk)$ be deep.
	\begin{enumerate}
		\item If $n$ is odd, then $p$ must quash at least one puncture.
		\item If $p$ is benevolent, then $p$ must satisfy the boundary relation 
		\[ \frac{b_1 \dotsm b_{n-1}}{b_2 \dotsm b_n} ~=~ (-1)^{\frac{n+2q + 2}{2}}. \]
		\item If $p$ is benevolent, then around each puncture, $p$ must kill alternating radii of one tagging and all radii of the opposite tagging. 
	\end{enumerate} 
\end{porism}

%=====================

\subsection{Disks with at least 2 marked points on the boundary}\label{sec: punc polygons}

We examine the deep points in two groups: those that quash a non-empty subset of the punctures, and those that quash no punctures. Let $\SS_q$ be a disk with $n\geq 2$ marked points on the boundary and $q\geq 2$ punctures, and let $p\in V(\CA(\SS_q),\Bbbk)$ be deep.

\begin{itemize}
    \item Suppose $p$ quashes $k \geq 1$ of the punctures, arbitrarily indexed $\{\xi_1, \dotsc, \xi_k\}$. Let $c_{i,1}$ be an arc that cuts out a once-punctured bigon opposite the boundary arc $b_i$, such that the bigon contains the puncture $\xi_1$. All arcs emanating from $\xi_1$ are killed by $p$, so $c_{i,1} = -b_i \neq 0$. The cut surface $\SS \smallsetminus c_{i,1}$ has two components \textemdash a once-punctured bigon and a $(q-1)$-punctured $n$-gon $\SS_{q-1}$. The deep point $\bar{p} \in V(\CA(\SS\smallsetminus c_{i,1}), \Bbbk)$ restricts to a deep point on the once-punctured bigon, so it need not be deep on $\SS_{q-1}$. Also, the value of the boundary arc $b_i$ is determined by the value of the corresponding boundary arc $c_{i,1}$ in $\SS_{q-1}$. As such, the set of deep points that quash a given puncture is isomorphic to $V(\CA(\SS_{q-1}),\Bbbk)$, the cluster variety corresponding to the $(q-1)$-punctured $n$-gon. 

    This process can be replicated for each additional quashed puncture, choosing $c_{i,j}$ such that it cuts out a once-punctured bigon with $c_{i,j-1}$ that contains the puncture $\xi_j$. All arcs emanating from the quashed punctures are killed by $p$, so $c_{i,j} = (-1)^j b_i \neq 0$. If $n\geq 3$ or $k<q$, the cut surface $\SS \smallsetminus \{c_{i,1}, \dotsc, c_{i,k}\}$ has $k+1$ components \textemdash $k$ once-punctured bigons and a $(q-k)$-punctured $n$-gon $\SS_{q-k}$. The deep point $\bar{p} \in V(\CA(\SS\smallsetminus \{c_{i,1}, \dotsc, c_{i,k}\}), \Bbbk)$ restricts to a deep point on each of the once-punctured bigons, so it need not be deep on $\SS_{q-k}$. Also, the value of the boundary arc $b_i$ is determined by the value of the corresponding boundary arc $c_{i,k}$ in $\SS_{q-k}$. As such, the set of deep points that quash a given set of $k$ punctures is isomorphic to $V(\CA(\SS_{q-k}),\Bbbk)$, the cluster variety corresponding to the $(q-k)$-punctured $n$-gon. 
    
    In particular, if $n\geq 3$, the set of malevolent points is isomorphic to the cluster variety of the unpunctured $n$-gon $V(\CA(\Delta_n),\Bbbk)$, which is of dimension $2n-3$. If $n=2$, then the cuts along $\{c_{i,1}, \dotsc, c_{i,q-1}\}$ produce $q$ components, each a once-punctured bigon. In this case, the value on one boundary arc determines the values on all other boundary arcs, so the set of malevolent points is isomorphic to $\Bbbk^\times$.
		
	\item Suppose no punctures are quashed by $p$, so $p$ is benevolent. From Lemma \ref{lemma: benevolent spares a radius}, we know that at least one radius is spared. From Porism \ref{por: benevolent characterization}, we know three things:
	\begin{itemize}
		\item Since $p$ is benevolent, $n$ must be even.
		\item Since $p$ is benevolent, it must satisfy the boundary relation 
		\[ \frac{b_1 \dotsm b_{n-1}}{b_2 \dotsm b_n} ~=~ (-1)^{\frac{n+2q + 2}{2}}, \] 
        so choosing $n-1$ of the boundary values determines the last one.
		\item Since $p$ is benevolent, $p$ kills alternating radii of one tagging and all radii of the opposite tagging around each puncture. Once the boundary values are chosen, a choice of non-zero value for one radius at each puncture determines the rest. 
	\end{itemize} 
	Recall that we have 2 possible taggings at each puncture, and a choice as to whether the odd-indexed or even-indexed radii are killed. This gives us $4^q$ sets of benevolent deep points, each isomorphic to $(\Bbbk^\times)^{q+n-1}$.
\end{itemize}

%=====================

\subsection{Monogons}\label{sec: punc monogons}

We now turn to punctured monogons, examining the deep points in three separate groups: the malevolent points, those that quash a proper non-empty subset of the punctures, and those that quash no punctures. 
Let $\SS$ be a disk with $n=1$ marked points on the boundary and $q\geq 2$ punctures, and let $p\in V(\CA(\SS),\Bbbk)$ be deep.

\begin{itemize}
	\item Suppose $p$ is a malevolent deep point, so every puncture is quashed by $p$. 
	\begin{itemize} 
		\item If $q=2$, then the only arc not incident to a puncture is the boundary arc. Considering the triangulation of Figure \ref{fig: twice-punc monogon quiver}, mutation at $x_1$ gives us the relation $x_1 x_1^\prime = x_3 x_4 + x_5$. Since $p$ quashes both punctures, we have $0 = 0 + p(x_5) \neq 0$, which is a contradiction. Hence, there are no malevolent deep points.  
		
		\item If $q \geq 3$, there are arcs which cut out $k$-punctured monogons for $2 \leq k \leq q-1$, which we denote $c_k$. Observe that the arc $c_{q-1}$ forms a once-punctured bigon with the boundary arc $b$. Since all radii are killed, we have the relation $c_{q-1} = -b_1 \neq 0$. Likewise, $c_{j-1} = -c_j$ for all $2 < j \leq q-1$. Cutting along $\{ c_2, \dotsc, c_{q-1} \}$ gives us $q-1$ components: a twice-punctured monogon and $q-2$ once-punctured bigons. However, since both punctures in the twice-punctured monogon are quashed, we arrive at a contradiction as in the $q=2$ case above. Hence, there are no malevolent deep points.
	\end{itemize}
	
	\item Suppose $p$ is not a malevolent deep point, but at least one puncture $\xi_i$ is quashed by $p$. 
	\begin{itemize}
		\item First, suppose that $q=2$. Since $p$ is not malevolent, there is some other puncture $\xi_j$ that is not quashed by $p$. Every arc emanating from $\xi_i$ is killed, including the tagged tunnels connecting $\xi_i$ and $\xi_j$. Since $\xi_j$ is not quashed, a radius $r_1$ connecting $\xi_j$ to the boundary must be spared. We can then choose a radius $r_2$ that cuts out a once-punctured bigon with $r_1$ that contains the quashed puncture $\xi_i$. Since $p(r_1) \neq 0$, and all arcs emanating from $\xi_i$ are killed by $p$, the relation $p(r_1) + p(r_2) = 0$ implies that $p(r_2) \neq 0$ as well. 
		
		Cutting along the arcs $\{ r_1, r_2 \}$ produces a disconnected surface with one component a once-punctured bigon and the other an unpunctured trigon $\Delta_3$. The restriction to $\Delta_3$ cannot be deep, but since the restriction of $p$ to the once-punctured bigon is deep, that is not an issue. Further, the value of $p(r_1)$ determines the values of $p(r_2)$ and consequently determines the values of two sides of the trigon as well; the value of the third side of the trigon remains to be freely chosen as there are no mutation relations in $\CA(\Delta_3)$. Thus, the set of deep points that quash one of the two punctures is isomorphic to $(\Bbbk^\times)^2$. 
		
		\item Now, suppose that $q\geq 3$, and let $\xi_1$ be a puncture that is quashed by $p$. Denote by $c_1$ the arc that cuts out a $(q-1)$-punctured monogon not containing $\xi_1$. Along with the boundary arc $b_1$, the arc $c_1$ cuts out a once-punctured bigon containing $\xi_1$. Since all radii from $\xi_1$ are killed by $p$, we have the relation $c_1 = -b_1 \neq 0$; cutting along $c_1$ gives us two connected components \textemdash a once-punctured bigon and a $(q-1)$-punctured monogon $\SS_{q-1}$. Because the restriction of $p$ to the once-punctured bigon is deep, the restriction to $\SS_{q-1}$ need not be deep. Further, the values on the boundary edges of the once-punctured bigon are determined by the boundary value on $\SS_{q-1}$, so the set of deep points that quash a given puncture is isomorphic to $V(\CA(\SS_{q-1}), \Bbbk)$, the cluster variety corresponding to a $(q-1)$-punctured monogon. 

        Similarly, if $p$ quashes $1 \leq j \leq q-2$ punctures, then the above process can be replicated, giving a set of deep points isomorphic to $V(\CA(\SS_{q-j}), \Bbbk)$. If $q-1$ punctures are quashed by $p$, the same cutting process gives us $q-2$ once-punctured bigons and a twice-punctured monogon. One of the punctures in the twice-punctured monogon is quashed, giving us a subset of the deep locus isomorphic to $(\Bbbk^\times)^2$ as in the $q=2$ case. 

    	Since there are $q$ punctures, there are $q$ deep subsets isomorphic to $V(\CA(\SS_{q-1}), \Bbbk)$. The intersection of any $2 \leq j \leq q-2 $ such sets corresponds to points which quash the $j$ corresponding punctures, and is isomorphic to $V(\CA(\SS_{q-j}), \Bbbk)$. The intersection of any $q-1$ such sets corresponds to points which quash those $q-1$ punctures, and is isomorphic to $(\Bbbk^\times)^2$. Since there are no malevolent points, the intersection of all $q$ such sets is empty. 
    \end{itemize}
	
	\item From Porism \ref{por: benevolent characterization}, we know that if $n$ is odd, then $p$ must quash at least one puncture. Thus, there are no benevolent deep points in a punctured monogon.
	
\end{itemize}

%============================

\subsection{Benevolent points and radii}\label{section: benevolent radii}

Now, we want to examine how a benevolent deep point handles radii with the same endpoints. In particular, in a disk with at least 2 punctures, there are radii with the same endpoints that are not homotopic, so it is natural to ask how the values assigned to these radii are related.

\begin{lemma}[The Unicorn Surgery Lemma\footnote{The construction in the proof of the lemma is inspired by the \emph{unicorn paths} of \cite{HPW15} and \emph{surgery paths} of \cite{HatcherTriangulations}.}]\label{lemma: radii with same endpoints}
	Let $\SS$ be a disk with $n\geq 1$ marked points on the boundary and $q\geq 2$ punctures, and let $p\in V(\CA(\SS),\Bbbk)$ be a benevolent deep point. If $r_1$ and $r_2$ are radii with the same endpoints, then \[p(r_1) = (-1)^k ~p(r_2) \] for some $k \in \bbN$.
\end{lemma}
\begin{proof}
	If $r_1$ and $r_2$ are compatible \textemdash that is, their homotopy classes contain representative arcs that do not intersect except at their endpoints \textemdash then they bound a $k$-punctured bigon for some $k\geq 1$. Since $p$ kills all radii of one of the taggings, if $k=1$, we have the relation $p(r_1) + p(r_2) = 0$. If $k>1$, there are radii $c_1, \dotsc, c_{k-1}$ that cut the $k$-punctured bigon into $k$ once-punctured bigons. The relations $p(r_1) + p(c_1) = 0$, $p(c_{k-1}) + p(r_2) = 0$, and $p(c_i) + p(c_{i+1}) = 0$ for all $1 \leq i \leq k-2$ combine to give us the desired relation. 
	
	If $r_1$ and $r_2$ are not compatible, we can construct a sequence of radii with the same endpoints $r_1 = c_0, c_1, \dotsc, c_\ell, c_{\ell+1} = r_2$ such that $c_j$ is compatible with both $c_{j-1}$ and $c_{j+1}$ for all $1 \leq j < \ell$. To begin, we choose minimally intersecting representatives of $r_1$ and $r_2$ that intersect transversely at each crossing. Next, assign an orientation to each radius so that both radii "flow" from the boundary marked point $m$ to the puncture $\xi$. 
	
	Now, let $c_1$ be the radius that starts at $m$ and follows $r_2$ until its first intersection with with $r_1$, then turns to follow $r_1$ in the direction of the puncture $\xi$. This new radius $c_1$ is such that $\mu(r_1, c_1) = 0$ and $\mu(c_1, r_2) < \mu(r_1, r_2)$.\footnote{This is the intersection pairing of Definition \ref{defn: intersection pairing}.} Since $r_1$ and $c_1$ have non-intersecting homotopy representatives, $p(r_1) = (-1)^{k_1}~p(c_1)$ for some $k_1 \in \bbN$.
	
	If $\mu(c_1, r_2) > 0$, repeat the process, choosing minimally intersecting representatives of $c_1$ and $r_2$ that intersect transversely at each crossing, and assigning an orientation to each radius so that both radii flow from $m$ to $\xi$. Let $c_2$ be the radius that starts at $m$ and follows $r_2$ until its first intersection with with $c_1$, then turns to follow $c_1$ in the direction of the puncture $\xi$. This radius $c_2$ is such that $\mu(c_1, c_2) = 0$ and $\mu(c_2, r_2) < \mu(c_1, r_2)$. Since $c_1$ and $c_2$ have non-intersecting homotopy representatives, $p(c_1) = (-1)^{k_2}~p(c_2)$ for some $k_2 \in \bbN$. If $\mu(c_2, r_2) > 0$, repeat this process until we have a radius $c_\ell$ such that $\mu(c_\ell, r_2) = 0$, so that $p(c_\ell) = (-1)^{k_\ell}~p(r_2)$ for some $k_\ell \in \bbN$. 
	
	Finally, we have a sequence $r_1 = c_0, c_1, \dotsc, c_\ell, c_{\ell+1} = r_2$ of pairwise compatible radii, with each pair giving us a relation $p(c_{j-1}) = (-1)^{k_j}~p(c_j)$. Combining these relations, we have \[p(r_1) = (-1)_{}^{\sum_{i=1}^\ell k_i}~p(r_2).\qedhere\]
\end{proof}

\begin{coro}\label{cor: radii with same endpoints}
	Let $\SS$ be a disk with $n\geq 1$ marked points on the boundary and $q\geq 2$ punctures, and let $\mathfrak{R}_{i,j}$ denote the set of all radii connecting a puncture $\xi_i$ to a marked point $m_j$. If $p\in V(\CA(\SS),\Bbbk)$ is a benevolent deep point, then either $p$ kills every radius in $\mathfrak{R}_{i,j}$ or $p$ spares every radius in $\mathfrak{R}_{i,j}$.  
\end{coro}

%============================

\section{Punctured surfaces with boundary}\label{sec: punc surfaces}

Recall from Theorem \ref{thm: tagged skein cluster algebra} that $\CA(\SS) \cong \TSk (\SS)$ if $\SS$ is a triangulable marked surface such that at least one of the following is true:
\begin{enumerate}
	\item $\SS$ is a disk.
	\item $\SS$ embeds into a disk. 
	\item $\SS$ has at least two boundary marked points in each connected component. 
\end{enumerate}
Since we assume that the surface is triangulable, the second type is covered by the first and third types, so we need not address it directly.

We have already addressed punctured marked disks, so we now turn our attention to more general punctured surfaces, possibly with positive genus or multiple boundary components. Because we assume that these surfaces are triangulable, we require there be at least one marked point on every boundary component. We will also require at least two marked points in each connected component; if there is only one boundary component, it must contain at least two marked points. Where convenient, we will also exclude the exceptional punctured bigons, which we have already addressed in Sections \ref{sec: once-punc bigon} and \ref{sec: punc polygons}.

%============================

\subsection{Points that quash at least one puncture}\label{section: surface quash}

First, consider a connected, triangulable marked surface $\SS_q$ with genus $g\geq 0$, $b \geq 1$ boundary components, $q\geq 1$ punctures, and $m\geq 2$ marked points, with at least two of the marked points on the same boundary component, such that $\SS_q$ is not a punctured bigon. Let $p \in V(\CA(\SS_q),\Bbbk)$ be deep.

\begin{itemize}
	\item Suppose $p$ is a malevolent deep point, so every puncture is quashed by $p$. Let $c_{i,1}$ be an arc that cuts out a once-punctured bigon opposite the boundary arc $b_i$. For $2 \leq j \leq p$, let $c_{i,j}$ be an arc that cuts out a once-punctured bigon opposite the arc $c_{i,j-1}$. Since every arc emanating from every puncture is killed by $p$, we know that $c_{i,1} = -b_i \neq 0$. The cut surface $\SS_q \smallsetminus c_{i,1}$ has two components \textemdash a once-punctured bigon and a $(q-1)$-punctured surface $\SS_{q-1}$. Since $c_{i,j} = (-1)^j~b_i$, we know $c_{i,j}\neq 0$, so we can cut along each $c_{i,j}$. The cut surface $\SS_q \smallsetminus \{c_{i,1}, \dotsc, c_{i,q}\}$ then has $q+1$ components: an unpunctured version of $\SS_q$, which we will call $\SS_0$, and $q$ once-punctured bigons.
	
	The deep point $\bar{p} \in V(\CA(\SS_q\smallsetminus \{c_{i,1}, \dotsc, c_{i,q} \}), \Bbbk)$ restricts to a deep point on each of the once-punctured bigons, so it need not be deep on the unpunctured surface $\SS_0$. Also, the boundary values on the once-punctured bigons are determined by the boundary values on $\SS_0$, so $p$ is determined by its values on $\SS_0$. As such, the set of malevolent points is isomorphic to the cluster variety coming from the unpunctured surface, $V(\CA(\SS_0),\Bbbk)$.
	
	\item Suppose $p$ is not a malevolent deep point, but quashes a subset of $1 \leq k \leq q-1$ punctures, arbitrarily indexed $\{\xi_1, \dotsc, \xi_k\}$. 
	Let $c_{i,1}$ be an arc that cuts out a once-punctured bigon opposite the boundary arc $b_i$, such that the bigon contains the puncture $\xi_1$. Since all arcs emanating from $\xi_1$ are killed by $p$, we have $c_{i,1} = -b_i \neq 0$. The cut surface $\SS_q \smallsetminus c_{i,1}$ has two components \textemdash a once-punctured bigon and a $(q-1)$-punctured surface $\SS_{q-1}$. The deep point $\bar{p} \in V(\CA(\SS_q\smallsetminus c_{i,1}), \Bbbk)$ restricts to a deep point on the once-punctured bigon, so it need not be deep on $\SS_{q-1}$. As such, the set of deep points which quash a given puncture is isomorphic to $V(\CA(\SS_{q-1}),\Bbbk)$, the cluster variety corresponding to the $(q-1)$-punctured surface. 
    Replicating this process for $2\leq j \leq k$, we see that the set of deep points that quash $j$ punctures is isomorphic to $V(\CA(\SS_{q-j}),\Bbbk)$.
	
	Since there are $q$ punctures, there are $q$ such deep subsets isomorphic to $V(\CA(\SS_{q-1}),\Bbbk)$. The intersection of $2 \leq j \leq q-1$ such sets corresponds to points that kill those $j$ punctures and is isomorphic to $V(\CA(\SS_{q-j}),\Bbbk)$; the intersection of all $q$ such sets is the malevolent points.
\end{itemize}

%============================

Now, consider a connected surface $\SS$ with genus $g\geq 0$, $b \geq 2$ boundary components, $q\geq 1$ punctures, and $m\geq 2$ marked points, with the assumption that each boundary component contains only a single marked point. Let $p \in V(\CA(\SS),\Bbbk)$ be deep. 

\begin{lemma}\label{lemma: spared arc isolated marked points}
	Let $p$ and $\SS$ be as defined above. Then $p$ spares at least one arc in $\SS$. Further, if $p$ quashes a puncture $\xi_i$, then $p$ spares an arc $c_i$ that cuts out a once-punctured bigon with a boundary arc that contains $\xi_i$. 
\end{lemma}
\begin{proof}
	Consider the arcs as depicted in Figure \ref{fig: claim 1 situation 1}. The arcs $a_q$, $b_1$, $a_0^\prime$, and $b_2$ form an unpunctured quadrilateral with diagonals $d_1$ and $d_2$. This gives us a Ptolemy relation $d_1 d_2 = a_q a_0^\prime + b_1 b_2$. Since $b_1$ and $b_2$ are boundary arcs, we have $b_1 b_2 \neq 0$, so at least one of the other terms must be non-zero as well. Therefore, either $a_q a_0^\prime \neq 0$ or $d_1 d_2 \neq 0$.

    Now, suppose that $p$ quashes a puncture $\xi_i$. Let $c_i$ be an arc that cuts out a once-punctured bigon with boundary arc $b_i$ such that $\xi_i$ is the puncture contained in that bigon, as depicted in Figure \ref{fig: claim 1 situation 2}. Then $p(c_i) + p(b_i) =0$, so $p(c_i) \neq 0$. 
\end{proof}

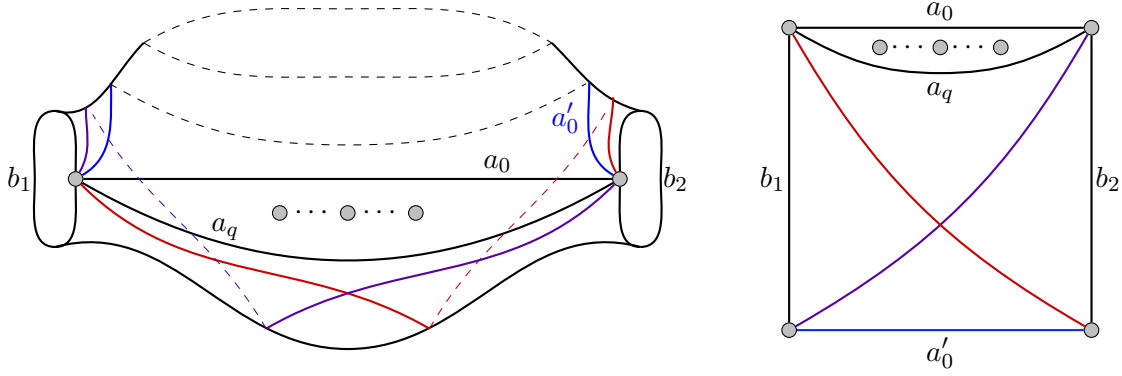
\begin{figure}[h!tb]
	\centering
	\begin{tikzpicture}[scale=0.9,baseline={(0,0)}]
		\node[dot] (p1) at (-1,-0.5) {};
		\node[dot] (p2) at (0,-0.5) {};
		\node[dot] (p3) at (1,-0.5) {};
		\node (p4) at (-0.5,-0.5) {$\cdots$};
		\node (p5) at (0.5, -0.5) {$\cdots$};
		\node[dot] (m1) at (-4,0) {};
		\node[dot] (m2) at (4,0) {};
		\draw[thick,out=270,in=0] (m1) to (-4.3,-1);
		\draw[thick,out=180,in=270] (-4.3,-1) to (-4.6,0);
		\draw[thick,out=90,in=180] (-4.6,0) to (-4.3,1) to [out=0,in=90] (m1) {};
		\draw[thick,out=270,in=180] (m2) to (4.3,-1);
		\draw[thick,out=0,in=270] (4.3,-1) to (4.6,0);
		\draw[thick,out=90,in=0] (4.6,0) to (4.3,1) to [out=180,in=90] (m2) {};
		\draw[thick] (-4.3,-1) to [out=15,in=180] (0,-2.5) to [out=0,in=165] (4.3,-1) {};
		\draw[thick,out=345,in=225] (-4.3,1) to (-3,2);
		\draw[thick,out=165,in=315] (4.3,1) to (3,2);
		\draw[dashed,out=330,in=180] (-3,2) to (0,1.5) to [out=0,in=210] (3,2); 
		\draw[dashed,out=30,in=180] (-3,2) to (0,2.5) to [out=0,in=150] (3,2);
		\draw[thick] (m1) to (m2) {};
		\draw[thick,bend right] (m1) to (m2) {};
		\draw[thick,blue,out=30,in=270] (m1) to (-3.48,1.4) {};
		\draw[thick,blue,out=150,in=270] (m2) to (3.55,1.42) {};
		\draw[dashed,out=330,in=180] (-3.5,1.4) to (0,0.5) to [out=0,in=210] (3.5,1.4) {};
		\node[blue] at (3.2,0.9) {$a_0^\prime$};
		\node at (2.2,0.22) {$a_0$};
		\node at (-1.8,-0.74) {$a_q$};
		\node at (-4.82,0) {$b_1$};
		\node at (4.82,0) {$b_2$};
		\draw[thick,dark red,out=315,in=150] (m1) to (1.2,-2.2) {}; 
		\draw[dashed,dark red,out=60,in=240] (1.2,-2.2) to (3.8,1.0) {};
		\draw[thick,dark red,out=270,in=120] (3.9,1.2) to (m2) {};
		\draw[thick,dark purple,out=225,in=30] (m2) to (-1.2,-2.2) {};
		\draw[dashed,dark purple,out=120,in=300] (-1.2,-2.2) to (-3.75,0.95) {};
		\draw[thick,dark purple,out=270,in=60] (-3.85,1.06) to (m1) {};
	\end{tikzpicture}
	\hspace{0.5cm}
	\begin{tikzpicture}[scale=2,baseline={(0,0)}]
		\node[dot] (m1a) at (-1,1) {};
		\node[dot] (m1b) at (-1,-1) {};
		\node[dot] (m2a) at (1,1) {};
		\node[dot] (m2b) at (1,-1) {};
		\draw[thick] (m1a) to (m1b) (m2a) to (m2b);
		\draw[thick,dark blue] (m1b) to (m2b);
		\draw[thick] (m1a) to (m2a);
		\draw[thick] (m1a) to [out=330,in=180] (0,0.7) to [out=0,in=210] (m2a);
		\node[dot] (p1) at (-0.4,0.87) {};
		\node[dot] (p2) at (0,0.87) {};
		\node[dot] (p3) at (0.4,0.87) {};
		\node (p4) at (-0.2,0.87) {$\cdots$};
		\node (p5) at (0.2,0.87) {$\cdots$};
		\draw[thick,dark purple,out=240,in=30] (m2a) to (m1b);
		\draw[thick,dark red,out=300,in=150] (m1a) to (m2b);
		\node at (0,0.57) {$a_q$};
		\node at (0,1.11) {$a_0$};
		\node at (-1.11,0) {$b_1$};
		\node at (1.11,0) {$b_2$};
		\node at (0,-1.15) {$a_0^\prime$};
	\end{tikzpicture}
	\caption{Local picture of arcs forming an unpunctured quadrilateral in the surface with one marked point on each boundary}
	\label{fig: claim 1 situation 1}
\end{figure}

\begin{figure}[h!tb]
	\centering
	\begin{tikzpicture}[scale=0.9,baseline={(0,0)}]
		\node[dot] (p1) at (-1,-0.5) {};
		\node[dot] (p2) at (0,-0.5) {};
		\node[dot] (p3) at (1,-0.5) {};
		\node (p5) at (0.5, -0.5) {$\cdots$};
		\node[dot] (m1) at (-4,0) {};
		\node[dot] (m2) at (4,0) {};
		\draw[thick,out=270,in=0] (m1) to (-4.3,-1);
		\draw[thick,out=180,in=270] (-4.3,-1) to (-4.6,0);
		\draw[thick,out=90,in=180] (-4.6,0) to (-4.3,1) to [out=0,in=90] (m1) {};
		\draw[thick,out=270,in=180] (m2) to (4.3,-1);
		\draw[thick,out=0,in=270] (4.3,-1) to (4.6,0);
		\draw[thick,out=90,in=0] (4.6,0) to (4.3,1) to [out=180,in=90] (m2) {};
		\draw[thick] (-4.3,-1) to [out=15,in=180] (0,-2.5) to [out=0,in=165] (4.3,-1) {};
		\draw[thick,out=345,in=225] (-4.3,1) to (-3,2);
		\draw[thick,out=165,in=315] (4.3,1) to (3,2);
		\draw[dashed,out=330,in=180] (-3,2) to (0,1.5) to [out=0,in=210] (3,2); 
		\draw[dashed,out=30,in=180] (-3,2) to (0,2.5) to [out=0,in=150] (3,2);
		\node at (-4.82,0) {$b_i$};
        \draw[thick,blue] (m1) to [out=0,in=180] (-1,-1.0) to [out=0,in=270] (-0.5,-0.5) to [out=90,in=315] (-3.48,1.4); 
        \draw[thick,blue] (m1) to [out=330,in=105] (-3.3,-1);
        \draw[dashed,blue] (-3.3,-1) to [out=75,in=235] (-3.48,1.4);
        \node[blue] at (-0.55,-1.1) {$c_i$};
        \node at (-1.22,-0.2) {$\xi_i$};
	\end{tikzpicture}
	\caption{Picture of an arc $c_i$ forming a once-punctured bigon with a boundary arc $b_i$ in the surface with one marked point on each boundary}
	\label{fig: claim 1 situation 2}
\end{figure}
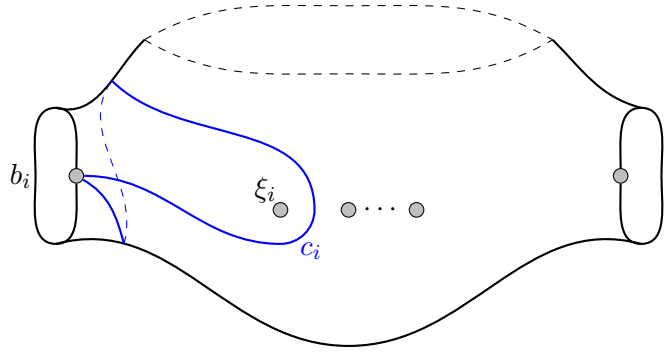

As a result of Lemma \ref{lemma: spared arc isolated marked points}, we can cut off once-punctured bigons in the same way as when there are two marked points on the same boundary component. Thus, the deep points that quash $j\geq 1$ punctures can be described in exactly the same way.

\begin{thm}\label{thm: inclusions surfaces}
	Let $\SS_q$ be a connected, triangulable marked surface with genus $g\geq 0$, $b\geq 1$ boundary components, $q$ punctures, and $m\geq 2$ marked points on the boundary, with at least $1$ marked point on each boundary component, such that $\SS_q$ is not a punctured bigon. There are inclusions of the cluster varieties \[\text{$V(\CA(\SS_j),\Bbbk)\;$ into the deep locus of $\;V(\CA(\SS_q),\Bbbk)\;$ for all $\;0\leq j < q$.}\]
	
	In particular, there is an inclusion of the cluster variety of the unpunctured surface, $V(\CA(\SS_0),\Bbbk)$, into the deep locus of the cluster variety of the punctured surface,  $V(\CA(\SS_q),\Bbbk)$.
\end{thm}
\begin{proof}
Choose an arbitrary ordering of the punctures $\xi_1, \xi_2, \dotsc, \xi_q$. As noted above, the points of $V(\CA(\SS_q),\Bbbk)$ that quash $\xi_1$ correspond to a deep subvariety $V_1$ isomorphic to $V(\CA(\SS_{q-1}),\Bbbk)$. Likewise, the points of $V(\CA(\SS_q),\Bbbk)$ that quash $\xi_1$ and $\xi_2$ correspond to a deep subvariety $V_2$ isomorphic to $V(\CA(\SS_{q-2}),\Bbbk)$, and $V_2 \subset V_1$. Continuing for any $1\leq j < q$, the points of $V(\CA(\SS_q),\Bbbk)$ that quash the first $j+1$ punctures correspond to a deep subvariety $V_{j+1}$ which is isomorphic to $V(\CA(\SS_{q-j-1}),\Bbbk)$, and $V_{j+1} \subset V_j$. This gives us a nested inclusion of deep subvarieties
\[ V_q \coloneqq V(\CA(\SS_0),\Bbbk) ~\subset~ V_{q-1} ~\subset~ \dotsb ~\subset~ V_2 ~\subset~ V_1.\qedhere\]
\end{proof}

%============================

\subsection{Cutting along radii}\label{section: cutting punc surf}

Before examining the benevolent points, we want to extend previous results regarding cutting along radii.\footnote{The results of Lemmata \ref{lemma: spared radius surface} and \ref{lemma: cutting radii for unpunctured surface} have also been developed independently in \cite{tyler}.}

\begin{lemma}\label{lemma: tunnel quadrilateral surface}
	Let $\SS$ be a connected, triangulable marked surface with genus $g\geq 0$, $b \geq 1$ boundary components, $m \geq 2$ marked points on the boundary, and $q \geq 2$ punctures. For any given tunnel $\tau$, there is an unpunctured quadrilateral in $\SS$ bounded by one boundary arc, two radii, and the tunnel $\tau$. 
\end{lemma}
\begin{proof}
	The proof is the same as in Lemma \ref{lemma: tunnel quadrilateral}. 
\end{proof}

\begin{lemma}\label{lemma: spared radius surface}
	Let $\SS$ be a connected, triangulable marked surface with genus $g\geq 0$, $b \geq 1$ boundary components, $m \geq 2$ marked points on the boundary, and $q \geq 1$ punctures. If $p \in V(\CA(\SS),\Bbbk)$ does not quash any punctures, then it spares at least one radius. 
\end{lemma}
\begin{proof}
	Seeking a contradiction, suppose that $p$ kills all radii. If $q=1$, then $p$ quashes the puncture so we have a contradiction.
	
	Now, suppose that $q\geq 2$. By supposition, $p$ does not quash any punctures, so there must exist a tunnel $\tau$ that is spared by $p$. By Lemma \ref{lemma: tunnel quadrilateral surface}, there exists a quadrilateral consisting of two radii, a boundary component $b$, and $\tau$, with $\tau$ opposite $b$ as in Figure \ref{fig: spared tunnel}. The Ptolemy relation on such a quadrilateral is then $0 = p(\tau) p(b) + 0$, but both $p(b)$ and $p(\tau)$ are non-zero, so we have a contradiction.
\end{proof}

\begin{lemma}\label{lemma: cutting radii for unpunctured surface}
	Let $\SS$ be a connected, triangulable marked surface with genus $g\geq 0$, $b \geq 1$ boundary components, $m \geq 2$ marked points on the boundary, and $q \geq 1$ punctures. If $p \in V(\CA(\SS),\Bbbk)$ does not quash any punctures, then there exist a collection of arcs $r_1, \dotsc, r_q$ such that $p(r_i) \neq 0$ for all $1\leq i \leq q$, and $\widehat{\SS} \coloneqq \SS\smallsetminus \{r_1, \dotsc, r_q\}$ is an unpunctured surface with genus $g$, $b$ boundary components, and $m+2q$ marked points on the boundary. 
\end{lemma}
\begin{proof}
	By Lemma \ref{lemma: spared radius surface}, there exists a radius $r_1$ such that $p(r_1) \neq 0$. Cutting along $r_1$, we get a surface $\SS_1$ with genus $g$, $b$ boundary components, $m+2$ marked points, and $q-1$ punctures.
	
	Repeatedly applying Lemma \ref{lemma: spared radius surface}, after each cut we have a new surface with two additional marked points and one fewer puncture. After $q$ steps, we have an unpunctured surface $\widehat{\SS}$ with genus $g$, $b$ boundary components, and $m+2q$ marked points on the boundary.
\end{proof}

Now, we wish to extend Corollary \ref{coro: deeppolydiss} to the setting of punctured surfaces.

\begin{prop}\label{prop: punc surf poly diss}
	Let $\SS$ be a connected, triangulable marked surface with genus $g\geq 0$, $b\geq 1$ boundary components, $m \geq 2$ marked points on the boundary, and $q\geq 1$ punctures. If $p\in V(\CA(\SS),\Bbbk)$ is a benevolent deep point, then there is a polygonal dissection $D$ of $\SS$ on which $p$ is non-zero.
\end{prop}
\begin{proof}
	First, let $p\in V(\CA(\SS),\Bbbk)$ be a benevolent deep point. By Lemma \ref{lemma: cutting radii for unpunctured surface}, there exist arcs $r_1, \dotsc, r_q$ such that $p(r_i) \neq 0$ for $1\leq i \leq q$, and the cut surface $\widehat{\SS} \coloneqq \SS\smallsetminus \{r_1, \dotsc, r_q\}$ is an unpunctured surface with genus $g$, $b$ boundary components, and $m+2q$ marked points on the boundary. 
	By Proposition \ref{prop: deep cutting inclusion}, the image of $p$ under the cutting isomorphism $\bar{p} \in V(\CA(\widehat{\SS}),\Bbbk)$ is deep. From Lemma \ref{lemma: polydissavoid}, there exists a polygonal dissection $D^\prime$ of $\widehat{\SS}$ that is non-zero under $\bar{p}$. Lifting, we have that $D \coloneqq D^\prime \cup \{r_1, \dotsc, r_q\}$ is a polygonal dissection of $\SS$ on which $p$ is non-zero. 
\end{proof}

Next, we wish to extend Proposition \ref{prop: benevolent bijection multiple punctures} to the current setting.

\begin{prop}\label{prop: benevolent bijection surface}
	Let $\SS$ be a connected, triangulable marked surface with genus $g\geq 0$, $b\geq 1$ boundary components, $m\geq 2$ marked points on the boundary, and $q\geq 1$ punctures. Let $D$ be a polygonal dissection of $\SS$. There is a bijection 
	\[ 
	\begin{Bmatrix} \text{benevolent deep points} \\ \text{in } V(\CA(\SS),\Bbbk) \\ \text{non-zero on } D \end{Bmatrix} 
	\longleftrightarrow
	\begin{Bmatrix}\text{deep points in } \\ V(\CA(\SS \smallsetminus D),\Bbbk) \\ \text{gluable along } D \end{Bmatrix}
	\]    
\end{prop}

\begin{proof}
	First, let $p\in V(\CA(\SS),\Bbbk)$ be a benevolent deep point that is non-zero on $D$. By Proposition \ref{prop: deep cutting inclusion}, its image under the cutting isomorphism $\bar{p} \in V(\CA(\SS\smallsetminus D),\Bbbk)$ is deep. Since $\SS \smallsetminus D$ is an unpunctured polygon, Theorem \ref{thm: deeppolygon} tells us that $m$ must be even. Further, combining with Proposition \ref{prop: polydiss arcs}, the alternating product of the boundary arcs of $\SS\smallsetminus D$ is \[(-1)^{\frac{4g+2b+m+2q}{2}}_{} = (-1)^{b+q+\frac{m}{2}}_{}. \]
	
	Now suppose that $\bar{p} \in V(\CA(\SS\smallsetminus D),\Bbbk)$ is a deep point that is gluable along $D$. Since $D$ is a polygonal dissection of $\SS$, there is a subset of $q$ arcs $R \coloneqq \{r_1, \dotsc, r_q\} \subset D$ such that each $r_i \in R$ has at least one end at a puncture, $\SS\smallsetminus R$ is an unpunctured surface, and $D^\prime \coloneqq D\smallsetminus R$ is a polygonal dissection of $\SS\smallsetminus R$.     
	By Corollary \ref{coro: deepcutpoly}, we have a bijection between 
	\begin{itemize}
		\item deep points of $V(\CA(\SS\smallsetminus R), \Bbbk)$ which are non-zero on $D^\prime$, and
		\item deep points of $V(\CA(\SS\smallsetminus D), \Bbbk)$ which are gluable along $D^\prime$.
	\end{itemize}
	Thus, we only need to consider the gluing along the arcs of $R$. 
	
	By our assumption of gluability, there are two boundary edges of $\SS\smallsetminus D$ that glue to give a radius $r_q$ from the boundary of $\SS$ to a puncture $\xi_q$. 
	By Lemma \ref{lemma: trianglesurf}, $\bar{p}$ kills an odd number of arcs in every triangle of $\SS\smallsetminus D$. Considering a fan triangulation of $\SS\smallsetminus D$ with all diagonals emanating from the marked point (or a choice of one of the marked points) corresponding to $\xi_q$, Theorem \ref{thm: deeppolygon} says that $\bar{p}$ will kill alternating diagonals; as a result, $p$ will kill alternating radii emanating from $\xi_q$ in $\SS$. Every triangulation of $\SS$ will contain at least two arcs emanating from $\xi_q$, but does not necessarily kill at least one such arc in every triangulation; suppose that $T$ is a triangulation of $\SS$ containing two arcs emanating from $\xi_q$ that are spared by $p$, and denote these two spared arcs as $r_1$ and $r_2$. If $r_1$ and $r_2$ are arcs in $\SS\smallsetminus D$, then they are spared by $\bar{p}$ as well. Let $\gamma$ be the arc in $\SS \smallsetminus D$ that completes a triangle with $r_1$ and $r_2$, and observe that $\gamma$ is also in $T$. Since $\bar{p}$ kills an odd number of arcs in every triangle, the arc $\gamma$ is necessarily killed by $\bar{p}$. Thus, $p$ must also kill $\gamma$, so $p$ kills at least one arc in $T$. Thus, $p$ must kill at least one arc in every triangulation of $\SS$, so $p$ is deep. 
	
	Now, suppose that $r_1, r_2$ in $T$ are not arcs in $\SS \smallsetminus D$, and let $\gamma$ denote an arc in $T$ that completes a triangle with $r_1$ and $r_2$. Since $\SS\smallsetminus R$ is an unpunctured surface, its universal cover $\widehat{\SS}$ is covered by lifts of the polygon $\SS\smallsetminus D$. Choose lifts $\widehat{r_1}$, $\widehat{r_2}$, and $\widehat{\gamma}$ that bound a triangle in $\widehat{\SS}$. Since this triangle is compact, it is contained in the union of finitely many lifts of $\SS\smallsetminus D$; let $\SS^\prime$ denote this union. As $\SS^\prime$ is a union of polygons in a simply connected space, it is again a polygon. Since $\bar{p}$ is deep, it gives a deep point on each lift of $\SS\smallsetminus D$. As in the proof of \ref{lemma: deepcut}, these deep points glue to a deep point $\widehat{p} \in V(\CA(\SS^\prime), \Bbbk)$. By Lemma \ref{lemma: trianglesurf}, $\widehat{p}$ must kill an odd number of arcs in the triangle $\{ \widehat{r_1}, \widehat{r_2}, \widehat{\gamma}\}$, so $\widehat{p}(\widehat{\gamma}) = 0$.       
	This descends to the deep point $\bar{p}$, which glues along $D$ to give the point $p$. As a result, $p$ must kill $\gamma$ as well, so $p$ kills at least one arc of $T$. Hence, $p$ must kill at least one arc in every triangulation of $\SS$, so $p$ is deep.
	
	Finally, since $\bar{p}$ is gluable along $D$, it is gluable along $R$, so there is at least one arc emanating from each puncture that is spared by $p$. Therefore, $p$ is a benevolent deep point.
\end{proof}

%============================

\subsection{Benevolent points}

By the bijection in Proposition \ref{prop: benevolent bijection surface}, we know that the benevolent deep points non-zero on a polygonal dissection $D$ are in correspondence with the deep points of a polygon with $4g + 2b + 2q + m -2$ sides, with each of the arcs in $D$ corresponding to two sides of the polygon. Since both edges corresponding to an arc in $D$ take the same value, we only have $2g+b+q+m-1$ choices of values for the edges of the polygon. Following from Theorem \ref{thm: deeppolygon}, we have the following:
\begin{prop}\label{prop: benevolent polydiss}
	Let $\SS$ be a connected, triangulable marked surface with genus $g\geq 0$, $b\geq 1$ boundary components, $m\geq 2$ marked points on the boundary, and $q\geq 1$ punctures. Let $D$ be a polygonal dissection of $\SS$. 
	\begin{enumerate}
		\item If $m$ is odd, then $V(\CA(\SS),\Bbbk)$ has no benevolent deep points. 
		\item If $m$ is even, then $p \in V(\CA(\SS),\Bbbk)$ is deep if and only if the alternating product around the boundary edges of $\SS\smallsetminus D$ equals
		\[ (-1)^{\frac{4g+2b+2q+m}{2}}_{} = (-1)^{b+q+\frac{m}{2}}_{}. \]
		As a consequence, these benevolent deep points are determined by the values on all but one boundary edge of $\SS\smallsetminus D$, so they determine $2^q$ deep subsets isomorphic to $(\Bbbk^\times)^{2g+b+q+m-2}$, one for each choice of tagging at each puncture.
	\end{enumerate}
\end{prop}

It remains to count the number of distinct polygonal dissections of $\SS$. By Proposition \ref{prop: polydisscount}, each unpunctured surface has $2^{2g+b-1}$ polygonal dissections, up to congruence. Thus, for each collection $R$ of arcs such that $\SS\smallsetminus R$ is an unpunctured surface with the same genus and number of boundary components as $\SS$, there are $2^{2g+b-1}$ distinct polygonal dissections of $\SS\smallsetminus R$. The arcs of each polygonal dissection $D^\prime$ are disjoint from $R$ except at the boundary, so $\SS\smallsetminus D^\prime$ is a $q$-punctured polygon \textemdash a disk with $q$ punctures and $4g + 2b + m -2$ marked points on the boundary. 

From Proposition \ref{prop: benevolent bijection multiple punctures} and the discussion on punctured disks in Section \ref{section: deep punctured disks}, we know that a benevolent deep point kills alternating radii around each puncture in $\SS\smallsetminus D^\prime$, and choosing a non-zero value for one arc emanating from each puncture determines the values of the other arcs emanating from that puncture. Further clarifying this point, Corollary \ref{cor: radii with same endpoints} tells us that a benevolent deep point either kills all radii with the same endpoints in the punctured polygon $\SS \smallsetminus D^\prime$ or spares all of them. That is, at each puncture, either the even-indexed radii are killed or the odd-indexed radii are killed, and a choice of non-zero value on one radius determines all remaining values around a given puncture. As such, there are $2^q$ ways to choose which radii receive non-zero values; hence, there are $2^q$ ways to choose the collection $R$, up to equivalence. This gives us a total of $2^{2g+b+q-1}$ distinct polygonal dissections. Combining with Proposition \ref{prop: benevolent polydiss}, we have the following result: 

\begin{thm}\label{thm: deepsurface punctured benevolent}
	Let $\SS$ be a connected, triangulable marked surface with genus $g\geq 0$, $b\geq 1$ boundary components, $m\geq 2$ marked points on the boundary, and $q\geq 1$ punctures. 
	\begin{enumerate}
		\item If $m$ is odd, then $V(\CA(\SS),\Bbbk)$ has no benevolent deep points. 
		\item If $m$ is even, then the deep locus of $V(\CA(\SS),\Bbbk)$ contains $2^{2g+b+2q-1}$ disjoint sets of benevolent points, each isomorphic to $(\Bbbk^\times)^{2g+b+q+m-2}$.
	\end{enumerate}
\end{thm}

Combined with the characterization of deep points that quash punctures in Section \ref{section: surface quash}, we have the following description of the deep locus. In particular, the deep locus can be described as a union of cluster varieties and algebraic tori.
\begin{thm}\label{thm: deepsurface punctured all}
	Let $\SS_q$ be a connected, triangulable marked surface with genus $g\geq 0$, $b\geq 1$ boundary components, $m\geq 2$ marked points on the boundary, and $q\geq 1$ punctures, such that $\SS_q$ is not a punctured bigon. 
	\begin{enumerate}
		\item If $m$ is odd, then the deep locus of $V(\CA(\SS_q),\Bbbk)$ is the union of $q$ non-disjoint sets, each isomorphic $V(\CA(\SS_{q-1}),\Bbbk)$. 
		\item If $m$ is even, then the deep locus of $V(\CA(\SS_q),\Bbbk)$ is the union of: 
        \begin{itemize}
            \item $q$ non-disjoint sets, each isomorphic to $V(\CA(\SS_{q-1}),\Bbbk)$, and
            \item $2^{2g+b+2q-1}$ disjoint sets, each isomorphic to $(\Bbbk^\times)^{2g+b+q+m-2}$.
        \end{itemize}
	\end{enumerate}
\end{thm}

\subsubsection{Benevolent points and radii}

Unfortunately, Lemma \ref{lemma: radii with same endpoints} and Corollary \ref{cor: radii with same endpoints} do not extend to more general punctured surfaces. That is, there may exist radii $r_1$ and $r_2$ with the same endpoints such that a benevolent deep point kills $r_1$ and spares $r_2$, or vice versa.

\begin{figure}[htb]
	\centering
	\begin{tikzpicture}[scale=3]
		\draw [fill=black!10, thick, even odd rule] (0,0) circle (1) (0,0) circle (0.3);
		\draw[thick] (-0.25,0) circle (0.75);
		\node[dot] (m1) at (-0.3,0) {};
		\node[dot] (m2) at (-1,0) {};
		\node[dot] (p) at (0.5,0) {};
		\node at (0,0.8) {$r_0$};
		\node at (0,-0.82) {$r_1$};
		\node at (0:0.22) {$b_1$};
		\node at (0:1.1) {$b_2$};
	\end{tikzpicture}
	\hspace{1cm}
	\begin{tikzpicture}[scale=3]
		\draw [fill=black!10, thick, even odd rule] (0,0) circle (1) (0,0) circle (0.3);
		\node[dot] (m1) at (180:0.3) {};
		\node[dot] (m2a) at (240:1) {};
		\node[dot] (p) at (180:1) {};
		\node[dot] (m2b) at (120:1) {};
		\node at (210:1.1) {$r_0$};
		\node at (150:1.1) {$r_0$};
		\draw[thick,out=0,in=270] (m2a) to (0:0.75);
		\draw[thick,out=90,in=0] (0:0.75) to (90:0.75);
		\draw[thick,out=180,in=45] (90:0.75) to (p);
		\node at (300:0.86) {$r_1$};
		\draw[thick] (p) to node [below] {$c_1$} (m1);
		\draw[thick] (m2a) to node [left] {$c_2$} (m1);
		\node at (0:0.22) {$b_1$};
		\node at (0:1.1) {$b_2$};
		\draw[thick,out=30,in=270] (m2a) to node [above] {$c_3$} (0:0.5);
		\draw[thick,out=90,in=0] (0:0.5) to (90:0.5);
		\draw[thick,out=180,in=120] (90:0.5) to (m1);
	\end{tikzpicture}
	\caption{A once-punctured $(1,1)$-annulus (left) and the result of cutting along $r_0$ (right)}
	\label{fig: radii different values}
\end{figure}
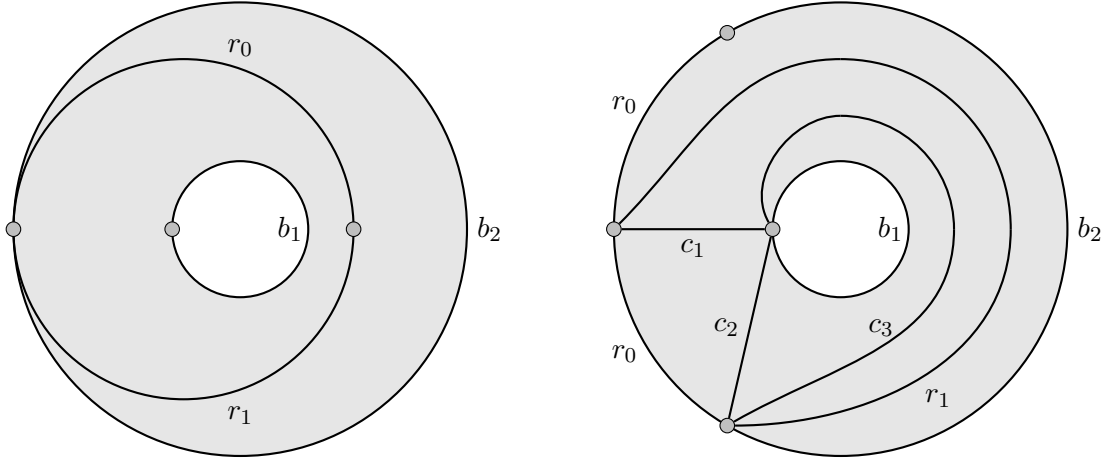

\begin{ex}
	Let $\SS$ be a $(1,1)$-annulus with a single puncture \textemdash that is, the once-punctured genus 0 surface with one marked point on each of its two boundary components. Let $p\in V(\CA(\SS),\Bbbk)$ be a benevolent deep point; by Lemma \ref{lemma: spared radius surface}, at least one radius $r_0$ must be spared by $p$. Cutting along $r_0$ gives us an unpunctured $(1,3)$-annulus $\SS^\prime$, and the deep points on $\SS^\prime$ with the same value on $r_0$ descend to deep points on $\SS$ which are non-zero on $r_0$ by Proposition \ref{prop: benevolent bijection surface}, so we know that such a $p$ exists.        
	
	Suppose that $r_0$ is as pictured in Figure \ref{fig: radii different values}, and $r_1$ is a radius with the same endpoints as in the same figure. After cutting along $r_0$, we have the image on the right. From Lemma \ref{lemma: trianglesurf}, we know that $p$ kills an odd number of arcs in each triangle of the unpunctured surface. Since $r_0 \neq 0$, we know that exactly one of $c_1$ and $c_2$ is killed by $p$. If $p(c_2) = 0$ and $p(c_1) \neq 0$, then $p(c_3)\neq 0$, so $p(r_1) = 0$. If instead $p(c_1) = 0$ and $p(c_2) \neq 0$, then $p(c_3) = 0$, so again $p(r_1) = 0$. In particular, $r_1$ is killed by $p$. 
	As a result, $r_0$ and $r_1$ are an example of radii with the same endpoints where $p$ kills one and spares the other.
\end{ex}

%============================

\bibliographystyle{alpha}
\bibliography{main}
	
\end{document}